\documentclass[12pt, reqno]{amsart}
\usepackage[T1]{fontenc}
\usepackage[utf8]{inputenc}
\usepackage[english]{babel}
\usepackage{amsmath,amsfonts,amssymb,amsthm,mathrsfs}
\usepackage{accents,mathtools,extarrows}
\usepackage{bbm}
\usepackage{yhmath}
\usepackage{graphicx}
\usepackage{rotating}
\usepackage{wasysym}
\usepackage{textcomp}
\usepackage{hyperref,cleveref}
\usepackage{xfrac}
\usepackage{enumitem}
\usepackage{tikz-cd}									
\usetikzlibrary{matrix,arrows,calc,positioning,decorations.pathmorphing}
\tikzset{
  shift left/.style ={commutative diagrams/shift left={#1}},
  shift right/.style={commutative diagrams/shift right={#1}}
}
\usepackage{todonotes}

\theoremstyle{plain}
\newtheorem{theorem}{Theorem}[section]
\newtheorem{lemma}[theorem]{Lemma}
\newtheorem{proposition}[theorem]{Proposition}
\newtheorem{corollary}[theorem]{Corollary}
\newtheorem*{theorem*}{Theorem}
\newtheorem*{proposition*}{Proposition}
\newtheorem*{corollary*}{Corollary}

\theoremstyle{definition}
\newtheorem{definition}[theorem]{Definition}
\AtBeginEnvironment{definition}{\pushQED{\qed}}
\AtEndEnvironment{definition}{\popQED\enddefinition}
\newtheorem{example}[theorem]{Example}
\AtBeginEnvironment{example}{\pushQED{\qed}}
\AtEndEnvironment{example}{\popQED\endexample}
\newtheorem{remark}[theorem]{Remark}
\AtBeginEnvironment{remark}{\pushQED{\qed}}
\AtEndEnvironment{remark}{\popQED\endremark}
\newtheorem{construction}[theorem]{Construction}
\AtBeginEnvironment{construction}{\pushQED{\qed}}
\AtEndEnvironment{construction}{\popQED\endconstruction}
\newtheorem{observation}[theorem]{Observation}
\AtBeginEnvironment{observation}{\pushQED{\qed}}
\AtEndEnvironment{observation}{\popQED\endobservation}
\newtheorem{observation/construction}[theorem]{Obervation/Construction}
\AtBeginEnvironment{observation/construction}{\pushQED{\qed}}
\AtEndEnvironment{observation/construction}{\popQED\endobservation}
\newtheorem{notation}[theorem]{Notation}
\AtBeginEnvironment{notation}{\pushQED{\qed}}
\AtEndEnvironment{notation}{\popQED\endnotation}
\newtheorem{assumption}[theorem]{Assumption}
\AtBeginEnvironment{assumption}{\pushQED{\qed}}
\AtEndEnvironment{assumption}{\popQED\endassumption}
\newtheorem*{remark*}{Remark}
\AtBeginEnvironment{remark*}{\pushQED{\qed}}
\AtEndEnvironment{remark*}{\popQED\endremark}

\newcommand{\N}{\mathbb N}
\newcommand{\Z}{\mathbb Z}
\newcommand{\Q}{\mathbb Q}

\newcommand{\F}{\mathbb F}

\renewcommand{\epsilon}{\varepsilon}
\renewcommand{\phi}{\varphi}

\begin{document}

\title{Unstable algebraic K-theory: homological stability and other observations}
\author{Mikala Ørsnes Jansen}

\begin{abstract}
We investigate stability properties of the reductive Borel--Serre categories; these were introduced as a model for unstable algebraic K-theory in previous work. We see that they seem exhibit better homological stability properties than the general linear groups. We also show that they provide an explicit model for Yuan's partial algebraic K-theory.
\end{abstract}

\maketitle

\tableofcontents

\section{Introduction}

\textbf{Unstable algebraic K-theory.} Let $A$ be an associative ring. A finitely generated projective $A$-module $M$ defines a point $[M]$ in the K-theory space of $A$; this association is functorial, so there is a map
\begin{align*}
\operatorname{BGL}(M)\rightarrow K(A)
\end{align*}
from the automorphism group of $M$ into the K-theory space. This is very far from being an equivalence essentially for two different reasons: firstly, they are just very different types of spaces, for example $K(A)$ is a simple space whereas $\operatorname{BGL}(M)$ is a $K(\pi,1)$ for generally non-abelian $\pi$; secondly, $\operatorname{BGL}(M)$ is defined in terms of linear algebra internal to the fixed module $M$ whereas $K(A)$ takes into account all finitely generated projective $A$-modules. We want to alleviate the first reason while keeping the second. In other words, we want to approximate $K(A)$ using only linear algebra internal to $M$.

To this end, we consider the notion of an \textit{unstable algebraic K-theory} by which we mean a family of ``intermediary'' spaces,
\begin{align*}
\operatorname{BGL}(M)\rightarrow \overline{\operatorname{BGL}(M)}\rightarrow K(A),
\end{align*}
one for each finitely generated projective $A$-module $M$, such that $\overline{\operatorname{BGL}(M)}$ is defined purely in terms of linear algebra internal to $M$, and through which the map above factorises. The notion and terminology goes back to van der Kallen (\cite{vanderKallen}), and a classical example is Quillen's $+$-construction $\operatorname{BGL}(M)^+$. The quality of a model for unstable algebraic K-theory depends on how close it is to $K(A)$ in terms of its nature and properties. A notable difference between the trivial model $\operatorname{BGL}(M)$ and Quillen's $+$-construction is the following: the automorphism groups $\operatorname{BGL}(M)$ stabilise to $K(A)$ via group completion
\begin{align*}
K(A)\simeq \bigg| \coprod_{[M]} \operatorname{BGL}(M)\bigg|^{\operatorname{grp}}
\end{align*}
whereas the plus-constructions stabilise more naively, as it were, via a directed colimit
\begin{align*}
K(A)\simeq K_0(A)\times \operatorname{BGL}_\infty(A)^+.
\end{align*}

\textbf{Reductive Borel--Serre categories.} In \cite{ClausenOrsnesJansen}, we proposed a new model for unstable algebraic K-theory inspired by a detailed analysis of the reductive Borel--Serre compactification from the point of view of stratified homotopy theory. To any finitely generated projective $A$-module $M$, we associate a \textit{reductive Borel--Serre category} (or simply \textit{$\operatorname{RBS}$-category}), $\operatorname{RBS}(M)$. This is a $1$-category encoding the system of splittable flags in $M$ and its realisation serves as a model for unstable algebraic K-theory
\begin{align*}
\operatorname{BGL}(M)\rightarrow |\operatorname{RBS}(M)| \rightarrow K(A).
\end{align*}

We started an investigation of the properties of $|\operatorname{RBS}(M)|$ and its quality as an unstable algebraic K-theory. We showed that for infinite fields, we simply recover the classical plus construction (\cite[Theorem 5.15]{ClausenOrsnesJansen}). For finite fields, however, the reductive Borel--Serre categories turn out to be very different and, in fact, exhibit much better homological stability properties than the plus constructions. Indeed, if $k$ is a finite field of characteristic $p$, then $\operatorname{RBS}(k^n)$ has vanishing $\F_p$-homology in positive degrees whereas the $\F_p$-homology of $\operatorname{GL}_n(k)$ is non-trivial, rather complicated and still largely unknown, but vanishes upon stabilisation (\cite[Theorem 5.18]{ClausenOrsnesJansen}, \cite{MilgramPriddy}, \cite{LahtinenSprehn}, \cite{Quillen72}).

We also proved that the reductive Borel--Serre categories stabilise to the K-theory space in the following sense: the $\operatorname{RBS}(M)$'s assemble to a monoidal category whose geometric realisation group completes to the K-theory space. More precisely, we introduce a \textit{monoidal reductive Borel--Serre category} (or simply \textit{monoidal $\operatorname{RBS}$-category}) $\mathscr{M}_{\operatorname{RBS}}(A)$ which decomposes as a disjoint union
\begin{align*}
\mathscr{M}_{\operatorname{RBS}}(A)\simeq \coprod_{[M]} \operatorname{RBS}(M),
\end{align*}
and then we identify
\begin{flalign*}
& & K(A)\simeq \Omega B|\mathscr{M}_{\operatorname{RBS}}(A)| \qquad\qquad  \text{(\cite[Theorem 7.38]{ClausenOrsnesJansen}).}
\end{flalign*}


\medskip

\textbf{Questions of stability.} The purpose of this paper is to continue this investigation. More specifically, we want to investigate the stability properties of the $\operatorname{RBS}$-categories. For any ring $A$ and any finitely generated projective $A$-modules $M$ and $N$, we have a stabilisation map
\begin{align*}
\operatorname{RBS}(M)\rightarrow \operatorname{RBS}(M\oplus N).
\end{align*}
There are two obvious questions to be asked. First of all, we'd like to compare the formal stabilisation procedure given by group completion, $\Omega \operatorname{B}(-)$, with the more naive stabilisation procedure of taking the colimit of the sequence
\begin{align*}
\operatorname{RBS}(M)\rightarrow \operatorname{RBS}(M\oplus A)\rightarrow \cdots \rightarrow  \operatorname{RBS}(M\oplus A^n) \rightarrow \operatorname{RBS}(M\oplus A^{n+1})\rightarrow \cdots
\end{align*}
Secondly, we want to address the question of homological stability: Do the stabilisations maps induce isomorphisms on the $d$'th homology groups
\begin{align*}
H_d(\operatorname{RBS}(M))\rightarrow H_d(\operatorname{RBS}(M\oplus N))
\end{align*}
for $\operatorname{rk} M$ sufficiently large compared to $d$?

The issue that has been impeding us from making any progress in this direction is the following: the category $\mathscr{M}_{\operatorname{RBS}}(A)$ is \textit{only} monoidal, it is not symmetric monoidal, not even braided. Hence, its realisation is a priori only an $\mathbb{E}_1$-space. The standard tools for attacking the questions above, for example the group completion theorem, require some kind of homotopy commutativity (even if a bare minimum), so we have simply had a very limited range of machinery to work with.

In some ways, the main result of this paper is that the realisation $|\mathscr{M}_{\operatorname{RBS}}(A)|$ is in fact an $\mathbb{E}_\infty$-space (though of course, it's the consequences of this that are really interesting). This additional structure completely opens up the investigation and we can answer a whole range of questions regarding the stabilisation maps above by simply plugging into established tools.

\medskip

\textbf{Partial algebraic K-theory.} We briefly recall the main ideas of \textit{partial algebraic K-theory} as this plays an important role in this work. It's an invariant introduced by Yuan (\cite{Yuan}) that shares some notable properties with the monoidal $\operatorname{RBS}$-categories.

Recall that the K-theory space, $K(A)$, can be obtained by freely turning Waldhausen's $S_\bullet$-construction into a \textit{grouplike} $\mathbb{E}_1$-space. Partial algebraic K-theory, $K^\partial(A)$, is defined by universally turning the $S_\bullet$-construction into a $\mathbb{E}_1$-space \textit{without} the grouplike condition. In other words, partial K-theory should be interpreted as a ``pre-group completed'' variant of algebraic K-theory. This makes the analogue of our stabilisation result a tautology:
\begin{align*}
K(A)\simeq K^\partial(A)^{\operatorname{grp}}
\end{align*}
(our proof, on other hand, is a lengthy and tedious exercise in simplicial manipulations). Thus, both $K^\partial(A)$ and $|\mathscr{M}_{\operatorname{RBS}}(A)|$ are by definition $\mathbb{E}_1$-spaces encoding flags of finitely generated projective $A$-modules and their associated gradeds, and they both group complete to the K-theory space. Moreover, for a finite field $k$ of characteristic $p$, Yuan proves that $K^\partial(k)$ also has vanishing $\F_p$-homology (\cite[Theorem 5.2]{Yuan}); this result is a crucial ingredient in his work on setting up a new model for unstable homotopy theory. In fact, Yuan's proof has the same rough outline as our calculation: after some combinatorial shuffling, one reduces to the fact that the $\F_p$-homology of the Steinberg representation of $\operatorname{GL}_n(k)$ vanishes.

We are naturally led to ask how these $\mathbb{E}_1$-spaces compare, Yuan's partial K-theory $K^\partial(A)$ on the one hand and the geometric realisation of the monoidal $\operatorname{RBS}$-category $\mathscr{M}_{\operatorname{RBS}}(A)$ on the other. We'll show that they are in fact equivalent as $\mathbb{E}_1$-spaces. Yuan shows by an Eckmann-Hilton argument that $K^\partial(A)$ naturally admits the structure of an $\mathbb{E}_\infty$-space, so this identification in fact supplies us with the $\mathbb{E}_\infty$-structure mentioned earlier. To better understand the underlying dynamics, we exhibit the $\mathbb{E}_\infty$-structure on $|\mathscr{M}_{\operatorname{RBS}}(A)|$ independently of this identification. The ideas of Yuan's construction and analysis are, however, central to the whole paper and was the starting point of this investigation, so we feel that it deserves a mention here.

\medskip

\textbf{Main results.} We already mentioned some of the results of the paper, but let us present them here again together with the consequent stabilisation results.

\begin{theorem*}[\ref{E-infinity structure}]
For any ring $A$, the $\mathbb{E}_1$-space $|\mathscr{M}_{\operatorname{RBS}}(A)|$ naturally admits the structure of an $\mathbb{E}_\infty$-space refining the $\mathbb{E}_1$-structure.
\end{theorem*}

This is also a corollary of the following identification, but we provide an independent proof (inspired by Yuan's proof that partial algebraic K-theory admits an $\mathbb{E}_\infty$-structure).

\begin{theorem*}[\ref{equivalence of E1-spaces}]
For any ring $A$, we have an equivalence of $\mathbb{E}_1$-spaces
\begin{align*}
K^\partial(A) \xrightarrow{\ \simeq\ } |\mathscr{M}_{\operatorname{RBS}}(A)|
\end{align*}
identifying partial algebraic K-theory of $A$ with the $\mathbb{E}_1$-space arising as the geometric realisation of the monoidal $\operatorname{RBS}$-category.
\end{theorem*}

The above two results are proved for general exact categories, but for notational ease, we restrict our attention to rings in this introduction.

As an immediate corollary of the established homotopy commutativity, we can apply the group completion theorem to compare the group completion $\Omega B|\mathscr{M}_{\operatorname{RBS}}(A)|$ with the colimit
\begin{align*}
\operatorname{RBS}_\infty(A):=\displaystyle\mathop{\operatorname{colim}}_{n\rightarrow \infty}\operatorname{RBS}(A^n).
\end{align*}

\begin{corollary*}[\ref{applying the group completion theorem}]
For any associative ring $A$, we have
\begin{align*}
K(A)\simeq \Omega B|\mathscr{M}_{\operatorname{RBS}}(A)|\simeq K_0(A)\times |\operatorname{RBS}_\infty(A)|^+
\end{align*}
If $A$ is such that all finitely generated projective modules are split Noetherian, then $|\operatorname{RBS}_\infty(A)|$ is its own plus-construction and
\begin{align*}
|\operatorname{RBS}_\infty(A)|\simeq \operatorname{BGL}_\infty(A)^+.
\end{align*}
\end{corollary*}

Recall that a finitely generated projective module $M$ is \textit{split Noetherian} if every increasing chain of splittable submodules of $M$ stabilises (\cite[Definition 1.1]{ClausenOrsnesJansen}). This is a technical condition that ensures that we have reasonable control over the poset of splittable flags and the action of the automorphism group on this poset (see \cite[Lemma 5.4]{ClausenOrsnesJansen}). Moreover note that if either:
\begin{enumerate}
\item $A$ is Noetherian, or
\item $A$ is commutative and $\operatorname{Spec}(A)$ has only finitely many connected components
\end{enumerate}
then every finitely generated projective $A$-module $M$ is split Noetherian (\cite[Lemma 5.5]{ClausenOrsnesJansen}). It seems plausible that one can get rid of this hypothesis in the Corollary above (see \Cref{get rid of split Noetherian hypothesis}).

The existence of an $\mathbb{E}_\infty$-structure also enables us to exploit some generic homological stability results in the setting of cellular $\mathbb{E}_k$-algebras developed by Galatius--Kupers--Randal-Williams (\cite{GalatiusKupersRandalWilliams}, \cite{KupersMillerPatzt}, \cite{JansenMiller}). Let $R$ be a commutative ring, $\Bbbk$ a commutative ring and consider the stabilisations maps on homology
\begin{align*}
H_d(\operatorname{RBS}(R^{n-1});\Bbbk)\rightarrow H_d(\operatorname{RBS}(R^n);\Bbbk),\quad d\geq 0,\ n\geq 2.
\end{align*}
For clarity, we'll just list the results:
\begin{enumerate}
\item For $R$ a PID, the stabilisation map above is an isomorphism for $2d<n-1$ with integer coefficients (\Cref{PID homological stability}).
\item For $R$ a local commutative ring, the stabilisation map above is an isomorphism for $3d<2n-1$ with integer coefficients  (\Cref{Euclidean domain 2/3}).
\item For $R$ a Euclidean domain, the stabilisation map above is an isomorphism for $d<n-1$ with integer coefficients (Theorem \ref{Euclidean domain homological stability}).
\item For $R=\mathcal{O}$ the ring of integers in an algebraic number field $K$ such that $\mathcal{O}$ has class number $1$ and $K$ is not one of the three imaginary quadratic extensions $\Q(\sqrt{-d})$, $d=43, 67, 163$, the stabilisation map above is an isomorphism for $3d<2n-1$ with integer coefficients  (\Cref{rings of integers 2/3}).
\item For $R=\mathcal{O}$ the ring of integers in $\Q(\sqrt{-d})$, for $d=43, 67, 163$, the stabilisation map above is an isomorphism for $3d<2n-1$ with $\Z[\sfrac{1}{6}]$-coefficients  (\Cref{rings of integers 2/3 mod 6} --- we do not know at this point whether this range extends to integer coefficients).
\item For $R$ a commutative ring with connected spectrum over which finitely generated projective modules are free and satisfying that the complex of partial bases $B_n(R)$ is Cohen–Macaulay for all $n\geq 0$ (e.g. $R$ local commutative), the stabilisation map above is an isomorphism for $d<n-1$ with $\Z[\sfrac{1}{2}]$-coefficients (\Cref{PID and partial bases}).
\end{enumerate}

We remark here that this is the result of a somewhat limited investigation (to moderate the extent and scope of the present paper) and that there is a lot left to be explored. However, these observations already seem to support the hope that the $\operatorname{RBS}$-categories satisfy better homological stability properties than the general linear groups: for Euclidean domains, the general linear groups have recently been shown to satisfy the same stability range with $\Z[\sfrac{1}{2}]$-coefficients (\cite{BernardMillerSroka}), but for e.g. $\Z$, this does not extend to $\F_2$-coefficients (see \Cref{small degrees}). For $\mathcal{O}$ the ring of integers in $\Q(\sqrt{-d})$ for $d=19, 43, 67, 163$, this also seems to be an improvement on the established range of homological stability for the general linear groups (to the best of the author's knowledge). Admittedly, the improvement seems to be found only in the $\F_2$-coefficients. See also \Cref{compare with GL}. We naively hope that the stability properties for the $\operatorname{RBS}$-categories observed here can be improved.

Randal-Williams has recently used the main results of this paper to show homological stability for the $\operatorname{RBS}$-categories over general Dedekind domains whereas we only address rings over which finitely generated projective modules are free (\cite{RandalWilliams24}; see also \Cref{general Dedekind}).

In order to investigate homological stability, we calculate the $\mathbb{E}_1$-homology of $|\mathscr{M}_{\operatorname{RBS}}(R)|$ (or rather of its non-unital analogue $|\mathscr{M}_{\operatorname{RBS}}^{\operatorname{nu}}(K)|$). For simplicity we'll just state this in the case of fields (see \Cref{E1-homology for general ring} for the general statement).

\begin{theorem*}[\ref{E1-homology for general ring}] For $K$ a field and $\Bbbk$ a commutative ring, we have
\begin{align*}
H^{\mathbb{E}_1}_{n,d}\big(|\mathscr{M}_{\operatorname{RBS}}^{\operatorname{nu}}(K)|;\Bbbk\big)\cong H_{d-n+1}\big(GL_n(K); \operatorname{St}_n(K;\Bbbk)\big)
\end{align*}
where $\operatorname{St}_n(K;\Bbbk)$ is the usual Steinberg module. In particular, $H^{\mathbb{E}_1}_{n,d}\big(|\mathscr{M}_{\operatorname{RBS}}^{\operatorname{nu}}(K)|;\Bbbk\big)=0$ 
for all $d< n-1$.
\end{theorem*}

This result should be compared with the analogous calculation for $\coprod \operatorname{BGL}_m(K)$ made in \cite[\S 17.2]{GalatiusKupersRandalWilliams}:
\begin{align*}
H^{\mathbb{E}_1}_{n,d}\Big(\coprod \operatorname{BGL}_m(K);\Bbbk\Big)\cong H_{d-n+1}\big(\operatorname{GL}_n(K); \operatorname{St}^{\operatorname{split}}_n(K;\Bbbk)\big)
\end{align*}
where $\operatorname{St}^{\operatorname{split}}(M;\Bbbk)$ is the so-called $\mathbb{E}_1$-Steinberg module which in this case identifies with the homology of Charney's \textit{split} Tits building (\cite{Charney80}, \cite[\S 18.2]{GalatiusKupersRandalWilliams}). An advantage to working with the $\operatorname{RBS}$-categories over the general linear groups is that Tits buildings and Steinberg modules are far better studied than their split analogues.

The calculation of the $\mathbb{E}_1$-homology follows easily from an analysis of a naturally arising rank filtration of $\mathscr{M}_{\operatorname{RBS}}(R)$, an analogue of Quillen's rank filtration (\cite{Quillen73a}). In some ways, this filtration is more fundamental in nature than Quillen's rank filtration for the following two reasons:
\begin{enumerate}
\item it recovers Quillen's rank filtration by group completion;
\item it is more highly structured as it is a filtration of monoidal categories.
\end{enumerate}
In \cite{Rognes}, Rognes introduces spectrum-level rank filtrations of algebraic K-theory generalising Quillen's rank filtration to rank filtrations of deloopings of the K-theory space. We can interpret the rank filtration of the monoidal $\operatorname{RBS}$-category as a generalisation in the other direction giving a rank filtration that \textit{deloops} to Quillen's (see \Cref{rank filtration for commutative ring}).

We explicitly identify the associated graded of this filtration: it is free on the homotopy quotient of $\operatorname{RBS}(M)$ modulo its boundary $\partial \operatorname{RBS}(M):=\operatorname{RBS}(M)\smallsetminus \operatorname{BGL}(M)$ (\Cref{E1-pushout diagram}). This quotient in turn identifies with an analogue of the usual Tits building $\mathcal{T}_n(K)$; this is why the Steinberg module appears in the calculation above.

\medskip

\textbf{Overview.} The first half of this paper is essentially just an exposition of the work of Yuan on partial algebraic K-theory (\cite{Yuan}): we study the functor
\begin{align*}
\mathbb{L}\colon \operatorname{Fun}(\Delta^{\operatorname{op}},\mathcal{C})\rightarrow \operatorname{Mon}(\mathcal{C})
\end{align*}
(when it exist!) that universally turns a simplicial object into a monoid object and we identify an explicit formula for it. This is done in \cite{Yuan} for $\mathcal{C}=\mathcal{S}$ the $\infty$-category of spaces, but for the purposes of that paper it suffices to single out the underlying space. We'll want to work with the case $\mathcal{C}=\operatorname{Cat}_\infty$ too, but more importantly, we need to also understand the structure maps of the monoid object $\mathbb{L}X$. For this reason, we choose to redo the whole thing in slightly greater generality and keep track of all structure maps as we do so.\footnote{This also fixes some minor errors in \cite{Yuan} (see Remarks \ref{eta not initial} and \ref{Fred in F not cofinal}).} We stress, however, that the main ideas are all due to Yuan.

In \S 2, we review monoid objects and a Lawvere theory for such. In \S 3, we introduce a selection of ($1$-)categories of partitioned linearly ordered sets and make some useful observations. This section deviates slightly from the strategy of \cite{Yuan} as we'll need some additional room for maneuvre in our colimit diagrams. We unravel the functor $\mathbb{L}$ for $\mathcal{E}$ a cartesian closed cocomplete $\infty$-category in \S 4, providing an explicit tangible expression for it --- this follows the proof strategy of \cite[Proposition 5.10]{Yuan}. We discuss Yuan's further unravelling for split exact categories which is crucial to his calculation $H^*(K^\partial (\F_p);\F_p)=0$, $\ast>0$ (\cite[Theorem 5.2]{Yuan}). In \S 5, we briefly recall free monoid objects and non-unital monoid objects.

With the technical setup in place, we come to the novel work of the paper. In \S 6, we recall the monoidal $\operatorname{RBS}$-category $\mathscr{M}_{\operatorname{RBS}}(\mathcal{C})$ of an exact category $\mathcal{C}$ and the (non-monoidal) $\operatorname{RBS}$-categories $\operatorname{RBS}(c)$ for $c$ an object in $\mathcal{C}$. We also introduce non-unital versions and a ``fattened up'' analogue of these. In \S 7, we exhibit an equivalence $K^\partial(\mathcal{C})\simeq |\mathscr{M}_{\operatorname{RBS}}(\mathcal{C})|$ of $\mathbb{E}_1$-spaces, identifying partial algebraic K-theory with the geometric realisation of the monoidal $\operatorname{RBS}$-category. Inspired by Yuan's proof that $K^\partial(\mathcal{C})$ admits an $\mathbb{E}_\infty$-structure, we prove such a result independently for the monoidal $\operatorname{RBS}$-categories in \S 8, and as an immediate corollary, we apply the group completion theorem. In \S 9, we introduce the monoidal rank filtration of $\mathscr{M}_{\operatorname{RBS}}(\mathcal{C})$ and identify its associated graded. Using this, we calculate the $\mathbb{E}_1$-homology of $|\mathscr{M}_{\operatorname{RBS}}^{\operatorname{nu}}(\mathcal{C})|$ in \S 10 and go on to explore homological stability of the $\operatorname{RBS}$-categories.

We finish off the paper by reviewing some questions and directions for possible future investigations in \S 11.

\medskip

\textbf{Acknowledgements.} We thank the IHES for their hospitality while this work was carried out. We thank Dustin Clausen for many fruitful discussions and Oscar Randal-Williams for sharing his insights into homological stability. We are grateful to Jeremy Miller for useful pointers and to Damià Rodríguez i Banús and Qingyuan Bai for helpful comments. We would also like to thank Allen Yuan for his careful work on partial algebraic K-theory; as we hope has already been made apparent, this perspective really was the point of entry to this work. We acknowledge the support of the Danish National Research Foundation through the Copenhagen Centre for Geometry and Topology (DNRF151).

\medskip

\textbf{Conventions and notation.} We adopt the set-theoretic conventions of \cite[\S 1.2.15]{LurieHTT}.

\section{Preliminaries on monoids}

We recall the basics of monoid objects in $\infty$-categories and make some preliminary observations. In this paper, we will be interested in $\mathbb{E}_1$-spaces and monoidal $\infty$-categories; that is, monoid objects in the $\infty$-category of spaces, respectively small $\infty$-categories (denoted by $\mathcal{S}$, respectively $\operatorname{Cat}_\infty$). As we want to work with both cases, we make the relevant observations in greater generality to ease notation. The content of this section follows \cite[\S 4.1.2]{LurieHA} and \cite[\S\S 4-5]{Yuan}.

\subsection{Monoid objects and partial algebraic K-theory}\label{E1-monoids}

We introduce monoid objects and review some examples relevant for this paper: partial algebraic K-theory as introduced by Yuan (\cite{Yuan}) and monoidal $1$-categories. We follow \cite[\S 4.1.2]{LurieHA}.

\begin{definition}
Let $\mathcal{C}$ be an $\infty$-category. A \textit{monoid object} in $\mathcal{C}$ is a simplicial object $X\colon \Delta^{\operatorname{op}}\rightarrow \mathcal{C}$ such that for every $n\geq 0$, the morphisms induced by the face maps 
\begin{align*}
[1]\xrightarrow{\cong} \{i-1<i\}\hookrightarrow [n],\quad 1\leq i\leq n,
\end{align*}
exhibit $X([n])$ as a product
\begin{align*}
X([n])\xrightarrow{\ \simeq\ } X([1])^n.
\end{align*}
We call this map the ($n$'th) \textit{Segal map}. We denote by $\operatorname{Mon}(\mathcal{C})\subset \operatorname{Fun}(\Delta^{\operatorname{op}},\mathcal{C})$ the full subcategory spanned by the monoid objects. Given $X\in \operatorname{Mon}(\mathcal{C})$, we call $X([1])$ the \textit{underlying object} of $X$.
\end{definition}

By \cite[Proposition 4.1.2.10]{LurieHA}, monoid objects in $\mathcal{C}$ are essentially the same as $\mathbb{E}_1$-monoid objects in $\mathcal{C}$. More explicitly, there is an equivalence of $\infty$-categories
\begin{align*}
\operatorname{Mon}_{\mathbb{E}_1}(\mathcal{C})\xrightarrow{\ \simeq\ }\operatorname{Mon}(\mathcal{C}).
\end{align*}
where the left hand side is the $\infty$-category of $\mathbb{E}_1$-monoid objects in $\mathcal{C}$ (see \cite[Definition 2.4.2.1]{LurieHA} for details; note that this is over the $\infty$-operad $\operatorname{Assoc}$, but we invoke the equivalence $\mathbb{E}_1\simeq \operatorname{Assoc}$ of \cite[Example 5.1.0.7]{LurieHA}).

\begin{remark} \ 
\begin{enumerate}
\item By abuse of terminology, we may sometimes refer to the underlying object $X([1])$ as a monoid object.
\item For $\mathcal{C}=\mathcal{S}$, the $\infty$-category of spaces, a monoid object is called an \textit{$\mathbb{E}_1$-space}.
\item For $\mathcal{C}=\operatorname{Cat}_\infty$, the $\infty$-category of small $\infty$-categories, a monoid object is called a \textit{monoidal $\infty$-category}.\qedhere
\end{enumerate}
\end{remark}

We denote the inclusion of monoid objects into the $\infty$-category of simplicial objects by
\begin{align*}
\mathbb{B}_\mathcal{C}\colon \operatorname{Mon}(\mathcal{C})\hookrightarrow \operatorname{Fun}(\Delta^{\operatorname{op}},\mathcal{C}).
\end{align*}

Suppose $\mathcal{C}$ is a presentable $\infty$-category. Since $\operatorname{Mon}(\mathcal{C})$ is closed under limits and filtered colimits, it follows by the Adjoint Functor Theorem that $\mathbb{B}$ admits a left adjoint (\cite[Corollary 5.5.2.9]{LurieHTT}):
\begin{align*}
\mathbb{L}_\mathcal{C}\colon \operatorname{Fun}(\Delta^{\operatorname{op}},\mathcal{C})\rightarrow \operatorname{Mon}(\mathcal{C}).
\end{align*}

If $\mathcal{C}$ is cartesian closed, we can unravel this left adjoint and get a reasonably tangible expression for it. In fact, we will show in \S \ref{unravelling} that for any cocomplete cartesian closed $\infty$-category $\mathcal{C}$, the inclusion $\mathbb{B}_\mathcal{C}$ admits a left adjoint $\mathbb{L}_\mathcal{C}$ and we provide a concrete expression for it.

\begin{observation}
For the $\infty$-categories $\mathcal{S}\subset \operatorname{Cat}_\infty$, we have a diagram of adjunctions
\begin{center}
\begin{tikzpicture}
\matrix (m) [matrix of math nodes,row sep=3em,column sep=3em,nodes={anchor=center}]
{
\operatorname{Mon}(\mathcal{S}) & \operatorname{Mon}(\operatorname{Cat}_\infty) \\
\operatorname{Fun}(\Delta^{\operatorname{op}},\mathcal{S}) & \operatorname{Fun}(\Delta^{\operatorname{op}},\operatorname{Cat}_\infty) \\
};
\path[-stealth]
(m-1-2.170) edge[bend right] node[above]{$\scriptstyle|-|$} (m-1-1.15)
(m-2-2.171) edge[bend right] node[above]{$\scriptstyle|-|$} (m-2-1.12)
(m-2-1.70) edge[bend right] node[right]{$\mathbb{L}_\mathcal{S}$} (m-1-1.290)
(m-2-2.70) edge[bend right] node[right]{$\mathbb{L}_{\operatorname{Cat}_\infty}$} (m-1-2.290)
;
\path[right hook-stealth]
(m-1-1) edge node[right, rotate=90]{$\scriptstyle \vdash$} (m-1-2)
(m-2-1) edge node[right, rotate=90]{$\scriptstyle \vdash$} (m-2-2)
(m-1-1.250) edge node[right]{$\scriptstyle \vdash$} node[left]{$\mathbb{B}_\mathcal{S}$} (m-2-1.110)
(m-1-2.250) edge node[right]{$\scriptstyle \vdash$} node[left]{$\mathbb{B}_{\operatorname{Cat}_\infty}$} (m-2-2.110)
;
\end{tikzpicture}
\end{center}

By uniqueness of left adjoints, we must have
\begin{align*}
\mathbb{L}_\mathcal{S}(|X|)\simeq |\mathbb{L}_{\operatorname{Cat}_\infty}(X)|.
\end{align*}
In particular,
\begin{equation*}
\mathbb{L}_\mathcal{S}(-)\simeq |\mathbb{L}_{\operatorname{Cat}_\infty}(-)|.\qedhere
\end{equation*}
\end{observation}

\begin{example}
A monoidal (1-)category $M$ defines a monoidal $\infty$-category $M^\otimes$ via the bar construction:
\begin{align*}
[n]\mapsto M^n
\end{align*}
with simplicial structure maps given by the monoidal structure. This defines a fully faithful functor
\begin{align*}
\operatorname{Mon}(\operatorname{Cat})\hookrightarrow \operatorname{Mon}(\operatorname{Cat}_\infty)
\end{align*}
In fact, the monoidal $\infty$-categories that we'll be interested in here all arise in this fashion. We want to consider them as monoidal $\infty$-categories as this allows us to establish stronger results about how they relate to each other (cf. \Cref{E1-pushout diagram}).
\end{example}

\begin{example}
In \cite{Yuan}, Yuan uses the left adjoint $\mathbb{L}_\mathcal{S}$ to define partial algebraic K-theory. Let's recall the definition. Let $\mathcal{W}$ be a Waldhausen $\infty$-category (\cite{Barwick}).\footnote{In the end, we will only be considering the algebraic K-theory of exact $1$-categories, so for our purposes we can safely stick to the $1$-categorical case and the reader is welcome to just think of Waldhausen $1$-categories (\cite{Waldhausen}) or even exact categories.} The \textit{partial algebraic K-theory} of $\mathcal{W}$ is the following $\mathbb{E}_1$-space
\begin{align*}
K^\partial(\mathcal{W}):=\mathbb{L}_{\mathcal{S}}(S_\bullet(\mathcal{W})^{\simeq})
\end{align*}
where $S_\bullet(\mathcal{W})$ is Waldhausen's $S_\bullet$-construction: a simplicial $\infty$-category such that $S_n(\mathcal{W})$ is equivalent to the $\infty$-category of sequences of cofibrations
\begin{align*}
\ast\hookrightarrow X_1\hookrightarrow X_2\hookrightarrow \cdots \hookrightarrow X_n \qquad \text{in }\mathcal{W}
\end{align*}
and $S_\bullet(\mathcal{W})^\simeq$ is the simplicial space obtained by considering the levelwise maximal subgroupoids (see \cite{Barwick} and \cite{Waldhausen} for details).

Recall that the algebraic K-theory of $\mathcal{S}$ can be defined as the $\mathbb{E}_1$-space
\begin{align*}
K(\mathcal{W})=\Omega |S_\bullet(\mathcal{W})^\simeq|.
\end{align*}

In other words, we can interpret algebraic K-theory as the universal way of turning the simplicial space $S_\bullet(\mathcal{W})^\simeq$ into a \textit{grouplike} $\mathbb{E}_1$-monoid. Analogously, partial algebraic K-theory should be interpreted as the universal way of turning $S_\bullet(\mathcal{W})^\simeq$ into an $\mathbb{E}_1$-monoid \textit{without} requiring it to be grouplike.
\end{example}

\subsection{Lawvere theory}\label{Lawvere theory}

In this section we provide a different description of the functor $\mathbb{B}_{\mathcal{C}}$ in order to better analyse the left adjoint $\mathbb{L}_{\mathcal{C}}$; this follows \cite[\S 5.1]{Yuan}. We generalise the Lawvere theory for $\mathbb{E}_1$-spaces to general monoid objects (see \cite[Proposition 5.6]{Yuan}; the proof given there goes through verbatim as can be seen below).

\begin{definition}
For a finite linearly ordered set $I$, let $I^\pm$ denote $I$ with a minimal and maximal element adjoined, denoted by $\bot$, respectively $\top$. Let $\operatorname{Ord}_\pm$ be the $1$-category whose objects are finite linearly ordered sets and whose morphisms are given by order preserving maps $I^\pm\rightarrow J^\pm$ preserving $\bot$ and $\top$. We denote the elements of $\operatorname{Ord}_\pm$ by $I^\pm$ to distinguish them from the linearly ordered sets appearing elsewhere (e.g. \S \ref{partitioned linearly ordered sets}).
\end{definition}

We have an equivalence of categories
\begin{align*}
\Delta^{\operatorname{op}}\xrightarrow{\ \simeq \ } \operatorname{Ord}_\pm.
\end{align*}
On objects, it sends $[n]$ to $\{1<\ldots <n\}^\pm$. Writing $(n)=\{1<\ldots <n\}$, we send an order preserving map $\theta\colon [m]\rightarrow [n]$ to the map
\begin{align*}
(n)^\pm \rightarrow (m)^\pm,\quad i\mapsto \begin{cases}
\bot & \text{if }i \leq \theta(0)\\
j & \text{if }\theta(j-1) < i\leq \theta(j) \\
\top & \text{if }\theta(m)<i.\\
\end{cases}
\end{align*}
One should think of $(n)$ as the ordered set of morphisms in $[n]$; $i$ corresponds to the morphism $i-1<i$. The object $[0]$ is sent to $\emptyset^\pm$, which we may also denote by $(0)^\pm$.

The inverse equivalence sends a morphism $\alpha\colon (m)^\pm \rightarrow (n)^\pm$ to the map
\begin{align*}
[n]\rightarrow [m],\qquad i\mapsto \mathop{\max}_{\bot\leq j\leq i} \alpha^{-1}(j)
\end{align*}
where we abusively identify $0=\bot$ both in the indexing set when $i=0$ and if the maximum on the right hand side is the adjoined minimal element $\bot$. In view of this equivalence, we will from now on consider all simplicial objects as functors out of $\operatorname{Ord}_\pm$ unless otherwise specified. To identify them, it suffices to evaluate on the objects $(n)^\pm$, $n\in \N$.

\begin{example}\label{face map}
Under the equivalence $\Delta^{\operatorname{op}}\xrightarrow{\ \simeq \ } \operatorname{Ord}_\pm$ above, the face maps
$$[1] \cong \{i-1<i\} \hookrightarrow [n]$$
inducing the Segal maps are given by
\begin{equation}
\theta_i\colon (n)^\pm \rightarrow (1)^\pm,\quad j\mapsto \begin{cases}
\bot & j<i \\
1 & j=i \\
\top & j>i \\
\end{cases}
\end{equation}
These will show up in the proof of \Cref{monoid objects as product preserving functors} below.
\end{example}

\begin{definition}
Let $\mathcal{T}_A$ denote the opposite category of the full subcategory of discrete associative monoids spanned by those that are free and finitely generated. For any finite set $S$, we denote by $\operatorname{Free}(S)$ the free monoid on $S$.
\end{definition}

We have a functor
\begin{align*}
\operatorname{Free}\colon \operatorname{Ord}_\pm \rightarrow \mathcal{T}_A,\quad I^\pm \mapsto \operatorname{Free}(I),
\end{align*}
and sending a morphism $f\colon I^\pm \rightarrow J^\pm$ to the map of monoids
\begin{align*}
\operatorname{Free}(f)\colon \operatorname{Free}(J)\rightarrow \operatorname{Free}(I)
\end{align*}
that sends a generator $j\in J$ to the ordered product of generators corresponding to elements in the preimage $f^{-1}(j)$.

The following is a Lawvere theory for monoid objects following \cite[Proposition 5.6]{Yuan}.

\begin{proposition}\label{monoid objects as product preserving functors}
Let $\mathcal{C}$ be an $\infty$-category admitting finite products. The restriction functor
\begin{align*}
\operatorname{Free}^*\colon \operatorname{Fun}(\mathcal{T}_A, \mathcal{C})\rightarrow \operatorname{Fun}(\operatorname{Ord}_\pm, \mathcal{C})
\end{align*}
restricts to an equivalence of $\infty$-categories
\begin{align*}
\operatorname{Fun}^\times(\mathcal{T}_A, \mathcal{C})\xrightarrow{\ \simeq\ } \operatorname{Mon}(\mathcal{C}).
\end{align*}
where the left hand side is the full subcategory of product preserving functors.
\end{proposition}
\begin{proof}
First of all, note that $\operatorname{Free}^*$ maps the full subcategory of product preserving functors into the full subcategory of monoid objects. To see that the restriction is an equivalence, we will show that right Kan extension along $\operatorname{Free}$ defines an inverse equivalence. We'll temporarily assume that $\mathcal{C}$ is complete; it will be evident from the proof that we in fact only need finite products. Let $R\colon \operatorname{Fun}(\operatorname{Ord}_\pm, \mathcal{C})\rightarrow \operatorname{Fun}(\mathcal{T}_A, \mathcal{C})$ denote the right Kan extension functor and consider the unit and counit of the adjunction $\operatorname{Free}^*\dashv R$. We will show that for any $X\in \operatorname{Mon}(\mathcal{C})$, the associated counit transformation
\begin{align*}
\epsilon_X\colon \operatorname{Free}^*\circ R(X) \Rightarrow X
\end{align*}
is an equivalence. Since  $\operatorname{Free}\colon \operatorname{Ord}_\pm \rightarrow \mathcal{T}_A$ is essentially surjective, this will imply by the triangle identities that also the unit transformation
\begin{align*}
\eta_{f}\colon f\Rightarrow R\circ \operatorname{Free}^*(f)
\end{align*}
is an equivalence for any product preserving functor $f\colon \mathcal{T}_A\rightarrow \mathcal{C}$, and this exhibits $R$ as a inverse equivalence to $\operatorname{Free}^*$ as desired.

To prove the claim, we have to show that for all $(n)^\pm$ in $\operatorname{Ord}_\pm$, the projection map
\begin{align*}
\operatorname{Free}^*\circ R(X)((n)^\pm)=R(X)(\operatorname{Free}(n))\simeq \mathop{\operatorname{lim}}_{((m)^\pm,d)} X((m)^\pm )\xrightarrow{\ \epsilon_X(n)\ } X((n)^\pm)
\end{align*}
is an equivalence; here the limit is taken over the comma category
\begin{align*}
(\operatorname{Ord}_\pm)_n:=\operatorname{Ord}_\pm \mathop{\times}_{\mathcal{T}_A} (\mathcal{T}_A)_{\operatorname{Free}(n)/}
\end{align*}
whose objects are
\begin{align*}
\big((m)^\pm, d\colon\operatorname{Free}(m)\rightarrow \operatorname{Free}(n)\big)
\end{align*}
(recall that $\mathcal{T}_A$ is the category of finitely generated free monoids and \textit{opposite} maps between them); the map $\epsilon_X(n)$ is the projection to the value on the object $((n)^\pm,\operatorname{id})$. For notational ease, we also denote by $X$  the following composite with the projection for any $n$:
\begin{align*}
(\operatorname{Ord}_\pm)_n\rightarrow \operatorname{Ord}_\pm \xrightarrow{X} \mathcal{C}.
\end{align*}

Consider the inclusions:
\begin{align*}
\mathcal{G}_n^1\,\mathop{\hookrightarrow}^{i} \,\mathcal{G}_n\,\mathop{\hookrightarrow}^{j}\,(\operatorname{Ord}_\pm)_n,
\end{align*}
where $\mathcal{G}_n$ is  the full subcategory spanned by the objects $((m)^\pm,d)$ such that the map of monoids $d\colon \operatorname{Free}(m)\rightarrow \operatorname{Free}(n)$ sends generators to generators, and $\mathcal{G}_n^1$ is the further full subcategory spanned by the objects $((1)^\pm,d)$. Note that a map of monoids
\begin{align*}
i\colon \operatorname{Free}(1)\rightarrow \operatorname{Free}(n)
\end{align*}
sending $1$ to the $i$'th generator $x_i\in \operatorname{Free}(n)$ must be induced by the face map $\theta_i$ in $\operatorname{Ord}_\pm$ (see \Cref{face map}).

We claim that the inclusion $j\colon \mathcal{G}_n\hookrightarrow  (\operatorname{Ord}_\pm)_n$ admits a right adjoint; in particular, it is a $\varprojlim$-equivalence (\cite[Theorem 2.19 and Example 2.21]{ClausenOrsnesJansen}). To define the right adjoint, let $((m)^\pm,d)\in (\operatorname{Ord}_\pm)_n$. Suppose first of all that $m=1$ and let $\ell(d)$ denote the length of the string $d(1)$. Let $((\ell(d))^\pm,\hat{d})$ denote the object in $\mathcal{G}_n$ where
\begin{align*}
\hat{d}\colon \operatorname{Free}(\ell(d))\rightarrow \operatorname{Free}(n)
\end{align*}
sends the $i$'th generator to the $i$'th element showing up in the string of generators $d(1)$. The morphism $(\ell(d))^\pm \rightarrow (1)^\pm$ sending all $i=1,\ldots,\ell(d)$ to $1$ induces a counit morphism
\begin{align*}
\epsilon_{((1)^\pm,d)}\colon ((\ell(d))^\pm,\hat{d})\rightarrow ((1)^\pm,d).
\end{align*}
For general $m$, we restrict along $i\colon \operatorname{Free}(1)\rightarrow \operatorname{Free}(m)$ for each $i=1,\ldots,m$ and apply the definition above. More precisely, set
\begin{align*}
\ell(d):=\ell(d\circ 1)+ \cdots +\ell(d\circ m)
\end{align*}
and let $\hat{d}\colon \operatorname{Free}(\ell(d))\rightarrow \operatorname{Free}(n)$ be the morphism induced by the maps
\begin{align*}
\widehat{d\circ i}\colon \operatorname{Free}(\ell(d\circ i))\rightarrow \operatorname{Free}(n),\qquad i=1,\ldots,m.
\end{align*}
The maps $\epsilon_{((1)^\pm,d\circ i)}$ induce a counit morphism $((\ell(d))^\pm,\hat{d})\rightarrow ((m)^\pm,d)$. On morphisms, the right adjoint sends $((m)^\pm,d)\rightarrow ((k)^\pm,e)$ given by a map $\alpha\colon (m)^\pm\rightarrow (k)^\pm$, to the morphism $((\ell(d))^\pm,\hat{d})\rightarrow ((\ell(e))^\pm,\hat{e})$ given by the map
\begin{align*}
\hat{\alpha}\colon (\ell(d))^\pm\rightarrow (\ell(e))^\pm
\end{align*}
defined by the partitions
\begin{align*}
\ell(e\circ j)=\sum_{i\in \alpha^{-1}(j)}\ell(d\circ i),\qquad j=1,\ldots,k.
\end{align*}
We leave the details to the reader.

Now we note that the restriction $X\circ j$ coincides with the right Kan extension $\operatorname{Ran}_i(X\circ j\circ i)$: indeed, for a given $((m)^\pm,d)$ in $\mathcal{G}_n$, the comma category $(\mathcal{G}^1_n)_{((m)^\pm,d)/}$ identifies with the disjoint union of objects
\begin{align*}
((m)^\pm,d)\xrightarrow{\theta_i} ((1)^\pm,j_i),\qquad i=1,\ldots,m,
\end{align*}
where $j_i\in \{1,\ldots, n\}$ is such that $d$ maps the $i$'th generator in $\operatorname{Free}(m)$ to the $j_i$'th generator in $\operatorname{Free}(n)$.
It follows that the natural map
\begin{align*}
(X\circ j)((m)^\pm,d)\longrightarrow \mathop{\operatorname{lim}}_{(\mathcal{G}^1_n)_{((m)^\pm,d)/}} (X\circ j\circ i\circ p) \simeq \operatorname{Ran}_i(X\circ j\circ i)((m)^\pm,d)
\end{align*}
identifies with the Segal map 
\begin{align*}
X((m)^\pm)\rightarrow X((1)^\pm)^m
\end{align*}
induced by the face maps $\theta_i\colon (m)^\pm \rightarrow (1)^\pm$; this is an equivalence since $X$ belongs to $\operatorname{Mon}(\mathcal{C})$.

Consider the following commutative diagram:
\begin{center}
\begin{tikzpicture}
\matrix (m) [matrix of math nodes,row sep=2em,column sep=3em,nodes={anchor=center}]
{
\displaystyle\mathop{\operatorname{lim}}_{(\operatorname{Ord}_\pm)_n} X
& \displaystyle\mathop{\operatorname{lim}}_{\mathcal{G}_n} (X\circ j)
& \displaystyle\mathop{\operatorname{lim}}_{\mathcal{G}_n}\operatorname{Ran}_i (X\circ j\circ i)
& \displaystyle\mathop{\operatorname{lim}}_{\mathcal{G}^1_n}( X\circ j\circ i) \\
& X((n)^\pm) & \operatorname{Ran}_i (X\circ j\circ i)((n)^\pm,\operatorname{id}) & \\
};
\path[-stealth]
(m-1-1.10) edge node[above]{$\simeq$} (m-1-2.171)
(m-1-1) edge node[below left]{$\epsilon_X(n)$} (m-2-2)
(m-1-2.9) edge node[above]{$\simeq$} (m-1-3.175) 
(m-1-2) edge (m-2-2)
(m-1-3) edge (m-2-3)
(m-1-4.173) edge node[above]{$\simeq$} (m-1-3.5) 
(m-1-4) edge (m-2-3)
(m-2-2) edge node[below]{$\simeq$} (m-2-3)
;
\end{tikzpicture}
\end{center}
The left most upper horizontal map is an equivalence because $j$ is a $\varprojlim$-equivalence, the two horizontal maps in the middle are equivalences by the identification $X\circ j\simeq \operatorname{Ran}_i(X\circ j \circ i)$, and the right most horizontal map is an equivalence by general properties of right Kan extensions. The first three maps from the upper line to the lower line are given by projection to the value on the object $((n)^\pm,\operatorname{id})$; we want to show that these are equivalences. The right hand map from the upper to the lower line is the map
\begin{align*}
\displaystyle\mathop{\operatorname{lim}}_{\mathcal{G}^1_n} (X\circ j\circ i) \longrightarrow  \displaystyle\mathop{\operatorname{lim}}_{(\mathcal{G}^1_n)_{((n)^\pm,\operatorname{id})/}} (X\circ j\circ i\circ p) \simeq \operatorname{Ran}_i (X\circ j\circ i)((n)^\pm,\operatorname{id})
\end{align*}
induced by the canonical projection functor $p\colon (\mathcal{G}^1_n)_{((n)^\pm,\operatorname{id})/}\longrightarrow \mathcal{G}^1_n$. Now simply note that $p$ is an equivalence since any $((1)^\pm,d)$ in $\mathcal{G}_n^1$ must have $d$ induced by one of the face maps $\theta_i\colon (n)^\pm \rightarrow (1)^\pm$ of \Cref{face map}. It follows that $\epsilon_X(n)$ is an equivalence as claimed and that finishes the proof under the temporary assumption that $\mathcal{C}$ is complete. But it is immediate from the proof that we only need existence of finite products in $\mathcal{C}$ as this suffices to define the desired right adjoint.
\end{proof}

As advertised, this identification allows us to express the inclusion $\mathbb{B}$ in a different way.

\begin{notation}
Let $\operatorname{Ord}_\pm^\times$ denote the free product completion of $\operatorname{Ord}_\pm$. Explicitly, an object of $\operatorname{Ord}_\pm^\times$ consists of a finite set $S$ and a collection of objects $\{I_s^\pm\}_{s\in S}$. A morphism
\begin{align*}
\{I_s^\pm\}_{s\in S}\rightarrow \{J_t^\pm\}_{t\in T}
\end{align*}
is the data of a map $\gamma\colon T\rightarrow S$ and for all $t\in T$ a morphism $I_{\gamma(t)}^\pm \rightarrow J_t^\pm$ in $\operatorname{Ord}_\pm$.

The free product completion enjoys the following universal property: for any $\infty$-category $\mathcal{C}$ admitting finite products, restriction along the inclusion $\operatorname{Ord}_\pm\hookrightarrow \operatorname{Ord}_\pm^\times$ induces an equivalence
\begin{equation*}
\operatorname{Fun}^\times(\operatorname{Ord}_\pm^\times,\mathcal{C})\xrightarrow{\ \simeq\ }\operatorname{Fun}(\operatorname{Ord}_\pm,\mathcal{C}).\qedhere
\end{equation*}
\end{notation}

Since $\mathcal{T}_A$ admits finite products, the functor $\operatorname{Free}$ extends uniquely to a product preserving functor
\begin{align*}
\operatorname{Free}^\times\colon \operatorname{Ord}_\pm^\times\rightarrow \mathcal{T}_A.
\end{align*}

Let $\mathcal{C}$ be an $\infty$-category with finite products. Under the identification of \Cref{monoid objects as product preserving functors} and the identification given by the universal property of the free product completion, the functor $\mathbb{B}_\mathcal{C}$ exactly corresponds to restriction along $\operatorname{Free}^\times$:
\begin{center}
\begin{tikzpicture}
\matrix (m) [matrix of math nodes,row sep=2em,column sep=4em]
{
\operatorname{Mon}(\mathcal{C}) & \operatorname{Fun}(\operatorname{Ord}_\pm,\mathcal{C}) \\
\operatorname{Fun}^\times(\mathcal{T}_A,\mathcal{C}) & \operatorname{Fun}^\times(\operatorname{Ord}_\pm^\times, \mathcal{C}) \\
};
\path[-stealth]
(m-2-1) edge node[left]{$\simeq$} (m-1-1)
(m-2-2) edge node[right]{$\simeq$} (m-1-2)
(m-2-1) edge node[below]{$(\operatorname{Free}^\times)^*$} (m-2-2)
;
\path[right hook-stealth]
(m-1-1) edge node[above]{$\mathbb{B}_\mathcal{C}$} (m-1-2) 
;
\end{tikzpicture}
\end{center}

Following Yuan, we exploit this identification in \S \ref{identifying the monoids} below and show that for a cartesian closed cocomplete $\infty$-category $\mathcal{C}$, $\mathbb{B}_\mathcal{C}$ admits a left adjoint $\mathbb{L}_\mathcal{C}$ given by pointwise left Kan extension along $\operatorname{Free}^\times$. In the case of a presentable $\infty$-category $\mathcal{C}$, we already knew that such a left adjoint exists, but if we additionally assume that $\mathcal{C}$ is cartesian closed, this identification provides us with an accessible expression for $\mathbb{L}_\mathcal{C}$ (\Cref{monoid over J} and \Cref{monoid over Igg}). In particular, we get a more tangible description of partial algebraic K-theory and a tool for computing colimits in $\operatorname{Mon}(\mathcal{C})$, specifically $\operatorname{Mon}(\operatorname{Cat}_\infty)$ (cf. \Cref{E1-pushout diagram}).

\section{Categories of partitioned linearly ordered sets}\label{partitioned linearly ordered sets}

In this section we introduce some $1$-categories that will be useful to us in analysing certain colimits on our way to identifying a left adjoint to the inclusion
\begin{align*}
\mathbb{B}_\mathcal{C}\colon \operatorname{Mon}(\mathcal{C})\hookrightarrow \operatorname{Fun}(\Delta^{\operatorname{op}},\mathcal{C}).
\end{align*}
The category $\mathcal{I}$ introduced below is that of \cite[Definition 5.8]{Yuan}. We have had to introduce a slightly larger category $\mathcal{J}$ in order to provide the necessary flexibility in analysing the colimits, but the ideas all stem from \cite{Yuan}.

\subsection{Partitioned linearly ordered sets}

Let $\operatorname{Ord}$ denote the category of finite linearly ordered sets and let $\mathcal{J}=\operatorname{Tw}(\operatorname{Ord})^{\operatorname{op}}$ denote the opposite of the twisted arrow category of $\operatorname{Ord}$. Explicitly, the objects of $\mathcal{J}$ are order preserving maps $I\rightarrow P$ between finite linearly ordered sets, and a morphism from $I\rightarrow P$ to $J\rightarrow Q$ is a commutative diagram as below:
\begin{center}
\begin{tikzpicture}
\matrix (m) [matrix of math nodes,row sep=2em,column sep=2em]
{
I & J \\
P & Q \\
};
\path[-stealth]
(m-1-1) edge node[above]{$\theta$} (m-1-2) edge (m-2-1)
(m-1-2) edge (m-2-2)
(m-2-2) edge node[below]{$\rho$} (m-2-1)
;
\end{tikzpicture}
\end{center}

\begin{remark}
Clearly, $\mathcal{J}\simeq \operatorname{Tw}(\Delta)^{\operatorname{op}}$ via the inclusion $\Delta\hookrightarrow \operatorname{Ord}$ as a skeletal subcategory. This inclusion should not be confused with the equivalence $\Delta^{\operatorname{op}}\simeq \operatorname{Ord}_\pm$. In fact, the inclusion $\operatorname{Ord}\hookrightarrow \operatorname{Ord}_\pm$, $I\mapsto I^\pm$, turns out to play quite an important role in some of our arguments, see e.g. the proof of \Cref{left Kan extensions versus maps out of JS}.
\end{remark}

\begin{notation} \ 
\begin{enumerate}
\item We think of the objects of $\mathcal{J}$ as partitioned linearly ordered sets: a map $s\colon I\rightarrow P$ defines an ordered partition of $I$ into (possibly empty) intervals $I_p=s^{-1}(p)$, $p\in P$. To ease notation, we will often write $(I_p)_{p\in P}$ or even just $I_P$, omitting the \textit{partitioning map} $I\rightarrow P$.
\item Note that concatenation of finite linearly ordered sets equips $\mathcal{J}$ with a monoidal structure; we denote this by juxtaposition, $I_PJ_Q=(IJ)_{PQ}$.
\item A map of the form $(\theta,\operatorname{id}_P)\colon I_P\rightarrow J_P$ is called a \textit{collapse map}. A map of the form $(\operatorname{id}_I,\rho)\colon I_P\rightarrow I_Q$ is called a \textit{splitting map}. This terminology follows \cite{Yuan}, see also \Cref{splitting maps and collapse maps} where the intuition behind it may be clearer. \qedhere
\end{enumerate}
\end{notation}

Let $\mathcal{I}\subseteq\mathcal{J}$ denote the full subcategory spanned by the surjective maps $I\twoheadrightarrow P$; in other words we require the partitions to consist of only non-empty intervals. Consider moreover the subcategories $\mathcal{I}_\gg\subseteq \mathcal{I}$ and $\mathcal{J}_\gg\subseteq \mathcal{J}$ with the same objects but requiring the morphisms $(\theta,\rho)$ to have $\theta$ surjective. We have a commutative diagram of categories as below.

\begin{center}
\begin{tikzpicture}
\matrix (m) [matrix of math nodes,row sep=2em,column sep=2em]
{
\mathcal{I} & \mathcal{J} \\
\mathcal{I}_\gg & \mathcal{J}_\gg \\
};
\path[right hook-stealth]
(m-1-1) edge (m-1-2)
(m-2-1) edge (m-2-2)
;
\path[left hook-stealth]
(m-2-2) edge (m-1-2)
(m-2-1) edge (m-1-1)
;
\end{tikzpicture}
\end{center}

For the purposes of taking colimits, it turns out that we may choose between these categories as we please. Indeed, we will now show that all four inclusion are $\varinjlim$-equivalences, or cofinal (\cite[Theorem 2.19]{ClausenOrsnesJansen}). This should reflect the fact that the inclusion $\Delta_i^{\operatorname{op}}\hookrightarrow \Delta^{\operatorname{op}}$ of the subcategory on injective maps is a $\varinjlim$-equivalence. In our case, we require the maps to be surjective as we think of the maps in $\operatorname{Ord}$ as living inside $\operatorname{Ord}_\pm \simeq \Delta^{\operatorname{op}}$ and injections in $\Delta^{\operatorname{op}}$ correspond to surjections in $\operatorname{Ord}_\pm$. By working with the twisted arrow category, we may require these maps to be surjective in steps so that we get four \textit{different} categories as above; note that $\mathcal{I}_\gg$ identifies with (the opposite of) the twisted arrow category $\operatorname{Tw}(\operatorname{Ord}^\gg)$ of the subcategory $\operatorname{Ord}^\gg\subset \operatorname{Ord}$ on surjective maps.

The most important tool for us in this paper is the identification of the left adjoint to the inclusion $\mathbb{B}_{\mathcal{C}}$ of monoid objects into the $\infty$-category of simplicial objects. This identification ultimately consist of taking various colimits over the category $\mathcal{J}$ introduced above. The results of this section mean that we can alternate between viewing the colimits over the categories $\mathcal{I}$, $\mathcal{J}$, $\mathcal{I}_\gg$ and $\mathcal{J}_\gg$.

\begin{lemma}\label{Jgg in J is cofinal}
The inclusions $\mathcal{I}_\gg\hookrightarrow \mathcal{I}$ and $\mathcal{J}_\gg\hookrightarrow \mathcal{J}$ are $\varinjlim$-equivalences.
\end{lemma}
\begin{proof}
We will show that the right fibre of the inclusion $\mathcal{J}_\gg\hookrightarrow \mathcal{J}$ is contractible (\cite[Theorem 2.19]{ClausenOrsnesJansen}); the case $\mathcal{I}_\gg\hookrightarrow \mathcal{I}$ is analogous.

Let $s\colon I\rightarrow P$ be an object in $\mathcal{J}$ and consider the right fibre $\mathcal{C}:=\mathcal{J}_\gg\times_{\mathcal{J}}\mathcal{J}_{I_P/}$ over $I_P$. Its objects are given by commutative squares as on the left below, and a morphism is given by a commutative diagram as on the right: here the map $J\twoheadrightarrow J'$ is required to be surjective and the diagram defines a morphism from the left most inner commutative square to the outer commutative square:
\begin{center}
\begin{tikzpicture}
\matrix (m) [matrix of math nodes,row sep=2em,column sep=2em]
{
I & J & \quad & I & J & J' \\
P & Q & \quad & P & Q & Q' \\
};
\path[-stealth]
(m-1-1) edge (m-1-2) edge (m-2-1)
(m-1-2) edge (m-2-2)
(m-2-2) edge (m-2-1)
(m-1-4) edge (m-1-5) edge[bend left] (m-1-6)
(m-2-6) edge (m-2-5) edge[bend left] (m-2-4)
(m-2-5) edge (m-2-4)
(m-1-4) edge (m-2-4)
(m-1-5) edge (m-2-5)
(m-1-6) edge (m-2-6)
;
\path[->>]
(m-1-5) edge (m-1-6)
;
\end{tikzpicture}
\end{center}

We define two functors
\begin{align*}
F&\colon \mathcal{C}\rightarrow \mathcal{C},\qquad (I_P\xrightarrow{(\theta, \rho)} J_Q)\ \mapsto\ (I_P\xrightarrow{(\theta, \operatorname{id}_P)}J_P), \\
G&\colon \mathcal{C}\rightarrow \mathcal{C},\qquad (I_P\xrightarrow{(\theta, \rho)} J_Q)\ \mapsto\ (I_P\xrightarrow{(s, \operatorname{id}_P)}s(I)_P), \\
\end{align*}
where $J_P$ is given by the composite $J\rightarrow Q\rightarrow P$ and $s(I)_P$ is given by the inclusion $s(I)\hookrightarrow P$ (if $s$ is surjective, this is just the complete partition of $P$ given by the identity map $P\xrightarrow{=} P$). The functor $F$ sends a morphism given by $(\epsilon,\delta)\colon J_Q\rightarrow J'_{Q'}$ to the one given by $(\epsilon,\operatorname{id}_P)\colon J_P\rightarrow J'_P$, and $G$ sends all morphisms to the identity on $I_P\xrightarrow{(s,\operatorname{id}_P)} s(I)_P$.

We will now exhibit natural transformations
\begin{align*}
\operatorname{id}_{\mathcal{C}}\ \xLeftarrow{\ \alpha\ }\ F\  \xRightarrow{\ \beta \ } \ G.
\end{align*}
Since $G$ is the constant functor on $I_P\xrightarrow{(s,\operatorname{id}_P)} s(I)_P$, this will imply that $|\mathcal{C}|\simeq *$ and hence that the inclusion $\mathcal{J}_\gg\hookrightarrow \mathcal{J}$ is a $\varinjlim$-equivalence as claimed. For a given object $(\theta,\rho)\colon I_P\rightarrow J_Q$ of $\mathcal{C}$, the associated component of $\alpha$, respectively $\beta$, is given by the diagram below on the left, respectively right. Here we use that the identity on $J$ and the map $\rho\circ r\colon J\rightarrow s(I)$ are both surjective.
\begin{center}
\begin{tikzpicture}
\matrix (m) [matrix of math nodes,row sep=2em,column sep=3em]
{
I & J & J & \quad & I & J & s(I) \\
P & P & Q & \quad & P & P & P \\
};
\path[-stealth]
(m-1-1) edge node[above right]{$\theta$} (m-1-2) edge[bend left] node[above]{$\theta$} (m-1-3)
(m-2-3) edge node[below left]{$\rho$} (m-2-2) edge[bend left] node[below]{$\rho$} (m-2-1)
(m-1-1) edge node[left]{$s$} (m-2-1)
(m-1-2) edge node[right]{$\rho\circ r$} (m-2-2)
(m-1-3) edge node[right]{$r$}(m-2-3)
;
\path[-]
(m-2-2) edge[double equal sign distance] (m-2-1)
(m-1-2) edge[double equal sign distance] (m-1-3)
;
\path[-stealth]
(m-1-5) edge node[above right]{$\theta$} (m-1-6) edge[bend left] node[above]{$s$} (m-1-7)
(m-1-5) edge node[left]{$s$} (m-2-5)
(m-1-6) edge node[left]{$\rho\circ r$} (m-2-6)
(m-1-6) edge node[below]{$\rho\circ r$} (m-1-7)
;
\path[-]
(m-2-6) edge[double equal sign distance] (m-2-5)
(m-2-7) edge[double equal sign distance] (m-2-6) edge[bend left,double equal sign distance] (m-2-5)
;
\path[right hook-stealth]
(m-1-7) edge (m-2-7)
;
\end{tikzpicture}
\end{center}
It's straightforward to check the relevant commutativity conditions.
\end{proof}

\begin{remark}
As already mentioned, one can think of the subcategory $\mathcal{J}_\gg\hookrightarrow \mathcal{J}$ as an analogue of the subcategory $\Delta_i^{\operatorname{op}}\hookrightarrow \Delta^{\operatorname{op}}$ of injective order preserving maps and the lemma above can then be interpreted as a version of the well-known fact that $\Delta_i^{\operatorname{op}}\hookrightarrow \Delta^{\operatorname{op}}$ is a $\varinjlim$-equivalence. Indeed, the proof as exhibited is also similar to the proof of this claim found in \cite[Lemma 6.5.3.7]{LurieHTT}.
\end{remark}

\begin{lemma}\label{Igg in Jgg cofinal}
The inclusion $\iota\colon \mathcal{I}_\gg\hookrightarrow \mathcal{J}_\gg$ admits a left adjoint. In particular, it is a $\varinjlim$-equivalence.
\end{lemma}
\begin{proof}
The left adjoint $L\colon \mathcal{J}_\gg \rightarrow \mathcal{I}_\gg$ sends an object $s\colon I\rightarrow P$ to the object $s\colon I\twoheadrightarrow s(I)$. On morphisms, $L$ sends a diagram as on the left below to the diagram on the right, where we note that $\theta$ being surjective implies that $\rho\vert_{r(J)}$ factors through the image of $s$.
\begin{center}
\begin{tikzpicture}
\matrix (m) [matrix of math nodes,row sep=2em,column sep=2em]
{
I & J & \quad & I & J \\
P & Q & \quad & s(I) & r(J) \\
};
\path[-stealth]
(m-1-1) edge node[left]{$s$} (m-2-1)
(m-1-2) edge node[right]{$r$} (m-2-2)
(m-2-2) edge node[below]{$\rho$} (m-2-1)
(m-2-5) edge node[below]{$\rho$} (m-2-4)
;
\path[->>]
(m-1-1) edge node[above]{$\theta$} (m-1-2)
(m-1-4) edge node[above]{$\theta$} (m-1-5)
(m-1-4) edge node[left]{$s$} (m-2-4)
(m-1-5) edge node[right]{$r$} (m-2-5)
;
\end{tikzpicture}
\end{center}
The composite $L\circ \iota$ is the identity, and the maps
\begin{center}
\begin{tikzpicture}
\matrix (m) [matrix of math nodes,row sep=2em,column sep=2em]
{
I & I \\
P & s(I) \\
};
\path[-stealth]
(m-1-1) edge node[left]{$s$} (m-2-1)
(m-1-2) edge node[right]{$s$} (m-2-2)
;
\path[-]
(m-1-1) edge[double equal sign distance] (m-1-2)
;
\path[left hook-stealth]
(m-2-2) edge (m-2-1)
;
\end{tikzpicture}
\end{center}
define a unit transformation $\eta\colon \operatorname{id}_{\mathcal{J}_\gg}\Rightarrow \iota\circ L$.
\end{proof}

One can interpret a morphism in $\mathcal{J}$ as a way of ``refining'' the partitions. There is caveat to this interpretation, however, that we will now explain.

\begin{definition}\label{refinement}
Let $I_P=(I_p)_{p\in P}$ and $J_Q=(J_q)_{q\in Q}$ be objects of $\mathcal{J}$ and let $\theta \colon I\rightarrow J$ be an order preserving map. We say that $\theta$ \textit{refines the partitions} if the following condition holds:
\begin{equation*}
\text{for all }q\in Q, \text{ there exists a } p\in P \text{ such that }\theta^{-1}(J_q)\subset I_p.\qedhere
\end{equation*}
\end{definition}

It is clear that if $(\theta,\rho)\colon I_P\rightarrow J_Q$ is a morphism in $\mathcal{J}$, then $\theta$ refines the partitions.

\begin{remark}\label{refinement versus morphism}
It is important to note that given a morphism $(\theta, \rho)\colon I_P\rightarrow J_Q$ in $\mathcal{J}$, the map $\rho$ is not uniquely determined by $\theta$. In that sense, a morphism in $\mathcal{J}$ is more than just a refinement of the partition. Consider for example the following diagram where $\theta$ is the inclusion. The left and right vertical maps are the identity maps and should be viewed as complete partitions (i.e. partitioning into singletons). 
\begin{center}
\begin{tikzpicture}
\matrix (m) [matrix of math nodes,row sep=2em,column sep=2em]
{
\{1<3\} & \{1<2<3\} \\
\{1<3\} & \{1<2<3\} \\
};
\path[-stealth]
(m-1-1) edge (m-2-1) edge node[above]{$\theta$} (m-1-2)
(m-1-2) edge (m-2-2)
(m-2-2) edge[dashed] node[below]{$\rho$}(m-2-1)
;
\end{tikzpicture}
\end{center}
There are two possibilities for the map $\rho$: it can send $2$ to either $1$ or $3$. These will give rise to two \textit{distinct} morphisms in $\mathcal{J}$.

Note, however, that in $\mathcal{I}_\gg$, the map $\rho$ \textit{will} be uniquely determined by $\theta$ as both $\theta$ and the partitioning maps are surjective.
\end{remark}

In view of the above remark, even if an order preserving map $\theta\colon I\rightarrow J$ refines the partition, it need not give rise to a canonical morphism $I_P\rightarrow J_Q$ in $\mathcal{J}$. It does, however, give rise to a canonical zigzag as in the observation below. This will be important for us in the next section.

\begin{observation}\label{refinement zigzag}
Let $I_P$ and $J_Q$ be objects of $\mathcal{J}$ and suppose $\theta\colon I\rightarrow J$ is an order preserving map refining the partitions.

Setting $Q_\theta:=\{q\in Q\mid \theta^{-1}(J_q)\neq \emptyset\}$, we have an order preserving map $\rho_\theta\colon Q_\theta \rightarrow P$ sending $q\in Q_\theta$ to the unique $p\in P$ such that $\theta^{-1}(J_q)\subset I_p$ (such a $p$ exists by our assumption that $\theta$ refines the partitions and it must be unique since $\theta^{-1}(J_q)\neq \emptyset$ and the $I_p$ are disjoint). Let $\iota_\theta \colon Q_\theta\hookrightarrow Q$ denote the inclusion and consider the following commutative diagram.
\begin{center}
\begin{tikzpicture}
\matrix (m) [matrix of math nodes,row sep=2em,column sep=2em]
{
I & I & I & J \\
P & Q_\theta & Q & Q \\
};
\path[-stealth]
(m-1-1) edge node[left]{$s$} (m-2-1)
(m-1-2) edge node[left]{$\scriptstyle r\circ \theta$} (m-2-2)
(m-1-4) edge node[right]{$r$} (m-2-4)
(m-1-3) edge node[above]{$\theta$} (m-1-4) edge node[left]{$\scriptstyle r\circ \theta$} (m-2-3)
(m-2-2) edge node[below]{$\rho_\theta$}(m-2-1) edge node[below]{$\iota_\theta$} (m-2-3)
;
\path[-]
(m-1-1) edge[double equal sign distance] (m-1-2)
(m-1-2) edge[double equal sign distance] (m-1-3)
(m-2-3) edge[double equal sign distance] (m-2-4)
;
\end{tikzpicture}
\end{center}

This defines a canonical three-step zigzag in $\mathcal{J}$:
\begin{equation*}
I_P\xrightarrow{(\operatorname{id},\rho_\theta)} I_{Q_{\theta}}\xleftarrow{(\operatorname{id},\iota_\theta)} I_Q\xrightarrow{(\theta,\operatorname{id})} J_Q.\qedhere
\end{equation*}
\end{observation}

Let us make a final small remark before relating the categories introduced in this section to monoid objects. 

\begin{remark}
The empty set can be a source of confusion in these categories, so let us single out two important objects:
\begin{itemize}
\item In all four categories we have the object $\emptyset_\emptyset$ --- this plays an important role in \S \ref{maximally snug substrings}. In $\mathcal{I}$ and $\mathcal{I}^\gg$, this is an isolated object;
\item In $\mathcal{J}$ and $\mathcal{J}_\gg$ we also have an initial object $\emptyset_\ast$ --- the flexibility obtained by including this (and similar objects) is essential to our arguments.\qedhere
\end{itemize}
\end{remark}

\subsection{Relation to the simplex category}

We want to use the categories introduced in the previous section in combination with the Lawvere theory of \S \ref{Lawvere theory}. To this end, we relate the categories of partitioned linearly ordered sets to the simplex category, or rather to the free product completion $\operatorname{Ord}_\pm^\times$, and to the finite free monoids $\operatorname{Free}(S)$ in the category $\mathcal{T}_A$.

\begin{notation}\label{J and monoids and Ord}
\ 
\begin{enumerate}
\item Let $S$ and $T$ be finite sets. A map $\theta\colon \operatorname{Free}(T)\rightarrow \operatorname{Free}(S)$ of monoids induces a map $\mathcal{J}^S\rightarrow \mathcal{J}^T$ by sending a collection of partitioned totally ordered sets $I^s_{P^s}$, $s\in S$, to the collection $J^t_{Q^t}$, $t\in T$, defined as follows: given $t\in T$, write $\theta(t)=s_1s_2\cdots s_{k_t}\in \operatorname{Free}(S)$ as an ordered product of generators $s_i\in S$ and consider the partitioned totally ordered set
\begin{align*}
J^t_{Q^t}:=I^{s_1}_{P^{s_1}} I^{s_2}_{P^{s_2}}\cdots I^{s_{k_t}}_{P^{s_{k_t}}},\quad t\in T,
\end{align*}
given by concatenating the partitioned sets corresponding to the given elements in $S$ in the given order.
\item We have a functor
\begin{align*}
\zeta_S\colon \mathcal{J}^S\rightarrow \operatorname{Ord}_{\pm}^\times, \qquad I^s_{P^s}\mapsto \{(I_p^s)^\pm\}_{s\in S,p\in P^s}.
\end{align*}
On morphisms it sends a collection of morphisms $(\theta_s,\rho_s)\colon I^s_{P^s}\rightarrow J^s_{Q^s}$, $s\in S$, to the morphism
\begin{align*}
\{(I_p^s)^\pm\}_{s\in S,p\in P^s}\rightarrow \{(J^s_q)^\pm\}_{s\in S,q\in Q^s}
\end{align*}
given by the map
\begin{align*}
\rho=\coprod_{s\in S}\rho_s\colon \coprod_{s\in S} Q^s\rightarrow \coprod_{s\in S} P^s
\end{align*}
together with the maps
\begin{equation*}
\theta_q\colon (I^s_{\rho_s(q)})^\pm \rightarrow (J^s_q)^\pm,\quad i\mapsto \begin{cases}
\bot & \theta_s(i)< \min J^s_q \\
\theta_s(i) & \theta_s(i)\in J^s_q \\
\top & \theta_s(i)> \max J^s_q.
\end{cases}
\end{equation*}
See \Cref{motivate functor JS to Ord} below for some intuition behind this definition.
\item Consider the subcategory $\operatorname{Ord}_\pm^\gg\subset \operatorname{Ord}_\pm$ with the same objects but only surjective maps. Under the equivalence $\operatorname{Ord}_\pm\simeq \Delta^{\operatorname{op}}$, this corresponds to the subcategory $\Delta_i\subset \Delta$ on the same objects and injective maps between them --- in other words, we only include face maps.
\item A \textit{semisimplicial object} in an $\infty$-category $\mathcal{C}$ is a functor $\operatorname{Ord}_\pm^\gg\rightarrow \mathcal{C}$.
\item For a semisimplicial object $Y\colon \operatorname{Ord}_\pm^\gg\rightarrow \mathcal{C}$, we denote the induced product preserving map out of the free product completion of $\operatorname{Ord}_\pm^\gg$ by $Y^\times\colon (\operatorname{Ord}_\pm^\gg)^\times\rightarrow \mathcal{C}$.
\item Write $i\colon \operatorname{Ord}_\pm^\gg\hookrightarrow \operatorname{Ord}_\pm$ for the inclusion and moreover, $i^\times\colon (\operatorname{Ord}_\pm^\gg)^\times\hookrightarrow \operatorname{Ord}_\pm^\times$ for the induced functor on free product completions.
\item The functor $\zeta_S\colon \mathcal{J}^S\rightarrow \operatorname{Ord}_\pm^\times$ restricts to a functor $\zeta_S^\gg\colon \mathcal{J}_\gg^S\rightarrow (\operatorname{Ord}_\pm^\gg)^\times$.\qedhere
\end{enumerate} 
\end{notation}

\begin{remark}\label{motivate functor JS to Ord}
Let's very quickly explain the definition of the functor $\mathcal{J}^S\rightarrow \operatorname{Ord}_\pm^\times$ in item (2) above since it looks more complicated than it actually is. For clarity, we take the case $S=\ast$. Let $(\theta,\rho)\colon I_P\rightarrow J_Q$ in $\mathcal{J}$. Then the morphism
\begin{align*}
\{I_p^\pm\}_{p\in P}\rightarrow \{J_q^\pm\}_{q\in Q}
\end{align*}
is given by $\rho\colon Q\rightarrow P$ together with the maps
\begin{align*}
\theta_q\colon I_{\rho(q)}^\pm \rightarrow J_q,\quad
i\mapsto \begin{cases}
\bot & \theta(i)< \min J_q \\
\theta(i) & \theta(i)\in J_q \\
\top & \theta(i)> \max J_q.
\end{cases}
\end{align*}
Intuitively, this picks out the $q$-part of the restriction $\theta|_{I_{\rho(q)}}$. More explicitly, note that for any $p\in P$, the image of $\theta|_{I_p}$ is contained in $\coprod_{q\in \rho^{-1}(p)} J_q$. The map $\theta_q$ is equal to $\theta$ on the preimage $\theta^{-1}(J_q)$ and it sends everything below this to $\bot$ and everything above this to $\top$.
\end{remark}

\begin{definition}\label{simplicial space out of J-category}
Let $S$ be a finite set and $X\colon \operatorname{Ord}_\pm\rightarrow \mathcal{C}$ a simplicial object in an $\infty$-category $\mathcal{C}$. Consider the following composite
\begin{align*}
\mathbf{X}_S\colon \mathcal{J}^S\xrightarrow{\ \zeta_S\ } \operatorname{Ord}_\pm^\times\xrightarrow{X^\times} \mathcal{C}
\end{align*}
where the first map is the one defined in item (2) above and the second is the unique product preserving map given by the universal property of the free product completion. Note that
\begin{align*}
\mathbf{X}_S\simeq\prod_{s\in S}\mathbf{X}_s,
\end{align*}
so it often suffices to understand the case $S=\ast$. We will in general write $\mathbf{X}:=\mathbf{X}_\ast$ and $\mathbf{X}_n:=\mathbf{X}_{(n)}$ when $S=(n)$ is the finite set on $n$ elements.

For a semisimplicial object $Y\colon \operatorname{Ord}_\pm^\gg\rightarrow \mathcal{E}$, we analogously consider
\begin{align*}
\mathbf{Y}_S\colon \mathcal{J}_\gg^S\xrightarrow{\ \zeta_S^\gg\ }( \operatorname{Ord}_\pm^\gg)^\times\xrightarrow{Y^\times} \mathcal{C}
\end{align*}
and as above, we write $\mathbf{Y}:=\mathbf{Y}_\ast$  and $\mathbf{Y}_n:=\mathbf{Y}_{(n)}$.
\end{definition}

We make some basic observations that we will need later on.

\begin{lemma}\label{LKE and free product completion}
Let $\mathcal{E}$ be a cocomplete cartesian closed $\infty$-category and let $Y\colon \operatorname{Ord}_\pm^\gg\rightarrow \mathcal{E}$ be a semisimplicial object in $\mathcal{E}$. There is an equivalence
\begin{align*}
(\operatorname{Lan}_iY)^\times\simeq \operatorname{Lan}_{i^\times} (Y^\times),
\end{align*}
identifying the unique product preserving functor $\operatorname{Ord}_\pm^\times\rightarrow \mathcal{E}$ restricting to the left Kan extension of $Y$ along $i$ as the left Kan extension of $Y^\times$ along $i^\times$.
\end{lemma}
\begin{proof}
All we need to see is that $\operatorname{Lan}_{i^\times}(Y^\times)$ is product preserving and restricts to $\operatorname{Lan}_i Y$.

Let $\{I_s^\pm\}_{s\in S}$ in $\operatorname{Ord}_\pm^\times$, and consider the comma category
\begin{align*}
(i^\times)_{/\{I_s^\pm\}_{s\in S}}:=(\operatorname{Ord}_\pm^\gg)^\times \mathop{\times}_{\operatorname{Ord}_\pm^\times}(\operatorname{Ord}_\pm^\times)_{/\{I_s^\pm\}_{s\in S}}
\end{align*}
over which the value of left Kan extension along $i^\times$ on $\{I_s^\pm\}_{s\in S}$ is calculated. Consider also the full subcategory
\begin{align*}
(i^\times)_{/\{I_s^\pm\}_{s\in S}}^S \subset (i^\times)_{/\{I_s^\pm\}_{s\in S}}
\end{align*}
spanned by the objects
\begin{align*}
\{J_t^\pm\}_{t\in T}\rightarrow \{I_s^\pm\}_{s\in S}\quad\text{with}\quad \gamma=\operatorname{id}_S\colon S\rightarrow S=T.
\end{align*}
The inclusion
\begin{align*}
(i^\times)_{/\{I_s^\pm\}_{s\in S}}^S \hookrightarrow (i^\times)_{/\{I_s^\pm\}_{s\in S}}
\end{align*}
admits a left adjoint
\begin{align*}
\big(\{J_t^\pm\}_{t\in T}\rightarrow \{I_s^\pm\}_{s\in S}\big) \mapsto \big(\{J_{\gamma(s)}^\pm\}_{s\in S}\rightarrow \{I_s^\pm\}_{s\in S}\big).
\end{align*}
In particular, the inclusion is a $\varinjlim$-equivalence. Having assumed that $\mathcal{E}$ is cartesian closed, it follows that we can pull out a finite product over $S$:
\begin{align*}
\mathop{\operatorname{colim}}_{\{J_s^\pm\}_{s\in S}\rightarrow \{I_s^\pm\}_{s\in S}}Y^\times(\{J_s^\pm\}_{s\in S}) \xrightarrow{\ \simeq\ }\mathop{\operatorname{colim}}_{\{J_s^\pm\}_{s\in S}\rightarrow \{I_s^\pm\}_{s\in S}}\prod_{s\in S} Y(J_s^\pm)\xrightarrow{\ \simeq\ }\prod_{s\in S}\mathop{\operatorname{colim}}_{J_s^\pm\rightarrow I_s^\pm} Y(J_s^\pm)
\end{align*}
where the first two colimits are taken over $(i^\times)_{/\{I_s^\pm\}_{s\in S}}^S$ and on the far right they are taken over the comma categories of $i$ over $I_s^\pm$, $s\in S$.
It follows directly from this observation that $\operatorname{Lan}_{i^\times}(Y^\times)$ is product preserving and restricts to $\operatorname{Lan}_i Y$.
\end{proof}

The lemma below relates left Kan extensions with the association $X\mapsto \mathbf{X}^\times_S$ of \Cref{simplicial space out of J-category}. Diagrammatically, it reads as follows: for a semisimplicial object $Y\colon \operatorname{Ord}_\pm^\gg\rightarrow \mathcal{E}$ in a cartesian closed cocomplete $\infty$-category $\mathcal{E}$, the lower composite in the diagram below identifies with the left Kan extension of the upper composite along $j\colon \mathcal{J}_\gg^S\hookrightarrow \mathcal{J}^S$.
\begin{center}
\begin{tikzpicture}
\matrix (m) [matrix of math nodes,row sep=1em,column sep=2em]
{
\mathcal{J}_\gg^S & (\operatorname{Ord}_\pm^\gg)^\times & \\
& & & \mathcal{E} \\
\mathcal{J}^S & \operatorname{Ord}_\pm^\times & \\
};
\path[-stealth]
(m-1-2) edge node[above right]{$Y^\times$} (m-2-4)
(m-3-2) edge node[below right]{$(\operatorname{Lan}_i Y)^\times$} (m-2-4)
(m-1-1) edge node[above]{$\zeta_S^\gg$} (m-1-2)
(m-3-1) edge node[above]{$\zeta_S$} (m-3-2)
;
\path[right hook-stealth]
(m-1-1) edge node[left]{$j$} (m-3-1)
(m-1-2) edge node[right]{$i^\times$} (m-3-2)
;
\end{tikzpicture}
\end{center}

\begin{lemma}\label{left Kan extensions versus maps out of JS}
Let $S$ be a finite set, $\mathcal{E}$ a cartesian closed $\infty$-category and $Y\colon \operatorname{Ord}_\pm^\gg\rightarrow \mathcal{E}$ a semisimplicial object in $\mathcal{E}$. The natural comparison map
\begin{align*}
\operatorname{Lan}_{j}(\mathbf{Y}_S^\times) \rightarrow (\mathbf{Lan}_i \mathbf{Y})^\times_S
\end{align*}
is an equivalence.
\end{lemma}
\begin{proof}
Writing out the definitions and applying the identification of \Cref{LKE and free product completion} above, we have to show that the natural comparison map
\begin{align*}
\operatorname{Lan}_{j}(Y^\times\circ \zeta_S^\gg) \rightarrow (\operatorname{Lan}_i (Y^\times))\circ \zeta_S
\end{align*}
is an equivalence. Note that, $\mathcal{E}$ being cartesian closed, this comparison map is equivalent to 
\begin{align*}
\prod_{s\in S}\operatorname{Lan}_{j}(Y^\times\circ \zeta_s^\gg) \rightarrow \prod_{s\in S}(\operatorname{Lan}_i (Y^\times))\circ \zeta_s.
\end{align*}
Thus is suffices to show the claim for $S=\ast$. Write $\zeta=\zeta_*$ and $\zeta^\gg=\zeta^\gg_*$ For a given $I_P$ in $\mathcal{J}$, the comparison map is given by
\begin{align*}
\mathop{\operatorname{colim}}_{J_Q\rightarrow I_P} (Y^\times \circ \zeta^\gg)(J_Q) \rightarrow \mathop{\operatorname{colim}}_{J_Q\rightarrow I_P} \mathop{\operatorname{colim}}_{\{L_t^\pm\}_{t\in T} \rightarrow \zeta(J_Q)}Y^\times (\{L_t^\pm\}_{t\in T})\rightarrow  \mathop{\operatorname{colim}}_{\{L_t^\pm\}_{t\in T} \rightarrow \zeta(I_P)}Y^\times (\{L_t^\pm\}_{t\in T}).
\end{align*}

Analogously to the proof of \Cref{LKE and free product completion} above, we may for the colimits in question restrict to objects with $Q=T=P$ (and relevant maps given by the identity on $P$) which allows us to once again exploit the assumption that $\mathcal{E}$ is cartesian closed and pull out a finite product:
\begin{align*}
\prod_{p\in P}\mathop{\operatorname{colim}}_{J_p\rightarrow I_p} (Y^\times \circ \zeta^\gg)(J_p) \rightarrow \prod_{p\in P}\mathop{\operatorname{colim}}_{J_p\rightarrow I_p} \mathop{\operatorname{colim}}_{\{L_p^\pm\} \rightarrow \zeta(J_p)}Y^\times (\{L_p^\pm\})\rightarrow  \prod_{p\in P}\mathop{\operatorname{colim}}_{\{L_p^\pm\} \rightarrow \zeta(I_p)}Y^\times (\{L_p^\pm\})
\end{align*}

This reduces the comparison map to a finite product of morphisms of the following form:
\begin{align*}
\mathop{\operatorname{colim}}_{J\rightarrow I} Y(J^\pm)\rightarrow \mathop{\operatorname{colim}}_{L^\pm\rightarrow I^\pm} Y(L^\pm);
\end{align*}
here the colimits are taken over the comma categories
\begin{align*}
\operatorname{Ord}^\gg\mathop{\times}_{\operatorname{Ord}}(\operatorname{Ord})_{/I},\quad \text{respectively} \quad\operatorname{Ord}_\pm^\gg\mathop{\times}_{\operatorname{Ord}_\pm}(\operatorname{Ord}_\pm)_{/I^\pm},
\end{align*}
where $\operatorname{Ord}$ is the category of finite linearly ordered sets and order preserving maps and $\operatorname{Ord}^\gg$ is the subcategory on surjective maps. The map of colimits is induced by the functor
\begin{align*}
\operatorname{Ord}\rightarrow \operatorname{Ord}_\pm,\quad I\mapsto I^\pm;
\end{align*}
or rather the induced functor of comma categories. Now simply note that this map is a $\varinjlim$-equivalence as it admits a left adjoint:
\begin{align*}
(L^\pm\xrightarrow{\theta} I^\pm)\  \mapsto (\theta^{-1}(I) \rightarrow I).
\end{align*}
The counit transformation is the identity and the unit transformation is given by the map
\begin{align*}
L^\pm \twoheadrightarrow \theta^{-1}(I)^\pm,\quad l\mapsto \begin{cases}
\bot & l< \min \theta^{-1}(I) \\
l & l\in \theta^{-1}(I) \\
\top & l> \max \theta^{-1}(I). \\
\end{cases}
\end{align*}
This finishes the proof.
\end{proof}

\subsection{\texorpdfstring{$\mathcal{J}$}{J} is sifted}

The following is a reflection of the fact that the simplex category $\Delta^{\operatorname{op}}$ is sifted (as will also be apparent in the proof). Recall that a non-empty ($\infty$-)category $\mathcal{C}$ is sifted if the diagonal functor $\bigtriangleup \colon \mathcal{C}\rightarrow \mathcal{C}\times \mathcal{C}$, $x\mapsto (x,x)$, is a $\varinjlim$-equivalence (\cite[Definition 5.5.8.1]{LurieHTT}).

\begin{proposition}\label{J is sifted}
The category $\mathcal{J}$ is sifted.
\end{proposition}
\begin{proof}
We have a commutative diagram
\begin{center}
\begin{tikzpicture}
\matrix (m) [matrix of math nodes,row sep=2em,column sep=2em]
{
\mathcal{J} & \mathcal{J}\times \mathcal{J} \\
\operatorname{Ord}^{\operatorname{op}} & \operatorname{Ord} ^{\operatorname{op}}\times \operatorname{Ord}^{\operatorname{op}} \\
};
\path[-stealth]
(m-1-1) edge node[above]{$\bigtriangleup$} (m-1-2)
(m-2-1) edge node[below]{$\bigtriangleup$} (m-2-2)
(m-1-1) edge node[left]{$\pi$} (m-2-1)
(m-1-2) edge node[right]{$\pi\times\pi$} (m-2-2)
;
\end{tikzpicture}
\end{center}
where $\pi\colon \mathcal{J}=\operatorname{Tw}(\operatorname{Ord})^{\operatorname{op}}\rightarrow \operatorname{Ord}^{\operatorname{op}}$ is the projection $I_P\mapsto P$. The vertical maps are $\varinjlim$-equivalences by \cite[Lemma 4.16]{ClausenOrsnesJansen} (composed with the equivalence $\operatorname{Tw}(\mathcal{C})\simeq \operatorname{Tw}(\mathcal{C}^{\operatorname{op}})$). The lower map is a $\varinjlim$-equivalence since $\operatorname{Ord}^{\operatorname{op}}\simeq \Delta^{\operatorname{op}}$ is sifted (\cite[5.5.8.4]{LurieHTT}). We must to show that the upper horizontal map is a $\varinjlim$-equivalence; equivalently, that for any pair of objects $I_P$, $J_Q$ in $\mathcal{J}$, the right fibre
\begin{align*}
\bigtriangleup_{(I_P,J_Q)/}:=\mathcal{J}\mathop{\times}_{\mathcal{J}\times \mathcal{J}} \big(\mathcal{J}_{I_P/}\times \mathcal{J}_{J_Q/}\big)
\end{align*}
has contractible geometric realisation. Recall that the right fibre is an instance of the lax pullback as defined in \cite{Tamme}:
\begin{align*}
\bigtriangleup_{(I_P,J_Q)/}=\{(I_P,J_Q)\}\mathop{\times}^{\rightarrow}_{\mathcal{J}\times \mathcal{J}} \mathcal{J}.
\end{align*}
Using this, we have a sequence of functors
\begin{align}\label{sequence of functors}
\{(I_P,J_Q)\}\mathop{\times}^{\rightarrow}_{\mathcal{J}\times \mathcal{J}} \mathcal{J}\longrightarrow \big( \pi_{P/}\times \pi_{Q/}\big)\mathop{\times}^{\rightarrow}_{\mathcal{J}\times \mathcal{J}} \mathcal{J}\longrightarrow \{(P,Q)\}\mathop{\times}^{\rightarrow}_{\operatorname{Ord}^{\operatorname{op}}\times \operatorname{Ord}^{\operatorname{op}}} \operatorname{Ord}^{\operatorname{op}}\longrightarrow \ast.
\end{align}
The first map is induced by the inclusion
\begin{align*}
\iota\colon \{(I_P,J_Q)\}\,\hookrightarrow \,\pi_{P/}\times \pi_{Q/}
\end{align*}
into the right fibres of $\pi$ over $P$ and $Q$. As $\pi$ is a $\varinjlim$-equivalence, this inclusion (and hence the induced map on lax pullbacks) is a homotopy equivalence on geometric realisations, 

The third map is the terminal functor. Since the pullback 
\begin{align*}
\{(P,Q)\}\mathop{\times}^{\rightarrow}_{\operatorname{Ord}^{\operatorname{op}}\times \operatorname{Ord}^{\operatorname{op}}} \operatorname{Ord}^{\operatorname{op}}
\end{align*}
is simply the right fibre of the $\varinjlim$-equivalence $\bigtriangleup\colon \operatorname{Ord}^{\operatorname{op}}\rightarrow \operatorname{Ord}^{\operatorname{op}}\times \operatorname{Ord}^{\operatorname{op}}$, we see that the terminal functor is also a homotopy equivalence on geometric realisations.

Finally, the second map is induced by $\pi$ and we will show that this also induces a homotopy equivalence on geometric realisations. 
Write
\begin{align*}
C_{P,Q}:=\big( \pi_{P/}\times \pi_{Q/}\big)\mathop{\times}^{\rightarrow}_{\mathcal{J}\times \mathcal{J}} \mathcal{J}\quad\text{and}\quad \bigtriangleup_{(P,Q)/}:= \{(P,Q)\}\mathop{\times}^{\rightarrow}_{\operatorname{Ord}^{\operatorname{op}}\times \operatorname{Ord}^{\operatorname{op}}} \operatorname{Ord}^{\operatorname{op}}.
\end{align*}
Explicitly, $C_{P,Q}$ has as objects cospans in $\mathcal{J}$:
\begin{align*}
K_P\rightarrow H_R\leftarrow L_Q
\end{align*}
and morphisms are diagrams in $\mathcal{J}$ of the form
\begin{equation}\label{morphism in CPQ}
\begin{tikzpicture}[baseline=(current  bounding  box.center)]
\matrix (m) [matrix of math nodes,row sep=2em,column sep=2em]
{
K_P & H_R & L_Q\\
K'_P & H'_{R'} & L'_Q\\
};
\path[-stealth]
(m-1-1) edge (m-1-2) edge (m-2-1)
(m-1-2) edge (m-2-2)
(m-1-3) edge (m-1-2) edge (m-2-3)
(m-2-1) edge (m-2-2)
(m-2-3) edge (m-2-2)
;
\end{tikzpicture}
\end{equation}
where the left, respectively right, vertical map must be given by the identity on $P$, respectively $Q$. The objects of $\bigtriangleup_{(P,Q)/}$ are spans in $\operatorname{Ord}$:
\begin{align*}
P\leftarrow R\rightarrow Q
\end{align*}
and a morphism from $P\leftarrow R\rightarrow Q$ to $P\leftarrow R'\rightarrow Q$ is given by a map $R'\rightarrow R$ in $\Delta$ commuting with the maps to $P$ and $Q$. The functor $\pi\colon C_{P,Q}\rightarrow \bigtriangleup_{(P,Q)/}$ sends a cospan $K_P\rightarrow H_R\leftarrow L_Q$ to the underlying span $P\leftarrow R\rightarrow Q$ and likewise on morphisms.

Consider the full subcategory $\hat{C}_{P,Q}\subset C_{P,Q}$ spanned by objects of the form $H_P\rightarrow H_R\leftarrow H_Q$ with maps given by the identity on $H$. The inclusion $\hat{C}_{P,Q}\hookrightarrow C_{P,Q}$ admits a left adjoint
\begin{align*}
K_P\rightarrow H_R\leftarrow L_Q\quad \mapsto \quad H_P\rightarrow H_R\leftarrow H_Q
\end{align*}
where the maps from $H$ to $P$, respectively $Q$, is given by composing $H\rightarrow R$ with the map $R\rightarrow P$, respectively $R\rightarrow Q$. On a morphism as in (\ref{morphism in CPQ}) above, it likewise just replaces the maps $K\rightarrow K'$ and $L\rightarrow L'$ with $H\rightarrow H'$. The unit transformation is given by the morphisms as below with the maps $K\rightarrow H$ and $L\rightarrow H$ appearing in the left, respectively right, vertical maps.
\begin{center}
\begin{tikzpicture}
\matrix (m) [matrix of math nodes,row sep=2em,column sep=2em]
{
K_P & H_R & L_Q\\
H_P & H_R & H_Q\\
};
\path[-stealth]
(m-1-1) edge (m-1-2) edge (m-2-1)
(m-1-2) edge (m-2-2)
(m-1-3) edge (m-1-2) edge (m-2-3)
(m-2-1) edge (m-2-2)
(m-2-3) edge (m-2-2)
;
\end{tikzpicture}
\end{center}

Thus the inclusion $\hat{C}_{P,Q}\hookrightarrow C_{P,Q}$ induces a homotopy equivalence on geometric realisations and we are reduced to showing the restriction $\pi\colon \hat{C}_{P,Q}\rightarrow \bigtriangleup_{(P,Q)/}$ (which we'll also denote by $\pi$) induces a homotopy equivalence on geometric realisations; we do this by applying Quillen's Theorem A. To this end, let $P\leftarrow R\rightarrow Q$ be an object in $\bigtriangleup_{(P,Q)/}$ and consider the left fibre
\begin{align*}
F=F_{/(P\leftarrow R\rightarrow Q)}:=\hat{C}_{P,Q}\mathop{\times}_{\bigtriangleup_{(P,Q)/}} \big(\bigtriangleup_{(P,Q)/}\big)_{/(P\leftarrow R\rightarrow Q)}.
\end{align*}
The objects of $F$ consist of a cospan
\begin{align*}
H_P\rightarrow H_S\leftarrow H_Q
\end{align*}
together with a map $S\rightarrow R$ commuting with the maps to $P$ and $Q$. A morphism
\begin{align*}
(H_P\rightarrow H_S\leftarrow H_Q,\ S\rightarrow R)\,\longrightarrow\, (H'_P\rightarrow H'_{S'}\leftarrow H'_Q,\ S'\rightarrow R)
\end{align*}
is given by a morphism of the cospans such that the map $S'\rightarrow S$ commutes with the maps to $R$.

Consider the full subcategory $F^R\subset F$ spanned by the objects $H_P\rightarrow H_R\leftarrow H_Q$ together with the identity map from $R$ to $R$. We claim that the inclusion $F^R\hookrightarrow F$ admits a right adjoint:
\begin{align*}
(H_P\rightarrow H_S\leftarrow H_Q,\ S\rightarrow R)\quad \mapsto\quad (H_P\rightarrow H_R\leftarrow H_Q,\ R=R)
\end{align*}
given by simply composing the maps $H\rightarrow S$ and $S\rightarrow R$. The counit adjunction is given by the following morphisms defined by the map $S\rightarrow R$:
\begin{center}
\begin{tikzpicture}
\matrix (m) [matrix of math nodes,row sep=2em,column sep=2em]
{
H_P & H_R & H_Q\\
H_P & H_S & H_Q\\
};
\path[-stealth]
(m-1-1) edge (m-1-2) edge (m-2-1)
(m-1-2) edge (m-2-2)
(m-1-3) edge (m-1-2) edge (m-2-3)
(m-2-1) edge (m-2-2)
(m-2-3) edge (m-2-2)
;
\end{tikzpicture}
\end{center}

Finally, note that $F^R$ has a terminal object
\begin{align*}
R_P\rightarrow R_R\leftarrow R_Q.
\end{align*}
The unique morphisms to the terminal object are given by the map $H\rightarrow R$ and the identity on $R$:
\begin{center}
\begin{tikzpicture}
\matrix (m) [matrix of math nodes,row sep=2em,column sep=2em]
{
H_P & H_R & H_Q\\
R_P & R_R & R_Q\\
};
\path[-stealth]
(m-1-1) edge (m-1-2) edge (m-2-1)
(m-1-2) edge (m-2-2)
(m-1-3) edge (m-1-2) edge (m-2-3)
(m-2-1) edge (m-2-2)
(m-2-3) edge (m-2-2)
;
\end{tikzpicture}
\end{center}
It follows that $|F|\simeq \ast$ and thus by Quillen's Theorem A, the functor $\pi\colon \hat{C}_{P,Q}\rightarrow \bigtriangleup_{(P,Q)/}$ induces a homotopy equivalence on geometric realisations. Thus, to wrap up, the sequence of functors (\ref{sequence of functors}) induces a homotopy equivalence on geometric realisations and we have shown that all right fibres of the diagonal functor on $\mathcal{J}$ have contractible geometric realisations, so $\mathcal{J}$ is sifted as claimed.
\end{proof}

\section{Universal monoids from simplicial objects}\label{unravelling}

We wish to exhibit a tangible formula for the left adjoint $\mathbb{L}_\mathcal{C}$ to the inclusion
\begin{align*}
\mathbb{B}_\mathcal{C}\colon \operatorname{Mon}(\mathcal{C})\hookrightarrow \operatorname{Fun}(\Delta^{\operatorname{op}}, \mathcal{C})
\end{align*}
when it exists. Recall that we can identify $\mathbb{B}_{\mathcal{C}}$ with the restriction map:
\begin{align*}
(\operatorname{Free}^\times)^*\colon \operatorname{Fun}^\times(\mathcal{T}_A,\mathcal{C})\rightarrow \operatorname{Fun}^\times( \operatorname{Ord}_\pm^\times, \mathcal{C}).
\end{align*}
We use this to identify the left adjoint $\mathbb{L}_{\mathcal{C}}$ with the left Kan extension along $\operatorname{Free}^\times$. This identification is due to Yuan in the case of $\mathbb{E}_1$-spaces (\cite[Proposition 5.10]{Yuan}); the objective there is a homology calculation, so the focus is on identifying the underlying space. Since we want to keep track of the $\mathbb{E}_1$-structure, we spell out the identifications in detail in order to keep track of the additional structure. Moreover, we want an analogous result for monoidal $\infty$-categories, so we treat the general case of monoid objects in a cartesian closed cocomplete $\infty$-category $\mathcal{E}$.

\subsection{Maximally snug substrings}\label{maximally snug substrings}

In this section we do the technical legwork enabling the identification of the left adjoint $\mathbb{L}$ as a left Kan extension (\Cref{identifying L as left Kan extension}). We analyse the relevant slice categories over which the left Kan extension of $\operatorname{Free}^\times$ is calculated and we show that for the purposes of calculating colimits, these can be replaced by the products $\mathcal{J}^S$ for varying finite sets $S$.

Let $S$ be a finite set and consider the monoid $\operatorname{Free}(S)$. Write
\begin{align*}
\mathcal{K}_S:=\operatorname{Ord}_\pm^\times \mathop{\times}_{\mathcal{T}_A} (\mathcal{T}_A)_{/\operatorname{Free}(S)}
\end{align*}
for the left fibre of $\operatorname{Free}^\times\colon \operatorname{Ord}_\pm^\times\rightarrow \mathcal{T}_A$ over $\operatorname{Free}(S)$. Explicitly, the objects of $\mathcal{K}_S$ are given by
\begin{align*}
(\{I_t^\pm\}_{t\in T},\{x_s\}_{s\in S})
\end{align*}
where $\{I_t^\pm\}_{t\in T}$ is an object of $\operatorname{Ord}_\pm^\times$ and $x_s\in \operatorname{Free}(\coprod_t I_t)$, $s\in S$. A morphism
\begin{align*}
(\{I_t^\pm\}_{t\in T},\{x_s\}_{s\in S})\ \longrightarrow\ (\{J_r^\pm\}_{r\in R},\{y_s\}_{s\in S})
\end{align*}
consists of a morphism $\{I_t^\pm\}_{t\in T}\rightarrow \{J_t^\pm\}_{r\in R}$ in $\operatorname{Ord}_\pm^\times$ such that the induced map on monoids $\operatorname{Free}(\coprod_r J_r)\rightarrow \operatorname{Free}(\coprod_t I_t)$ sends $y_s$ to $x_s$ for all $s\in S$. Note that the sets $S$, $T$ and $R$ are not required to be linearly ordered.

Now, consider the product $\mathcal{J}^S$ and define a functor
\begin{align*}
G_S\colon \mathcal{J}^S\rightarrow \mathcal{K}_S
\end{align*}
sending a collection of partitioned ordered sets $I^s_{P^s}$, $s\in S$, to the object
\begin{align*}
\big(\{(I^s_p)^\pm\}_{s\in S,\,p\in P^s},\  (x_{I^s})_{s\in S}\big)
\end{align*}
where $x_{I^s}\in \operatorname{Free}(\coprod_{s\in S}\coprod_{p\in P_s} I^s_p)$ is the element given by the ordered product of the generators corresponding to the elements of $I^s$: writing out $P^s=\{1<\cdots <k_s\}$ and each $I^s_p=\{x_1^{(s,p)}<x_2^{(s,p)}<\cdots < x_{n_p}^{(s,p)}\}$, $p\in P^s$, we have
\begin{align}\label{xI}
x_{I^s}=\bigg(x_{1}^{(s,1)}x_{2}^{(s,1)}\cdots x_{n_1}^{(s,1)}\bigg)\bigg(x_{1}^{(s,2)}x_{2}^{(s,2)}\cdots x_{n_2}^{(s,2)}\bigg)\cdots \bigg(x_{1}^{(s,k_s)}x_{2}^{(s,k_s)}\cdots x_{n_{k_s}}^{(s,k_s)}\bigg).
\end{align}
A collection of morphisms $\phi_s\colon I^s_{P^s}\rightarrow J^s_{Q^s}$, $s\in S$, in $\mathcal{J}$ defines a morphism in $\operatorname{Ord}_\pm^\times$ via the functor $\mathcal{J}^S\rightarrow \operatorname{Ord}_\pm^\times$ (\Cref{J and monoids and Ord}). This defines a map in $\mathcal{K}_S$: indeed, the induced map of monoids sends the element $x_{I^s}$ to $x_{J^s}$ as the map $I^s\rightarrow J^s$ given by $\phi_s$ is order preserving.

\begin{remark}
To avoid confusion, let's spell out what happens to the objects given by empty sets; the difference between $\emptyset_\emptyset$ and $\emptyset_\ast$ is crucial here.

Let $S=\ast$ and consider the functor $G_\ast\colon \mathcal{J}\rightarrow \mathcal{K}_\ast$. This sends $\emptyset_\emptyset$ to the object $(\emptyset, 0)$ in $\mathcal{K}$ and it sends $\emptyset_\ast$ to the object $(\{\emptyset^\pm\},0)$ in $\mathcal{K}_\ast$. In the former case, we have the \textit{empty collection} of objects in $\operatorname{Ord}_\pm$, whereas in the latter case we have the collection with a single object, namely the empty set viewed as an object in $\operatorname{Ord}_\pm$. An object $\emptyset_P$ is sent to $(\{\emptyset^\pm\}_{p\in P},0)$ with a copy of $\emptyset^\pm$ for each $p\in P$.

For general $S$, suppose given a collection of partioned ordered sets $I^s_{P^s}$, $s\in S$. If $I^s_{P^s}=\emptyset_\emptyset$ for some $s$, then under $G_S$, $I^s_{P^s}$ does not contribute to the resulting object of $\operatorname{Ord}_\pm^\times$ and we have $x_s=0$ in the corresponding monoid. If $I^s_{P^s}=\emptyset_\ast$ for some $s$, then under $G_S$, $I^s_{P^s}$ contributes with a copy of $\emptyset^\pm$ to the collection of objects in $\operatorname{Ord}_\pm$ and we also have $x_s=0$ in the corresponding monoid.
\end{remark}

\begin{definition}\label{partition into maximally snug substrings}
Let $\{I_t^\pm\}_{t\in T}$ be an object in $\operatorname{Ord}_\pm^\times$ and let $x\in \operatorname{Free}(\coprod_t I_t)$. For each $t\in T$, write out $I_t=\{x_1^t<x_2^t<\cdots x_{n_t}^t\}$.
\begin{enumerate}
\item A non-zero element $y\in \operatorname{Free}(\coprod_t I_t)$ is called \textit{snug} if it is of the form 
\begin{align*}
y=x_j^tx_{j+1}^t \cdots x_{k-1}^t x_k^t,\quad \text{for some }t\in T, \ 1\leq j\leq k\leq n_t.
\end{align*}
In words, $y$ is snug if it is the consecutive ordered product of generators from the same linearly ordered set $I_t$.
\item A \textit{snug substring} of $x$ is a snug element $y$ such that $x=ayb$ for some $a,b\in \operatorname{Free}(\coprod_t I_t)$.
\item A snug substring of $x$ is \textit{maximal} if it is not a proper snug substring of a larger snug substring (respecting the position of the substrings in $x$).
\item The string $x$ admits a unique partition into maximally snug substrings. Viewing $x$ as a finite linearly ordered set by ordering the generators according to the string, this defines an element in $\mathcal{I}$ that we denote by $I^x_{P^x}$. In other words, $I^x$ is the ordered string $x$ and $P^x$ is the ordered set of maximally snug substrings of $x$.\qedhere
\end{enumerate}
\end{definition}

Here are some examples of (4) above.

\begin{example} \ 
\begin{enumerate}
\item Consider the object $\{\{1<2<3\}^\pm, \{4<5\}^\pm, \{6\}^\pm \}$ in $\operatorname{Ord}_\pm^\times$ and the elements
\begin{align*}
x=123456,\quad y=654321,\quad z=12423456\quad \text{in }\operatorname{Free}(\{1,2,3,4,5,6\}).
\end{align*}
The unique partitions of these elements into maximally snug substrings are
\begin{align*}
x=(123)(45)(6),\quad y=(6)(5)(4)(3)(2)(1),\quad z=(12)(4)(23)(45)(6).
\end{align*}
The associated elements in $\mathcal{I}$ are the order preserving maps below
\begin{align*}
&I^x=\{1<2<3<4<5<6\}\ \longrightarrow\ \{(123)<(45)<(6)\} =P^x, \\
&I^y=\{6<5<4<3<2<1\}\ \longrightarrow\ \{(6)<(5)<(4)<(3)<(2)<(1)\} =P^y, \\
&I^z=\{1<2<4<2<3<4<5<6\}\ \longrightarrow\ \{(12)<(4)<(23)<(45)<(6)\} =P^z.
\end{align*}

\item For $v=0$, we have $I^v_{P^v}=\emptyset_\emptyset$.

\item Let $I_P$ be an object of $\mathcal{J}$ and consider $G_\ast(I_P)=(\{I_p^\pm\}_{p\in P}, x_I)$ in $\mathcal{K}_\ast$. \Cref{xI} above exactly writes out the partition of $x_I$ into maximally snug substrings. In other words, we can identify $I^{x_I}_{P^{x_I}}\cong I_P$.\qedhere
\end{enumerate}
\end{example}

\begin{observation}\label{refining maximally snug substrings}
Consider a morphism in $\operatorname{Ord}_\pm^\times$,
\begin{align*}
\epsilon\colon \{I_t^\pm \}_{t\in T} \longrightarrow \{J_r^\pm \}_{r\in R};
\end{align*}
this is given by a map of sets $\gamma_\epsilon\colon R\rightarrow T$ and for each $r\in R$, a map $\epsilon_r\colon I_{\gamma_\epsilon(r)}\rightarrow J_r$ in $\operatorname{Ord}_\pm$. Let $x\in \operatorname{Free}(\coprod_t I_t)$, $y\in \operatorname{Free}(\coprod_r J_r)$ and assume that the induced map of monoids
\begin{align*}
\epsilon_*\colon \operatorname{Free}(\coprod_r J_r)\rightarrow \operatorname{Free}(\coprod_t I_t)
\end{align*}
sends $y$ to $x$. Write $I^y=\{y_1<\cdots < y_m\}$ and $I^x=\{x_1<\cdots <x_n\}$; that is, we are writing $y=y_1\cdots y_m$ and $x=x_1\cdots x_n$ as strings of generators (there is no notational relationship between the elements $x_i$ and the generators coming from the $I_t$ --- indeed some of the $x_i$ may be the same generator, but they define different elements in the \textit{linearly ordered set} $I^x$). Then 
\begin{align*}
x_1\cdots x_n =\epsilon_*(y)=\epsilon_*(y_1) \cdots \epsilon_*(y_m)
\end{align*}
and we define an order preserving map
\begin{align*}
\hat{\epsilon}\colon\{x_1<\cdots <x_n\}\rightarrow \{y_1<\cdots < y_m\}
\end{align*}
by sending $x_j$ to the unique $y_i$ such that $x_j$ is part of the substring $\epsilon_*(y_i)$. Since the maps $\epsilon_i$ are order preserving, $\epsilon_*$ sends a snug substring of $y$ to a snug (possibly empty) substring of $x$, but it does not necessarily preserve maximality. Hence the order preserving map that we just defined \textit{refines} the partitions into maximally snug substrings in the sense of \Cref{refinement}: the preimage of a maximally snug substring of $y$ is a substring of a maximally snug substring of $x$. Note that the map $\hat{\epsilon}$ need not be surjective.
\end{observation}

The following lemma will be crucial to the unravelling of the left adjoint $\mathbb{L}$ in the next section.

\begin{lemma}\label{cofinal}
The functor $G_S\colon \mathcal{J}^S\rightarrow \mathcal{K}_S$ is a $\varinjlim$-equivalence.
\end{lemma}
\begin{proof}
We will show that the right fibres of this functor are contractible. To ease notation, we do the case $S=\ast$; the general case is completely analogous, but with additional super- and subscripts all over the place. Write $\mathcal{K}=\mathcal{K}_\ast$ for the left fibre of $\operatorname{Free}^\times $ over $\operatorname{Free}(\ast)$ and consider the functor $G=G_\ast\colon \mathcal{J}\rightarrow \mathcal{K}$ defined as above. Let $(\{I_t^\pm\}_{t\in T},x)$ be an object of $\mathcal{K}$ and let
\begin{align*}
\mathcal{C}:=\mathcal{J} \mathop{\times}_{\mathcal{K}} \mathcal{K}_{(\{I_t^\pm\}_{t\in T},x)/}
\end{align*}
denote the right fibre of $G$ over $(\{I_t^\pm\}_{t\in T},x)$. An object of $\mathcal{C}$ consists of an object $J_Q$ in $\mathcal{J}$ together with a morphism
\begin{align*}
\epsilon\colon (\{I_t^\pm\}_{t\in T}, \, x)\rightarrow G(J_Q)=(\{J_q^\pm\}_{q\in Q}, \, x_J)
\end{align*}
given by a map of sets $\gamma_\epsilon\colon Q\rightarrow T$ together with a map $\epsilon_q\colon I_{\gamma_\epsilon(q)}^\pm \rightarrow J_q^\pm$ in $\operatorname{Ord}_\pm$ for each $q\in Q$ such that the induced map of monoids, $\epsilon_*\colon \operatorname{Free}(J)\rightarrow \operatorname{Free}(\coprod_{t}I_t)$, sends $x_J$ to $x$. The morphisms of $\mathcal{C}$ are induced by morphisms in $\mathcal{J}$ as in the diagram below where the morphism $(\theta,\rho)\colon J_Q\rightarrow K_W$ in $\mathcal{J}$ defines a morphism $\epsilon\rightarrow \delta$ in $\mathcal{C}$ if the diagram commutes.
\begin{center}
\begin{equation}\label{morphism}
\begin{tikzpicture}[baseline=(current  bounding  box.center)]
\matrix (m) [matrix of math nodes,row sep=1em,column sep=3em]
{
& G(J_Q) \\
(\{I_t^\pm\}_{t\in T}, \, x) & \\
& G(K_W) \\
};
\path[-stealth]
(m-2-1) edge node[above left]{$\epsilon$} (m-1-2) edge node[below left]{$\delta$} (m-3-2)
(m-1-2) edge node[right]{$G(\theta,\rho)$} (m-3-2)
;
\end{tikzpicture}
\end{equation}
\end{center}

Consider the partitioned ordered set $I^x_{P^x}$ given by partitioning the string $x$ into maximally snug substrings (\Cref{partition into maximally snug substrings}). We begin by defining an object of $\mathcal{C}$:
\begin{align*}
\eta\colon (\{I_t^\pm\}_{t\in T},\,x)\longrightarrow G(I^x_{P^x})= (\{(I^x_r)^\pm \}_{r\in P^x},\, x).
\end{align*}

Write $I_t=\{x_1^{(t)}<x_2^{(t)}<\ldots<x_{n_t}^{(t)}\}$ for each $t\in T$ and write out $x$ into maximally snug substrings, i.e.
\begin{align*}
x=\bigg(x_{j_1}^{(i_1)}x_{j_1+1}^{(i_1)}\cdots x_{k_1}^{(i_1)}\bigg)\bigg(x_{j_2}^{(i_2)}x_{j_2+1}^{(i_2)}\cdots x_{k_2}^{(i_2)}\bigg)\cdots \bigg(x_{j_l}^{(i_l)}x_{j_l+1}^{(i_l)}\cdots x_{k_l}^{(i_l)}\bigg)
\end{align*}
for $\{i_1,\ldots,i_l\}\subset T$ and $1\leq j_r\leq k_r\leq n_{i_r}$, $r=1,\ldots,l$. Then we may identify
\begin{align*}
P^x=\{1<\cdots <l\} \quad\text{ and further }\quad I^x_r=\{x_{j_1}^{(i_r)}<\cdots < x_{k_1}^{(i_r)}\}
\end{align*}
for each $r=1,\ldots,l$. We define a morphism in $\operatorname{Ord}_\pm^\times$
\begin{align*}
\eta\colon \{I_t^\pm\}_{t\in T}\rightarrow \{(I^x_r)^\pm \}_{r\in P^x}
\end{align*}
by the map of sets $\gamma_\eta\colon \{1,\ldots,l\}\rightarrow T$, $r\mapsto i_r$, together with the maps
\begin{align*}
\eta_r\colon I_{i_r}^\pm\rightarrow (I^x_r)^\pm,\quad x^{i_r}_j\mapsto \begin{cases}
\bot & j< j_r \\
x^{i_r}_j & j_r\leq j\leq k_r \\
\top & k_r< j
\end{cases}
\end{align*}
in $\operatorname{Ord}_\pm$ (these maps $I_{i_r}^\pm\rightarrow (I^x_r)^\pm$ should be interpreted as identifying $I^x_r$ as a subset of $I_{i_r}$). This defines an object $\eta\colon (\{I_t^\pm\}_{t\in T},\,x)\longrightarrow G(I^x_{P^x})$ of $\mathcal{C}$ as desired. We will show that $\mathcal{C}$ has contractible realisation by showing that it contracts onto the object $\eta$; we do this by exhibiting a zigzag of natural transformations
\begin{align*}
\operatorname{id}_\mathcal{C}\xLeftarrow{\ \alpha\ } A \xRightarrow{\ \beta\ } B  \xLeftarrow{\ \gamma\ } \operatorname{const}_\eta.
\end{align*}
The key input here will be \Cref{refinement zigzag}; indeed, those zigzags will define the components of the natural transformations above. Note first of all that by \Cref{refining maximally snug substrings}, an object
\begin{align*}
\epsilon\colon (\{I_t^\pm\}_{t\in T}, \, x)\longrightarrow G(J_Q)=(\{J_q^\pm \}_{q\in Q}, \, x^J)
\end{align*}
in $\mathcal{C}$ gives rise to an order preserving map $\hat{\epsilon}\colon I^x\rightarrow J$ satisfying the hypothesis of \Cref{refinement zigzag} and hence, we have a zigzag of morphisms in $\mathcal{J}$:
\begin{align*}
J_Q\xleftarrow{(\hat{\epsilon},\operatorname{id})}I^x_Q\xrightarrow{(\operatorname{id},\iota_\epsilon)} I^x_{Q_{\epsilon}}  \xleftarrow{(\operatorname{id},\rho_\epsilon)}I^x_{P^x}.
\end{align*}

Using this, we first of all define the functors $A,B\colon \mathcal{C}\rightarrow \mathcal{C}$. The functor $A$ sends an object $\epsilon\colon (\{I_t^\pm\}_{t\in T}, \, x)\rightarrow G(J_Q)$ to 
\begin{align*}
A(\epsilon)\colon (\{I_t^\pm\}_{t\in T}, \, x)\longrightarrow G(I^x_Q)=(\{\hat{\epsilon}^{-1}(J_q)^\pm\}_{q\in Q},\, x)
\end{align*}
given by the map of sets $\gamma_\epsilon\colon Q\rightarrow T$ together with the maps
\begin{align*}
I_{\gamma_\epsilon(q)}^\pm \rightarrow \hat{\epsilon}^{-1}(J_q)^\pm,\quad i\mapsto \begin{cases}
\bot & i< \min \hat{\epsilon}^{-1}(J_q) \\
i & i\in \hat{\epsilon}^{-1}(J_q) \\
\top & i> \max \hat{\epsilon}^{-1}(J_q)\\
\end{cases}
\end{align*}
(again, we are identifying $\hat{\epsilon}^{-1}(J_q)$ as a subset of $I_{\gamma_\epsilon(t)}$).

The functor $B$ sends $\epsilon\colon (\{I_t^\pm\}_{t\in T}, \, x)\rightarrow G(J_Q)$ to the object
\begin{align*}
B(\epsilon)\colon (\{I_t^\pm\}_{t\in T}, \, x)\longrightarrow G(I^x_{Q_\epsilon})=(\{\hat{\epsilon}^{-1}(J_q)^\pm\}_{q\in Q_\epsilon},\, x)
\end{align*}
given by $\gamma_\epsilon|_{Q_\epsilon}\colon Q_\epsilon\rightarrow T$ together with the maps $I_{\gamma_\epsilon(t)}^\pm\rightarrow \hat{\epsilon}^{-1}(J_q)^\pm$ defined above (we simply discard any empty $\hat{\epsilon}^{-1}(J_q)$).

The functor $A$, respectively $B$, sends a morphism $\epsilon\rightarrow \delta$ given by $(\theta, \rho)\colon J_Q\rightarrow K_W$ in $\mathcal{J}$ to the morphism given by
\begin{align*}
(\operatorname{id},\rho)\colon I^x_Q\rightarrow I^x_W,\quad \text{respectively}\quad (\operatorname{id},\rho\vert_{W_\delta})\colon I^x_{Q_\epsilon}\rightarrow I^x_{W_\delta},
\end{align*}
in $\mathcal{J}$. Here we use that commutativity of the diagram (\ref{morphism}) implies that $\theta\circ \hat{\epsilon} = \hat{\delta}\colon I^x\rightarrow K$ and that $\rho\vert_{W_\delta}$ factors through $Q_\epsilon$.

Finally, the components of the three natural transformations, $\alpha$, $\beta$, $\gamma$, are given by the morphisms of the induced zigzags in $\mathcal{J}$:
\begin{align*}
\alpha_\epsilon = (\hat{\epsilon}, \operatorname{id}),\quad \beta_\epsilon=(\operatorname{id},\iota_\epsilon),\quad \gamma_\epsilon=(\operatorname{id},\rho_\epsilon).
\end{align*}

We leave it to the reader to verify that the relevant squares commute.
\end{proof}

\begin{remark}
Let $(\{I_t^\pm\}_{t\in T},x)$ as in the proof above. Note that if $x=0$, then $I^x_{P^x}=\emptyset_\emptyset$ and the map
\begin{align*}
\eta\colon (\{I_t^\pm\}_{t\in T}, \,x)\rightarrow G(I^x_{P^x})=(\emptyset,0)
\end{align*}
is given by the unique map of sets $\emptyset\rightarrow T$.
\end{remark}

\begin{remark}\label{eta not initial}
In the proof of \cite[Lemma 5.11]{Yuan}, it is claimed that the object $\eta$ is initial in the right fibre as it is used to construct the unit transformation of an adjunction. This is not generally true, however. Indeed, consider the right fibre over the object $(\{x_3<x_1\}^\pm,x=x_1x_3)$ in $\mathcal{K}$ and the associated object $\eta$ as defined in the proof above:
\begin{align*}
\eta\colon (\{x_3<x_1\}^\pm,x_1x_3)\,\longrightarrow\, G(\{1<3\}_{\{1<3\}})=(\{1^\pm, 3^\pm \}, x=13)
\end{align*}
given by the unique map $\gamma_\eta\colon \{1<3\}\rightarrow \ast$ together with the maps
\begin{align*}
&\eta_1\colon \{1<3\}^\pm \rightarrow 1^\pm, \quad 1\mapsto 1, \ 3\mapsto \top, \\
&\eta_3\colon \{1<3\}^\pm \rightarrow 3^\pm, \quad 1\mapsto \bot, \ 3\mapsto 3,
\end{align*}
where $n^\pm$ denotes the singleton $\{n\}^\pm$. Consider also the object
\begin{align*}
\epsilon\colon (\{x_3<x_1\}^\pm,x_1x_3)\,\longrightarrow\, G(\{1<2<3\}_{\{1<2<3\}})=(\{1^\pm, 2^\pm,  3^\pm \}, y=123)
\end{align*}
given by the unique map $\gamma_\eta\colon \{1<2<3\}\rightarrow \ast$ together with the maps
\begin{align*}
&\epsilon_1\colon \{1<3\}^\pm\rightarrow 1^\pm, \quad 1\mapsto 1, \ 3\mapsto \top, \\
&\epsilon_2\colon \{1<3\}^\pm\rightarrow 2^\pm, \quad 1\mapsto \bot, \ 3\mapsto \top, \\
&\epsilon_3\colon \{1<3\}^\pm\rightarrow 3^\pm, \quad 1\mapsto \bot, \ 3\mapsto 3.
\end{align*}

Then both the maps $(\theta,\rho)\colon \{1<3\}_{\{1<3\}}\rightarrow \{1<2<3\}_{\{1<2<3\}}$ of \Cref{refinement versus morphism} will satisfy $\epsilon=G(\theta,\rho)\circ \eta$, so $\eta$ is not initial. We leave the details to the reader.

We run into the problem that is explained in \Cref{refinement versus morphism}: Given objects $I_P,J_Q$ in $\mathcal{J}$, an order preserving map $I\rightarrow J$ refining the partitions does not necessarily give rise to a canonical map $Q\rightarrow P$. As a consequence, the exhibited left adjoint $H$ in \cite{Yuan} is not well-defined on morphisms. The application, however, is a comparison of colimits and for this we do not need the full strength of an adjunction, we just need $G_S$ to be a $\varinjlim$-equivalence. Equivalently, all right fibres must be contractible, a much weaker condition than all right fibres admitting an initial object. Thus to rectify, we simply invoked \Cref{refinement zigzag} to exhibit a zigzag of homotopies contracting $|\mathcal{C}|$ to the point $\eta$. As observed in \Cref{refinement versus morphism}, in the subcategory $\mathcal{I}_\gg$ the map $\rho$ \textit{is} uniquely determined by $\theta$, and in fact the objects $I^x_{P^x}$ do belong to $\mathcal{I}_\gg$. The maps $\hat{\epsilon}$ need not be surjective, however, so we cannot restrict to this subcategory. Moreover, we have to work over $\mathcal{J}$ instead of $\mathcal{I}$ because this is where the canonical zigzag lives --- indeed, the object $I_Q$ need not belong to $\mathcal{I}$.
\end{remark}

\begin{observation}\label{JS to Ord}
The functor $\mathcal{J}^S\rightarrow \operatorname{Ord}_\pm^\times$ defined in item (2) of \Cref{J and monoids and Ord} coincides with the composite
\begin{align*}
\mathcal{J}^S\xrightarrow{G_S} \mathcal{K}_S\rightarrow \operatorname{Ord}_\pm^\times
\end{align*}
where the last map is the projection functor $(\{I_t^\pm\}_{t\in T}, x)\mapsto \{I_t^\pm\}_{t\in T}$.
\end{observation}

\subsection{Unravelling the left adjoint}\label{identifying the monoids}

Let $\mathcal{E}$ be a cartesian closed cocomplete $\infty$-category. The technical results of the previous section make it a straightforward task to exhibit a left adjoint to the inclusion
\begin{align*}
\mathbb{B}_{\mathcal{E}}=(\operatorname{Free}^\times)^*\colon \operatorname{Fun}^\times(\mathcal{T}_A,\mathcal{E})\rightarrow \operatorname{Fun}^\times( \operatorname{Ord}_\pm^\times, \mathcal{E})
\end{align*}
in the form of the pointwise left Kan extension along $\operatorname{Free}^\times$.

\begin{proposition}\label{identifying L as left Kan extension}
For a cartesian closed cocomplete $\infty$-category $\mathcal{E}$, the inclusion $\mathbb{B}=\mathbb{B}_\mathcal{E}$ admits a left adjoint
\begin{align*}
\mathbb{L}=\mathbb{L}_\mathcal{E}\colon \operatorname{Fun}^\times(\operatorname{Ord}_\pm^\times,\mathcal{E})\rightarrow \operatorname{Fun}^\times(\mathcal{T}_A, \mathcal{E}).
\end{align*}
Given a simplicial object $X$ in $\mathcal{E}$, the monoid object $\mathbb{L}X$ in $\mathcal{E}$ identifies with the product preserving functor $\mathcal{T}_A \rightarrow \mathcal{E}$ which sends $\operatorname{Free}(S)$ to the colimit
\begin{align*}
\mathop{\operatorname{colim}}_{\mathcal{J}^S}\mathbf{X}_S^\times 
\end{align*}
where $\mathbf{X}_S\colon\mathcal{J}^S\rightarrow \operatorname{Ord}_\pm^\times\xrightarrow{X^\times} \mathcal{E}$, and a map $\operatorname{Free}(T)\rightarrow \operatorname{Free}(S)$ is sent to the map on colimits induced by the map $\mathcal{J}^S\rightarrow \mathcal{J}^T$ as described in \Cref{J and monoids and Ord} item (1).
\end{proposition}
\begin{proof}
Since $\mathbb{B}$ is given by restriction along $\operatorname{Free}^\times$, it suffices to show that the left Kan extension of $X^\times$ along $\operatorname{Free}^\times$ identifies with the functor described in the statement and that this functor moreover preserves products. The left Kan extension of $X^\times$ along $\operatorname{Free}^\times$ evaluates to
\begin{align*}
\operatorname{colim} X^\times\vert_{\mathcal{K}_S}
\end{align*}
on $\operatorname{Free}(S)$ and functoriality is given by the induced maps of comma categories. By \Cref{cofinal}, the functor $G_S\colon \mathcal{J}^S\rightarrow \mathcal{K}_S$ is a $\varinjlim$-equivalence and as a result the left Kan extension of $X^\times$ along $\operatorname{Free}^\times$ evaluates to
\begin{align*}
\mathop{\operatorname{colim}}_{\mathcal{J}^S}\mathbf{X}_S
\end{align*}
on $\operatorname{Free}(S)$ as claimed since $\mathbf{X}_S= X^\times\vert_{\mathcal{K}_S}\circ G_S$ by \Cref{JS to Ord}. It is straightforward to check that for a given map of monoids $\theta\colon \operatorname{Free}(T)\rightarrow \operatorname{Free}(S)$, the following diagram commutes where the left hand vertical map is the one described in \Cref{J and monoids and Ord} item (1) and the right vertical map is the map of comma categories given by post-composition by $\theta$.
\begin{center}
\begin{tikzpicture}
\matrix (m) [matrix of math nodes,row sep=2em,column sep=2em]
{
\mathcal{J}^S & \mathcal{K}_S \\
\mathcal{J}^T & \mathcal{K}_T	\\
};
\path[-stealth]
(m-1-1) edge node[above]{$G_S$} (m-1-2) edge (m-2-1)
(m-2-1) edge node[below]{$G_T$} (m-2-2)
(m-1-2) edge node[right]{$\theta_*$} (m-2-2)
;
\end{tikzpicture}
\end{center}
This gives the claimed functoriality. To conclude that the left adjoint $\mathbb{L}$ to $\mathbb{B}=(\operatorname{Free}^\times)^*$ is indeed given by the pointwise left Kan extension, it only remains to be seen that this functor is product preserving. This is clear, since $\mathcal{E}$ is cartesian closed, so
\begin{align*}
\mathop{\operatorname{colim}}_{\mathcal{J}^S\times \mathcal{J}^T}(\mathbf{X}_S\times\mathbf{X}_T)\xrightarrow{\ \simeq\ } \mathop{\operatorname{colim}}_{\mathcal{J}^S}\mathbf{X}_S\times \mathop{\operatorname{colim}}_{\mathcal{J}^T}\mathbf{X}_T
\end{align*}
for all finite sets $S$ and $T$.
\end{proof}

For clarity, let's write out the resulting monoid as a simplicial object considering only standard face and degeneracy maps.

\begin{corollary}\label{monoid over J}
Let $\mathcal{E}$ be a cartesian closed cocomplete $\infty$-category and let $X$ be a simplicial object in $\mathcal{E}$. The simplicial object $\mathbb{B}\mathbb{L}X$ is given by 
\begin{align*}
(n)^\pm \mapsto \mathop{\operatorname{colim}}_{\mathcal{J}^n}\mathbf{X}_n \simeq \big(\mathop{\operatorname{colim}}_{\mathcal{J}}\mathbf{X}\big)^n
\end{align*}
with face maps induced by the functors $d_i\colon \mathcal{J}^n\rightarrow \mathcal{J}^{n-1}$, $0<i<n$, concatenating the $i$'th and $(i+1)$'st factors:
\begin{align*}
,\qquad (I^1_{P^1},\ldots,I^n_{P^n})\mapsto (I^1_{P^1},\ldots,I^{i-1}_{P^{i-1}},\, I^i_{P^i}I^{i+1}_{P^{i+1}}\, , I^{i+2}_{P^{i+2}},\ldots,I^n_{P^n}),
\end{align*}
and $d_0,d_n\colon \mathcal{J}^n\rightarrow \mathcal{J}^{n-1}$ forgetting the factor $I^1_{P^1}$, respectively $I^n_{P^n}$; and
with degeneracy maps induced by the functors $s_i\colon \mathcal{J}^n\rightarrow \mathcal{J}^{n+1}$, $0\leq i\leq n$ inserting the object $\emptyset_\emptyset$ in the $i+1$'st spot:
\begin{align*}
(I^1_{P^1},\ldots,I^n_{P^n})\mapsto (I^1_{P^1},\ldots,I^i_{P^i},\emptyset_\emptyset,I^{i+1}_{P^{i+1}},\ldots,I^n_{P^n}).
\end{align*}
\end{corollary}
\begin{proof}
This is a case of tracing through the definitions. We leave the details to the reader.
\end{proof}

\begin{remark}
We may in the statement above replace the category $\mathcal{J}$ by any one of the subcategories $\mathcal{I}$, $\mathcal{J}_\gg$ or $\mathcal{I}_\gg$, since all the relevant inclusions are $\varinjlim$-equivalences (\Cref{Jgg in J is cofinal}, \Cref{Igg in Jgg cofinal}).
\end{remark}

\subsection{Simplifying the colimit calculations}\label{picking out non-degenerates}

For some simplicial objects $X$, we can simplify things further by working with subfunctors of the underlying semisimplicial objects. This is where the categories $\mathcal{J}_\gg$ and $\mathcal{I}_\gg$ come into play. The idea here is to pick out the ``non-degenerate'' simplices. We make a general statement and then spell out the main examples of interest.

\begin{proposition}\label{monoid over Igg}
Let $\mathcal{E}$ be a cartesian closed cocomplete $\infty$-category and let $X\colon \operatorname{Ord}_\pm\rightarrow \mathcal{E}$ be a simplicial object in $\mathcal{E}$. Suppose $X$ identifies with the left Kan extension of a semisimplicial object
\begin{align*}
Y\colon \operatorname{Ord}_\pm^\gg\rightarrow \mathcal{E}
\end{align*}
along the inclusion $i\colon \operatorname{Ord}_\pm^\gg\hookrightarrow \operatorname{Ord}_\pm$. Then the simplicial object $\mathbb{B}\mathbb{L}X$ is given by
\begin{align*}
(n)^\pm \mapsto \mathop{\operatorname{colim}}_{\mathcal{I}_\gg^n}\mathbf{Y}_n\simeq \big(\mathop{\operatorname{colim}}_{\mathcal{I}_\gg}\mathbf{Y}^{\times}\big)^n
\end{align*}
with face maps induced by the functors $d_i\colon \mathcal{I}_\gg^n\rightarrow \mathcal{I}_\gg^{n-1}$, $0<i<n$, concatenating the $i$'th and $(i+1)$'st factors:
\begin{align*}
,\qquad (I^1_{P^1},\ldots,I^n_{P^n})\mapsto (I^1_{P^1},\ldots,I^{i-1}_{P^{i-1}},\, I^i_{P^i}I^{i+1}_{P^{i+1}}\, , I^{i+2}_{P^{i+2}},\ldots,I^n_{P^n}),
\end{align*}
and $d_0,d_n\colon \mathcal{I}_\gg^n\rightarrow \mathcal{I}_\gg^{n-1}$ forgetting the factor $I^1_{P^1}$, respectively $I^n_{P^n}$; and with degeneracy maps induced by the functors $s_i\colon \mathcal{I}_\gg^n\rightarrow \mathcal{I}_\gg^{n+1}$, $0\leq i\leq n$, inserting the isolated object $\emptyset_\emptyset$ in the $i+1$'st spot:
\begin{align*}
(I^1_{P^1},\ldots,I^n_{P^n})\mapsto (I^1_{P^1},\ldots,I^i_{P^i},\emptyset_\emptyset,I^{i+1}_{P^{i+1}},\ldots,I^n_{P^n}).
\end{align*}
\end{proposition}
\begin{proof}
By \Cref{left Kan extensions versus maps out of JS}, the map $\mathbf{X}_S^\times\colon \mathcal{J}^S\rightarrow \mathcal{E}$ is the left Kan extension of $\mathbf{Y}_S^\times\colon \mathcal{J}_\gg^S\rightarrow \mathcal{E}$ along the inclusion $\mathcal{J}_\gg^S\hookrightarrow \mathcal{J}^S$, so the colimit of $\mathbf{X}_S$ agrees with the colimit of $\mathbf{Y}_S$ for any $S$. The result then follows from \Cref{monoid over J} and the fact that $\mathcal{I}_\gg\rightarrow \mathcal{J}_\gg$ is a $\varinjlim$-equivalence (\Cref{Igg in Jgg cofinal}).
\end{proof}

Our first example is trivial but useful.

\begin{example}
Let $f\colon D\rightarrow \operatorname{Ord}_\pm^\gg$ be any functor and let $X\colon D\rightarrow \mathcal{E}$ be a functor into a cartesian closed cocomplete $\infty$-category $\mathcal{E}$. By transitivity of left Kan extensions, we can apply the proposition above to
\begin{align*}
\operatorname{Lan}_{i\circ f}X \simeq \operatorname{Lan}_i(\operatorname{Lan}_f X).
\end{align*}
The basic example here is the case of free simplicial objects in $\mathcal{E}$: that is, considering the left Kan extension along the inclusion $(1)^\pm\hookrightarrow \operatorname{Ord}_\pm^\gg\hookrightarrow \operatorname{Ord}_\pm$ (see also \S \ref{Free monoid objects}).
\end{example}

The following observation is a simple recognition principle.

\begin{observation}
Let $\mathcal{E}$ be a cocomplete $\infty$-category. Let $X\colon \operatorname{Ord}_\pm\rightarrow \mathcal{E}$, respectively $Y\colon \operatorname{Ord}_\pm^\gg\rightarrow \mathcal{E}$, be a simplicial, respectively semisimplicial, object and suppose we are given a map $Y\rightarrow X\circ i$ of semisimplicial objects. The comparison map
\begin{align*}
\operatorname{Lan}_i Y\rightarrow \operatorname{Lan}_i (X\circ i)\rightarrow X
\end{align*}
is an equivalence if and only if the following condition is satisfied for every $n\geq 0$: the collection of injections $(m)^\pm \hookrightarrow (n)^\pm$ exhibit $X(n)$ as the coproduct
\begin{align*}
X(n)\xleftarrow{ \ \simeq \ } \coprod_{(m)^\pm \hookrightarrow (n)^\pm}Y(m).
\end{align*}
This is easy to see and can also be found in \cite[Remark 10.1.2.30]{kerodon}.
\end{observation}

\begin{example}\label{examples: picking out the non-degenerate simplices} \ 
\begin{enumerate}
\item Let $X$ be a simplicial set, and let $X^{\operatorname{nd}}$ denote the semisimplicial set given by the non-degenerate simplices of $X$. The map
\begin{align*}
X(n)\xleftarrow{ \ \simeq \ } \coprod_{(m)^\pm \hookrightarrow (n)^\pm}X^{\operatorname{nd}}(m),
\end{align*}
is an equivalence for every $n\geq 0$ in view of the following standard fact given an $n$-simplex $x$, there is a unique injection $\theta\colon (m)^\pm \hookrightarrow (n)^\pm$ and a unique non-degenerate simplex $x'$ such that $x=\theta_*(x')$ (see e.g. \cite[Proposition 1.1.3.4]{kerodon}). Under the inclusion $\operatorname{Set}\subseteq \mathcal{S}\subseteq \operatorname{Cat}_\infty$, we get an equivalence
\begin{align*}
X\xleftarrow{\ \simeq \ } \operatorname{Lan}_i X^{\operatorname{nd}}
\end{align*}
both when viewing the simplicial set as a simplicial \textit{space} and as a simplicial \textit{$\infty$-category}.
\item Let $W$ be a Waldhausen category and consider Waldhausen's $S_\bullet$-construction
\begin{align*}
S:=S_\bullet(W)^\simeq\colon \operatorname{Ord}_\pm\rightarrow \mathcal{S}.
\end{align*}
Consider the semisimplicial space of \textit{flags}: that is,
\begin{align*}
S^{\operatorname{f}}:=S^{\operatorname{f}}_\bullet(W)^\simeq\colon \operatorname{Ord}_\pm^\gg\rightarrow \mathcal{S}
\end{align*}
such that $S^{\operatorname{f}}_n(W)^\simeq\subseteq S_n(W)^\simeq$ is the subspace of sequences of \textit{non-invertible} cofibrations
\begin{align*}
\ast\hookrightarrow X_1\hookrightarrow X_2\hookrightarrow \cdots \hookrightarrow X_n\quad \text{in }W,
\end{align*}
together with choices of subquotients. It is a simple observation that the map
\begin{align*}
S(n)\xleftarrow{ \ \simeq \ } \coprod_{(m)^\pm \hookrightarrow (n)^\pm}S^{\operatorname{f}}(m),
\end{align*}
is an equivalence for every $n\geq 0$: indeed, any filtration has an underlying flag given by composing all invertible maps with the preceding monomorphism and there is a unique injection $\theta\colon (m)^\pm \hookrightarrow (n)^\pm$ such that the induced simplicial degeneracy map sends this underlying flag to a filtration that is (uniquely) isomorphic to the original filtration.\qedhere
\end{enumerate}
\end{example}

\subsection{Partial algebraic K-theory and filtered dimension sequences}

In this section, we consider the partial algebraic K-theory of split exact categories. We will not need this in what follows, but it indicates how our calculational strategy fits into that of \cite{Yuan}.

Let $\mathcal{C}$ be a split exact category and consider Waldhausen's $S_\bullet$-construction: $S:=S_\bullet(\mathcal{C})^{\simeq}$ (\cite{Waldhausen}). In \cite{Yuan}, Yuan considers a Cartesian square
\begin{center}
\begin{tikzpicture}
\matrix (m) [matrix of math nodes,row sep=2em,column sep=3em]
{
\mathcal{F} & \mathcal{I}  \\
\mathcal{Z}_0 & \mathcal{S} \\
};
\path[-stealth]
(m-1-1) edge (m-1-2) edge (m-2-1)
(m-1-2) edge node[right]{$\mathbf{S}$} (m-2-2)
(m-2-1) edge (m-2-2)
;
\end{tikzpicture}
\end{center}
and constructs a functor $\mathbf{S}\colon \mathcal{F}\rightarrow \mathcal{S}$ whose colimit coincides with that of $\mathbf{S} \colon \mathcal{I}\rightarrow \mathcal{S}$ (in our notation, this is the restriction of $\mathbf{S}\colon \mathcal{J}\rightarrow \mathcal{S}$ along $\mathcal{I}\hookrightarrow \mathcal{J}$). Yuan proceeds to consider a better behaved subcategory $\mathcal{F}^{\operatorname{red}}\hookrightarrow \mathcal{F}$ over which the colimit calculation simplifies a great deal (see, however, \Cref{Fred in F not cofinal} below). Yuan considers the special case $\mathcal{C}=\operatorname{Vect}_{\F_p}^{\operatorname{fd}}$, but as observed in \cite{Nardin}, the definitions of $\mathcal{F}$ and $\mathcal{F}^{\operatorname{red}}$ generalise readily to split exact categories.

Now, we claim that we obtain $\mathcal{F}^{\operatorname{red}}$ directly by replacing $\mathbf{S} \colon \mathcal{I}\rightarrow \mathcal{S}$ in the cartesian square above with the functor $\mathbf{S}^{\operatorname{f}} \colon \mathcal{I}_\gg\rightarrow \mathcal{S}$ induced by the subfunctor of flags (\Cref{examples: picking out the non-degenerate simplices}). We make the relevant observations below, but do not go into great detail as this is well explained in \cite{Yuan} and we will in any case not need this for our purposes.

Let us first of all define the categories $\mathcal{F}$ and $\mathcal{F}^{\operatorname{red}}$. We will continue to work over arbitrary finite ordered sets, but it is easy to see that the definitions of \cite{Yuan} and \cite{Nardin} define skeletal subcategories of our definitions below.

\begin{definition} \ 
\begin{enumerate}
\item A \textit{filtered dimension} consists of a finite ordered set $I$ and a map $\mathbf{d}\colon I\rightarrow \pi_0(\mathcal{C}^\simeq)$. In other words, $\mathbf{d}=(d_i)_{i\in I}$ is a finite ordered collection of isomorphism classes of objects $d_i=[X_i]\in \pi_0(\mathcal{C}^\simeq)$. We say that $\mathbf{d}$ is of length $l(\mathbf{d})=|I|$. If we want to specify $I$, we say that $\mathbf{d}$ is an \textit{$I$-indexed filtered dimension}.

\item A \textit{filtered dimension sequence} consists of an object $s\colon I\rightarrow P$ in $\mathcal{I}$ and for every $p\in P$, an $s ^{-1}(p)$-indexed filtered dimension $\mathbf{d}^{(p)}$. We write $D=\langle\mathbf{d}^{(p)}\rangle_{p\in P}$. We say that $D$ is of length $l(D)=\sum_{p\in P}l(\mathbf{d}^{(p)})=|I|$. If we want to specify the object $I_P$, we say that $D$ is an \textit{$I_P$-indexed filtered dimension sequence}.

\item Note that there is a unique $\emptyset_\emptyset$-indexed filtered dimension sequence; we call this the \textit{empty filtered dimension sequence}.

\item We say that a filtered dimension $\mathbf{d}=(d_i)_{i\in I}$ is \textit{reduced} if $d_i\neq 0$ for all $i\in I$. Likewise, a filtered dimension sequence $D=\langle\mathbf{d}^{(p)}\rangle_{p\in P}$ is \textit{reduced} if each $\mathbf{d}^{(p)}$ is reduced. We also say that the empty filtered dimension sequence is reduced.\qedhere
\end{enumerate}
\end{definition}

\begin{remark}
Since $\mathcal{C}$ is split exact, there is a bijection between filtered dimension sequences of length $n$ and elements of $\pi_0(S_n)$. Thus we think of the filtered dimension $\mathbf{d}=(d_1,d_2,\ldots,d_n)$ with $d_i=[X_i]$ as specifying the isomorphism class of the filtration
\begin{align*}
X_1\hookrightarrow X_1\oplus X_2\hookrightarrow \cdots \hookrightarrow X_1\oplus \cdots \oplus X_n
\end{align*}
in $\pi_0(S_n)$. A filtered dimension is reduced exactly if it corresponds to a non-degenerate simplex in $S_\bullet$, i.e. a flag.
\end{remark}

\begin{definition}
Let $\mathcal{F}$ be the category whose objects are filtered dimension sequences and where a morphism
\begin{align*}
D=\langle \mathbf{d}^{(p)}\rangle \longrightarrow E=\langle \mathbf{e}^{(q)}\rangle
\end{align*}
between filtered dimension sequences indexed over $I_P$, respectively $J_Q$, is given by a morphism $(\theta, \rho)\colon I_P\rightarrow J_Q$ in $\mathcal{I}$ such that
\begin{align*}
e_j^{(q)}=\bigoplus_{i\in \theta^{-1}(j)}d_i^{(\rho(q))},\quad \text{for all }q\in Q, \ j\in J_q.
\end{align*}

We let $\mathcal{F}^{\operatorname{red}}\subseteq \mathcal{F}$ denote the full subcategory spanned by the reduced filtered dimension sequences.
\end{definition}

\begin{remark}\label{splitting maps and collapse maps} \ 
\begin{enumerate}
\item Note that the conditions mean that a morphism in $\mathcal{F}^{\operatorname{red}}$ must be induced by a morphism in $\mathcal{I}_\gg$: if $D\rightarrow E$ is a morphism in $\mathcal{F}$ and $j\notin \theta(I)$, then we must have $e_j^{(q)}=0$.
\item Note also that if the monoid $\pi_0(\mathcal{C}^\simeq)$ is zerosumfree and cancellative, then $\mathcal{F}^{\operatorname{red}}$ is a poset, since for any sequence $(d_1,\ldots,d_n)$ of objects in $\mathcal{C}$ and any object $e\in \mathcal{C}$, there is at most one $j$ such that $e\cong d_1\oplus \cdots \oplus d_j$.
\item The non-identity morphisms of $\mathcal{F}^{\operatorname{red}}$ are generated under concatenation of filtered dimension sequences by the following two operations:
\begin{enumerate}
\item \textit{collapse maps}:
\begin{align*}
\langle(d_1,\ldots,d_n)\rangle\rightarrow \langle(d_1,\ldots,d_{j-1},d_j\oplus d_{j+1},d_{j+2},\ldots,d_n)\rangle, \quad 1\leq j< n;
\end{align*}
\item and \textit{splitting maps}:
\begin{equation*}
\langle(d_1,\ldots,d_n)\rangle\rightarrow \langle(d_1,\ldots,d_j),(d_{j+1},\ldots,d_n)\rangle, \quad 1\leq j< n.\qedhere
\end{equation*}
\end{enumerate}
\end{enumerate}
\end{remark}

Following \cite{Yuan}, we consider a diagram of Cartesian squares as below, where the $\mathcal{Z}\rightarrow \mathcal{S}$ is the universal left fibration, $\mathcal{Z}_0\rightarrow \mathcal{S}$ is the left fibration classifying the functor $\pi_0\colon \mathcal{S}\rightarrow \mathcal{S}$, and $p$ is the map of fibrations induced by the natural transformation $\operatorname{id}_S\Rightarrow \pi_0$.
\begin{center}
\begin{tikzpicture}
\matrix (m) [matrix of math nodes,row sep=2em,column sep=3em]
{
\mathcal{Y}_{\,} & \mathcal{Y}_0 & \mathcal{I}_\gg \\
\mathcal{Z}_{\,} & \mathcal{Z}_0 & \mathcal{S}_{\,} \\
};
\path[-stealth]
(m-1-1) edge (m-1-2) edge (m-2-1)
(m-1-2) edge (m-2-2) edge (m-1-3)
(m-2-1) edge node[below]{$p$} (m-2-2)
(m-1-3) edge node[right]{$\mathbf{S}^{\operatorname{f}}$} (m-2-3)
(m-2-2) edge (m-2-3)
;
\end{tikzpicture}
\end{center}

Then $\mathcal{Y}_0$ recovers the $1$-categorical Grothendieck construction of the functor
\begin{align*}
\mathcal{I}_\gg\rightarrow \operatorname{Set},\qquad I_P\mapsto \pi_0 (\mathbf{S}^{\operatorname{f}}(I_P)).
\end{align*}
It is not hard to see that this is equivalent to the category $\mathcal{F}^{\operatorname{red}}$ defined above: by definition, the objects of $\mathcal{Y}_0$ are pairs $\big(I_P, [x]\in \pi_0 (\mathbf{S}^{\operatorname{f}}(I_P))\big)$, which exactly correspond to reduced filtered dimension sequences since $\mathcal{C}$ is split exact. A morphism $(I_P,[x])\rightarrow (J_Q,[y])$ is a map $\kappa\colon I_P\rightarrow J_Q$ in $\mathcal{I}_\gg$ such that $\mathbf{S}^{\operatorname{f}}(\kappa)(x)\cong y$; this is readily translated into the data of a morphism of reduced filtered dimension sequences.

Now, the map $p\colon \mathcal{Z}\rightarrow \mathcal{Z}_0$ is itself a left fibration: it classifies the functor $\mathcal{Z}_0\rightarrow \mathcal{S}$ which sends a pair $(X, a\in \pi_0X)$ to the component $X_a\subseteq X$ corresponding to $a$. It follows that the map $\mathcal{Y}\rightarrow \mathcal{F}^{\operatorname{red}}$ is a left fibration and it classifies the functor $\mathbf{S}\colon \mathcal{F}^{\operatorname{red}}\rightarrow \mathcal{S}$ which sends an $I_P$-indexed reduced filtered dimension sequence $D=\langle \mathbf{d}^{(p)}\rangle$ to the component
\begin{align*}
\mathbf{S}(D) \subseteq \mathbf{S}^{\operatorname{f}}(I_P) 
\end{align*}
corresponding to the collection of split flags given by the filtered dimensions $\mathbf{d}^{(p)}$, $p\in P$.

Applying \cite[3.3.4.6]{LurieHTT} twice, we conclude that
\begin{align*}
\mathop{\operatorname{colim}}_{I_P\in \mathcal{I}_\gg} \mathbf{S}^{\operatorname{f}}(I_P)\simeq \mathop{\operatorname{colim}}_{D\in \mathcal{F}^{\operatorname{red}}} \mathbf{S}(D),
\end{align*}
and by \Cref{monoid over Igg}, this identifies the underlying space of the monoid $K^\partial(\mathcal{C})$, recovering the identification of \cite[\S 5.2]{Yuan} which is used in the subsequent calculation $H_*(K^\partial(\operatorname{Vect}_{\F_p}^{\operatorname{fd}}, \F_p))=0$, $*>0$ (\cite[Theorem 5.2]{Yuan}).

\begin{remark}\label{Fred in F not cofinal}
Note that there is a minor error in the proof of the identification above given in \cite{Yuan}. As remarked, applying the arguments above to the functor $\mathbf{S}\colon \mathcal{I}\rightarrow \mathcal{S}$ identifies the underlying space of the monoid $K^\partial(\mathcal{C})$ as a colimit over the category $\mathcal{F}$:
\begin{align*}
\mathop{\operatorname{colim}}_{D\in \mathcal{F}} \mathbf{S}(D).
\end{align*}
To restrict to $\mathcal{F}^{\operatorname{red}}$,  \cite[Lemma 5.14]{Yuan} then claims that the inclusion $\mathcal{F}^{\operatorname{red}}\hookrightarrow \mathcal{F}$ is a $\varinjlim$-equivalence. This is not, however, quite true as the right fibre over the $\ast_\ast$-indexed filtered dimension sequence $\langle (0)\rangle$ is empty. Indeed, consider a map
\begin{align*}
\langle\mathbf{d}^{(p)}\rangle =\langle(0)\rangle \longrightarrow \langle \mathbf{e}^{(q)}\rangle
\end{align*}
in $\mathcal{F}$ from $\langle (0)\rangle$ to some $J_Q$-indexed filtered dimension sequence. This is given by a morphism $(\theta, \rho) \colon I_P=\ast_\ast\rightarrow J_Q$ in $\mathcal{I}$ and by definition, we must then have
\begin{align*}
e_j^{(q)}=\bigoplus_{i\in \theta^{-1}(j)} d_i^{(\rho(q))} = 0 \quad \text{for all }q\in Q,\ j\in J_q.
\end{align*}
In particular, there are no such maps to a \textit{reduced} filtered dimension sequence.

This can be rectified as we have done above by taking a roundabout way through the larger categories $\mathcal{J}$ and $\mathcal{J}_\gg$ where we have some more flexibility in allowing empty sets to show up in our partitions.
\end{remark}

\section{Aside: Free and non-unital monoid objects}\label{Free monoid objects}

We recall free and non-unital monoid objects.

\subsection{Free monoid objects}

Let $\mathcal{C}$ be an $\infty$-category. Evaluation at the object $(1)^\pm$ in $\operatorname{Ord}_\pm$ defines a forgetful functor $u\colon \operatorname{Mon}(\mathcal{C})\rightarrow \mathcal{C}$ which sends a monoid object to its underlying object. This admits a left adjoint $\operatorname{Free}\colon \mathcal{C}\rightarrow \operatorname{Mon}(\mathcal{C})$ (\cite[Example 3.1.3.6]{LurieHA}).

\begin{definition}
We say that a monoid object in $\mathcal{C}$ is \textit{free} if it belongs to the essential image of $\operatorname{Free}$.
\end{definition}

Let $\mathcal{E}$ be a cartesian closed cocomplete $\infty$-category.We have a diagram of adjunctions:

\begin{center}
\begin{tikzpicture}
\matrix (m) [matrix of math nodes,row sep=3em,column sep=3em,nodes={anchor=center}]
{
\operatorname{Mon}(\mathcal{E}) & \operatorname{Fun}(\operatorname{Ord}_\pm,\mathcal{E}) \\
\mathcal{E} & \\
};
\path[-stealth]
(m-1-1) edge node[right]{$u$} node[left]{$\dashv\!$} (m-2-1)
(m-1-2) edge node[above]{$\scriptstyle\operatorname{ev}_1$} node[below left,rotate=120]{$\dashv\,$} (m-2-1)
(m-2-1) edge[bend left] node[left]{$\operatorname{Free}$} (m-1-1)
(m-2-1.1) edge[bend right] node[below right]{$\scriptstyle\operatorname{Lan}_1$} (m-1-2.220)
(m-1-2.170) edge[bend right] node[above]{$\scriptstyle\mathbb{L}$} (m-1-1.15)
;
\path[right hook-stealth]
(m-1-1) edge node[below]{$\scriptstyle\mathbb{B}$} node[right, rotate=90]{$\!\vdash$} (m-1-2)
;
\end{tikzpicture}
\end{center}

We have the following well-known explicit formula for free monoid objects: informally, a free monoid object can be interpreted as the monoid of \textit{words of arbitrary length}.

\begin{lemma}\label{free monoid explicit}
Let $\mathcal{E}$ be a cartesian closed cocomplete $\infty$-category and let $X$ be an object in $\mathcal{E}$. The free monoid object $\operatorname{Free}(X)$ is given by the simplicial object
\begin{align*}
(n)^\pm\mapsto \coprod_{k_1,\ldots,k_n\geq 0} X^{k_1}\times \cdots \times X^{k_n}
\end{align*}
with face maps given by concatenation $X^k\times X^m\rightarrow X^{k+m}$ or forgetting the first or last factor, and degeneracy maps given by inserting a $k=0$ factor $X^0=\ast$.
\end{lemma}
\begin{proof}
As observed in \Cref{examples: picking out the non-degenerate simplices}, we can apply \Cref{monoid over Igg}, since
\begin{align*}
\operatorname{Lan}_1(X)\simeq \operatorname{Lan}_i(X_\gg)
\end{align*}
where $i\colon \operatorname{Ord}_\pm^\gg\hookrightarrow \operatorname{Ord}_\pm$ is the inclusion of the subcategory on surjective maps and where $X_\gg$ denotes the left Kan extension along the inclusion $(1)^\pm \hookrightarrow \operatorname{Ord}_\pm^\gg$.

The semisimplicial object $X_\gg$ is given by
\begin{align*}
(0)^\pm \mapsto \ast\sqcup\ast,\qquad (1)^\pm \mapsto X, \qquad \text{and}\quad (n)^\pm \mapsto \emptyset, \ n>1.
\end{align*}
It follows that the resulting functor $\mathbf{X}_\gg^\times\colon \mathcal{I}_\gg\rightarrow \mathcal{E}$ described in \Cref{simplicial space out of J-category} is given by
\begin{align*}
(s\colon I\rightarrow P) \mapsto \begin{cases}
X^P & \text{if } s \text{ is an isomorphism} \\
\emptyset & \text{else}.
\end{cases}
\end{align*}
Let $\operatorname{Ord}^{\simeq}$ denote the category of finite linearly ordered sets with isomorphisms between them and consider the functor
\begin{align*}
\operatorname{Ord}^{\simeq}\rightarrow \mathcal{I}_\gg,\quad P\mapsto P_P.
\end{align*}
It is easily verified that $\mathbf{X}_\gg^\times$ is the left Kan extension of $\operatorname{Ord}^{\simeq}\rightarrow \mathcal{E}$, $P\mapsto X^P$, along the functor above, and likewise for the $n$-fold product of $\mathbf{X}_\gg^\times$. Hence by \Cref{monoid over Igg}, we see that $\operatorname{Free}(X)\simeq \mathbb{L}\operatorname{Lan}_1(X)$ is of the claimed form:
\begin{align*}
(n)^\pm \mapsto \mathop{\operatorname{colim}}_{\mathcal{I}_\gg^n} (\mathbf{X}_\gg^\times)^n \simeq \mathop{\operatorname{colim}}_{(P_1,\ldots,P_n)\in (\operatorname{Ord}^{\simeq})^n} X^{P_1}\times \cdots \times X^{P_n}\simeq \coprod_{k_1,\ldots,k_n\geq 0} X^{k_1}\times \cdots \times X^{k_n}.
\end{align*}
The structure maps are readily identified from the description in \Cref{monoid over Igg}.
\end{proof}

We will be interested in monoid objects in $\mathcal{S}$ and $\operatorname{Cat}_\infty$. We record the established adjunctions in the following diagram (we only depict the ones we'll be using explicitly).

\begin{center}
\begin{tikzpicture}
\matrix (m) [matrix of math nodes,row sep=4em,column sep=3em,nodes={anchor=center}]
{
\operatorname{Cat} & \operatorname{Cat}_\infty & \mathcal{S} \\
\operatorname{Mon}(\operatorname{Cat}) & \operatorname{Mon}(\operatorname{Cat}_\infty) & \operatorname{Mon}(\mathcal{S}) \\
\operatorname{Fun}(\Delta^{\operatorname{op}},\operatorname{Cat}) & \operatorname{Fun}(\Delta^{\operatorname{op}},\operatorname{Cat}_\infty) & \operatorname{Fun}(\Delta^{\operatorname{op}},\mathcal{S}) \\
};
\path[-stealth]
(m-2-1) edge node[left]{$u$} node[right]{$\!\vdash$} (m-1-1)
(m-2-2) edge node[left]{$u$} node[right]{$\!\vdash$} (m-1-2)
(m-2-3) edge node[left]{$u$} node[right]{$\!\vdash$} (m-1-3)
;
\path[right hook-stealth]
(m-1-1) edge (m-1-2)
(m-2-1) edge (m-2-2)
(m-3-1) edge (m-3-2)
(m-2-1) edge node[left]{$\mathbb{B}$} (m-3-1)
(m-2-2) edge node[left]{$\mathbb{B}$} node[right]{$\!\vdash$} (m-3-2)
(m-2-3) edge node[left]{$\mathbb{B}$} node[right]{$\!\vdash$} (m-3-3)
;
\path[left hook-stealth]
(m-1-3) edge node[right, rotate=90]{$\vdash$} (m-1-2)
(m-2-3) edge node[right, rotate=90]{$\!\vdash$} (m-2-2)
(m-3-3) edge node[right, rotate=90]{$\!\vdash$} (m-3-2)
;
\path[-stealth]
(m-1-2.20) edge[bend left] node[above]{$\scriptstyle|-|$} (m-1-3.140)
(m-2-2.10) edge[bend left] node[above]{$\scriptstyle|-|$} (m-2-3.165)
(m-3-2.10) edge[bend left] node[above]{$\scriptstyle|-|$} (m-3-3.165)
(m-1-1) edge[bend left] node[right]{$\operatorname{Free}$} (m-2-1)
(m-1-2) edge[bend left] node[right]{$\operatorname{Free}$} (m-2-2)
(m-3-2) edge[bend right] node[right]{$\mathbb{L}$} (m-2-2)
(m-1-3) edge[bend left] node[right]{$\operatorname{Free}$} (m-2-3)
(m-3-3) edge[bend right] node[right]{$\mathbb{L}$} (m-2-3)
;
\end{tikzpicture}
\end{center}

We make some simple observations.

\begin{observation}\label{observations about free objects}
\ 
\begin{enumerate}
\item By uniqueness of left adjoints, we have
\begin{align*}
\operatorname{Free}(|-|)\simeq |\operatorname{Free}(-)|.
\end{align*}
\item By the explicit formula above, we see that taking free monoidal $1$-, respectively $\infty$-categories, commutes with the inclusion $\operatorname{Cat}\hookrightarrow \operatorname{Cat}_\infty$ (it suffices to observe that the nerve functor giving this inclusion preserves finite products and disjoint unions). In other words, the free monoidal $\infty$-category on a $1$-category $C$ is equivalent to the so-called \textit{strict free monoidal category} on $C$.\qedhere
\end{enumerate}
\end{observation}

\subsection{Non-unital monoid objects}\label{non-unital monoids}

We recall non-unital monoid objects and observe that the identification of the functor $\mathbb{L}$ generalises readily to the non-unital setting.

\begin{definition}
Let $\mathcal{C}$ be an $\infty$-category. A \textit{non-unital monoid object} in $\mathcal{C}$ is a semisimplicial object $X\colon \Delta_i^{\operatorname{op}}\rightarrow \mathcal{C}$ such that for every $n\geq 0$, the morphisms induced by the face maps $[1]\xrightarrow{\cong} \{i-1<i\}\hookrightarrow [n]$, $1\leq i\leq n$, exhibit $X([n])$ as a product
\begin{align*}
X([n])\xrightarrow{\ \simeq\ } X([1])^n.
\end{align*}
We denote by $\operatorname{Mon}^{\operatorname{nu}}(\mathcal{C})\subset \operatorname{Fun}(\Delta^{\operatorname{op}}_i,\mathcal{C})$ the full subcategory of non-unital monoid objects. Given $X\in \operatorname{Mon}^{\operatorname{nu}}(\mathcal{C})$, we call $X([1])$ the \textit{underlying object} of $X$.
\end{definition}

Analogously to the case of monoid objects, non-unital monoid objects in $\mathcal{C}$ are essentially the same as monoid objects over the non-unitary $\mathbb{E}_1$-operad $\mathbb{E}_1^{\operatorname{nu}}$ (see \cite[5.4.4.1]{LurieHA}). More precisely, by \cite[Corollary 4.1.2.14]{LurieHA}, we have an equivalence of $\infty$-categories
\begin{align*}
\operatorname{Mon}^{\operatorname{nu}}_{\mathbb{E}_1}(\mathcal{C})\xrightarrow{\ \simeq \ }\operatorname{Mon}^{\operatorname{nu}}(\mathcal{C}).
\end{align*}

\begin{remark} \ 
\begin{enumerate}
\item By abuse of terminology, we may sometimes refer to the underlying object $X([1])$ as a non-unital monoid object.
\item For $\mathcal{C}=\mathcal{S}$, the $\infty$-category of spaces, a non-unital monoid object is called a \textit{non-unital $\mathbb{E}_1$-space}.
\item For $\mathcal{C}=\operatorname{Cat}_\infty$, the $\infty$-category of small $\infty$-categories, a non-unital monoid object is called a \textit{non-unital monoidal $\infty$-category}.\qedhere
\end{enumerate}
\end{remark}

As in the case of (unital) monoidal $\infty$-categories, we'll be interested in the ones arising from $1$-categories as below.

\begin{example}
A non-unital monoidal $1$-category $M$ defines a non-unital monoidal $\infty$-category $M^\otimes$ via the bar construction:
\begin{align*}
[n]\mapsto M^n
\end{align*}
with the semisimplicial structure maps given by the monoidal structure. This defines a fully faithful functor
\begin{equation*}
\operatorname{Mon}^{\operatorname{nu}}(\operatorname{Cat})\hookrightarrow \operatorname{Mon}^{\operatorname{nu}}(\operatorname{Cat}_\infty).\qedhere
\end{equation*}
\end{example}

We denote the inclusion of monoid objects into the $\infty$-category of semisimplicial objects by
\begin{align*}
\mathbb{B}^{\operatorname{nu}}=\mathbb{B}^{\operatorname{nu}}_\mathcal{C}\colon \operatorname{Mon}^{\operatorname{nu}}(\mathcal{C})\hookrightarrow \operatorname{Fun}(\Delta_i^{\operatorname{op}},\mathcal{C}).
\end{align*}

Restriction along $i\colon \Delta_i^{\operatorname{op}}\hookrightarrow \Delta^{\operatorname{op}}$ defines a forgetful functor from monoid objects to non-unital monoid objects. It is straightforward to verify that left Kan extension along $i\colon \Delta_i^{\operatorname{op}}\hookrightarrow \Delta^{\operatorname{op}}$ defines a left adjoint $(-)^+$ to this, and that at the level of underlying objects it adds a disjoint basepoint: $X^+([1])=X([1])_+=X([1])\sqcup \ast$.

Evaluation at $1$ defines a forgetful functor $u^{\operatorname{nu}}\colon \operatorname{Mon}^{\operatorname{nu}}(\mathcal{C})\rightarrow \mathcal{C}$. This admits a left adjoint $\operatorname{Free}^{\operatorname{nu}}$ (\cite[Example 3.1.3.6]{LurieHA}), and we have $\operatorname{Free}\simeq (-)^+\circ \operatorname{Free}^{\operatorname{nu}}$ by uniqueness of left adjoints. If $\mathcal{C}$ is a presentable $\infty$-category, then $\mathbb{B}$ admits a left adjoint by the Adjoint Functor Theorem (\cite[Corollary 5.5.2.9]{LurieHTT}):
\begin{align*}
\mathbb{L}^{\operatorname{nu}}=\mathbb{L}^{\operatorname{nu}}_\mathcal{C}\colon \operatorname{Fun}(\Delta_i^{\operatorname{op}},\mathcal{C})\rightarrow \operatorname{Mon}^{\operatorname{nu}}(\mathcal{C}).
\end{align*}
Again by uniqueness of left adjoints, we see that $(-)^+\circ\mathbb{L}^{\operatorname{nu}}\simeq \mathbb{L}\circ \operatorname{Lan}_i(-)$. We record these adjunctions in the diagram below.

\begin{center}
\begin{tikzpicture}
\matrix (m) [matrix of math nodes,row sep=2em,column sep=2em,nodes={anchor=center}]
{
\mathcal{C} & \operatorname{Mon}^{\operatorname{nu}}(\mathcal{C}) & \operatorname{Fun}(\Delta_i^{\operatorname{op}},\mathcal{C}) \\
& \operatorname{Mon}(\mathcal{C}) & \operatorname{Fun}(\Delta^{\operatorname{op}},\mathcal{C}) \\
};
\path[-stealth]
(m-1-2) edge node[right, rotate=90]{$\scriptstyle\!\vdash$} (m-1-1)
(m-2-2) edge node[left, rotate=55]{$\scriptstyle\dashv$} (m-1-1)
(m-1-2) edge node[below]{$\scriptstyle\mathbb{B}^{\operatorname{nu}}$} node[right, rotate=90]{$\scriptstyle\vdash$} (m-1-3)
(m-1-2) edge[bend left] node[right]{$\scriptstyle(-)^+$} (m-2-2)
(m-2-2) edge node[above]{$\scriptstyle\mathbb{B}$} node[left,rotate=90]{$\scriptstyle\dashv$} (m-2-3)
(m-1-1) edge[bend left] node[above]{$\scriptstyle \operatorname{Free}^{\operatorname{nu}}$} (m-1-2)
(m-1-1) edge[bend right] node[left]{$\scriptstyle \operatorname{Free}$} (m-2-2)
(m-1-3) edge[bend left] node[right]{$\scriptstyle\operatorname{Lan}_i(-)$} (m-2-3)
(m-1-3) edge[bend right] node[above]{$\scriptstyle\mathbb{L}^{\operatorname{nu}}$} (m-1-2)
(m-2-3) edge[bend left] node[below]{$\scriptstyle\mathbb{L}$} (m-2-2)
(m-2-2) edge node[right]{$\scriptstyle\!\vdash$} (m-1-2)
(m-2-3) edge node[right]{$\scriptstyle\!\vdash$} (m-1-3)
;
\end{tikzpicture}
\end{center}

The above identifications and the work carried out in \S \ref{unravelling} allow us to easily identify the left adjoint $\mathbb{L}^{\operatorname{nu}}$. Consider the category $\mathcal{I}_\gg^\circ=\mathcal{I}_\gg\smallsetminus \emptyset_\emptyset$ obtained by removing the isolated object from $\mathcal{I}_\gg$.

\begin{proposition}\label{non-unital monoid over Igg}
Let $\mathcal{C}$ be a presentable $\infty$-category and let $X$ be a semisimplicial object in $\mathcal{C}$. Then the semisimplicial object $\mathbb{B}^{\operatorname{nu}}\mathbb{L}^{\operatorname{nu}}X$ is given by
\begin{align*}
(n)^\pm \mapsto \mathop{\operatorname{colim}}_{(\mathcal{I}^\circ_\gg)^n}\mathbf{X}_n\simeq \big(\mathop{\operatorname{colim}}_{\mathcal{I}^\circ_\gg}\mathbf{X}^{\times}\big)^n
\end{align*}
where $\mathbf{X}_n$ is as defined in \Cref{simplicial space out of J-category} and where the face maps are induced by the functors $d_i\colon (\mathcal{I}^\circ_\gg)^n\rightarrow (\mathcal{I}^\circ_\gg)^{n-1}$, $0<i<n$, concatenating the $i$'th and $(i+1)$'st factors:
\begin{align*}
,\qquad (I^1_{P^1},\ldots,I^n_{P^n})\mapsto (I^1_{P^1},\ldots,I^{i-1}_{P^{i-1}},\, I^i_{P^i}I^{i+1}_{P^{i+1}}\, , I^{i+2}_{P^{i+2}},\ldots,I^n_{P^n}),
\end{align*}
and $d_0,d_n\colon (\mathcal{I}_\gg^\circ)^n\rightarrow (\mathcal{I}_\gg^\circ)^{n-1}$ forgetting the factor $I^1_{P^1}$, respectively $I^n_{P^n}$.
\end{proposition}
\begin{proof}
Denoting temporarily by $\mathbb{A}(X)$ the non-unital monoid in the statement, it follows from \Cref{monoid over Igg} that we have an equivalence
\begin{align*}
(\mathbb{L}^{\operatorname{nu}}X)^+ \simeq \mathbb{L}(\operatorname{Lan}_i(X))\simeq \mathbb{A}(X)^+,
\end{align*}
which at the level of underlying spaces identifies
\begin{align*}
(\mathbb{L}^{\operatorname{nu}}X)([1])_+\simeq \mathop{\operatorname{colim}}_{\mathcal{I}_\gg^\circ} \mathbf{X} \sqcup \mathop{\operatorname{colim}}_{\emptyset_\emptyset}\ast \simeq (\mathop{\operatorname{colim}}_{\mathcal{I}_\gg^\circ} \mathbf{X})_+.
\end{align*}
As this map preserves the added basepoints, this equivalence restricts to an equivalence
\begin{align*}
\mathbb{L}^{\operatorname{nu}}X\simeq \mathbb{A}(X)
\end{align*}
of non-unital monoid objects as desired.
\end{proof}

\begin{remark}
The arguments of \S \ref{unravelling} should go through for non-unital monoids, identifying a left adjoint $\mathbb{L}^{\operatorname{nu}}$ for any cartesian closed cocomplete $\infty$-category $\mathcal{E}$, and the identification above would be completely general. We refrain from writing this out in detail.
\end{remark}

Applying the same arguments to the equivalence $\operatorname{Free}\simeq (-)^+\circ \operatorname{Free}^{\operatorname{nu}}$ and using the identification of \Cref{free monoid explicit}, we get a concrete description of free \textit{non-unital} monoid objects.

\begin{corollary}\label{free non-unital monoid objects}
Let $\mathcal{E}$ be a cartesian closed cocomplete $\infty$-category and let $X$ be an object in $\mathcal{E}$. The free non-unital monoid object $\operatorname{Free}^{\operatorname{nu}}(X)$ is given by the semisimplicial object
\begin{align*}
(n)^\pm\mapsto \coprod_{k_1,\ldots,k_n> 0} X^{k_1}\times \cdots \times X^{k_n}
\end{align*}
with face maps given by concatenation $X^k\times X^m\rightarrow X^{k+m}$ or forgetting the first or last factor.
\end{corollary}

\section{Reductive Borel--Serre categories}

In \cite{ClausenOrsnesJansen}, the \textit{reductive Borel--Serre categories} were introduced as a model for unstable algebraic K-theory of associative rings. These are explicitly defined $1$-categories encoding splittable flags in a given finitely generated projective module. The reductive Borel--Serre categories assemble to a monoidal category whose geometric realisation is an $\mathbb{E}_1$-space that group completes to the algebraic K-theory space. We recall these objects now. More specifically, we introduce the \textit{monoidal} reductive Borel--Serre category associated to a category with filtrations and we define the \textit{non-monoidal} reductive Borel--Serre categories as the components of these. We also introduce a ``fat'' monoidal Borel--Serre category which gives us some more flexibility when working with the realisations.

\subsection{The monoidal reductive Borel--Serre category}\label{reductive borel serre categories}

We introduce the monoidal and non-monoidal reductive Borel--Serre (or simply $\operatorname{RBS}$) categories and remark on the geometric origins of these objects. These are defined for any \textit{category with filtrations}, a notion that is slightly more general than exact categories. Essentially, a category with filtrations is just a category with a well-behaved collection of short exact sequences, but we importantly do not require it to be additive. Thus one can think of these as ``exact categories without a biproduct''; examples include exact categories but also the category of vector spaces of rank at most $n$ with the usual short exact sequences. The latter example is exactly the reason for us to work in this generality as we'll see in \S \ref{monoidal rank filtration}. We refer to \cite[\S 7.1]{ClausenOrsnesJansen} for details.

Now we introduce the monoidal reductive Borel--Serre categories. These were originally defined in \cite[\S 7]{ClausenOrsnesJansen} where they are called \textit{monoidal categories of flags and associated gradeds}.

\begin{definition}\label{monoidal category of flags and associated gradeds}
Let $\mathcal{C}$ be a category with filtrations. Let $\mathscr{M}_{\operatorname{RBS}}(\mathcal{C})$ be the category whose objects are finite ordered lists of non-zero objects in $\mathcal{C}$ (including the empty list $\emptyset$). A morphism
\begin{align*}
(m_i)_{i\in I}\rightarrow (n_j)_{j\in J}
\end{align*}
consists of the following data:
\begin{itemize}
\item a surjective order preserving map $\theta\colon I\rightarrow J$;
\item for each $j\in J$, a $\theta^{-1}(j)$-indexed flag in $n_j$ with an identification of the associated graded with the (ordered) sublist $(m_i)_{i\in \theta^{-1}(j)}$.
\end{itemize}
Concatenation of lists defines a strict monoidal product on $\mathscr{M}_{\operatorname{RBS}}(\mathcal{C})$ with unit the empty list; we denote this operation by $\circledast$.

The full subcategory $\mathscr{M}_{\operatorname{RBS}}^{\operatorname{nu}}(\mathcal{C})\subseteq \mathscr{M}_{\operatorname{RBS}}(\mathcal{C})$ on non-empty lists is a strict \textit{non-unital} monoidal category. We call $\mathscr{M}_{\operatorname{RBS}}(\mathcal{C})$ the \textit{(unital) monoidal $\operatorname{RBS}$-category} of $\mathcal{C}$, and $\mathscr{M}_{\operatorname{RBS}}^{\operatorname{nu}}(\mathcal{C})$) the \textit{non-unital monoidal $\operatorname{RBS}$-category} of $\mathcal{C}$.
\end{definition}

\begin{remark}
We will not describe composition of morphisms in full detail here, but it relies on a generalisation of the following statement for
vector spaces: for a surjective order preserving map $\theta\colon  I \rightarrow J$ of finite linearly ordered sets,
specifying a $J$-indexed filtration $\{V_j\}_{j\in J}$ of a vector space $V$ together with a $\theta^{-1}(j)$-indexed filtration of each cokernel $V_j/V_{j-1}$ is equivalent to specifying an $I$-indexed filtration of $V$. This observation enables us to ``merge'' flags and this defines composition in $\mathscr{M}_{\operatorname{RBS}}(\mathcal{C})$ (see \cite[Construction 7.14 and Definition 7.7]{ClausenOrsnesJansen}).
\end{remark}

\begin{observation}
Observe that the unital monoidal $\operatorname{RBS}$-category identifies with the unitalisation of the non-unital monoidal $\operatorname{RBS}$-category:
\begin{equation*}
\mathscr{M}_{\operatorname{RBS}}(\mathcal{C})\simeq (\mathscr{M}_{\operatorname{RBS}}^{\operatorname{nu}}(\mathcal{C}))^+.\qedhere
\end{equation*}
\end{observation}

\begin{notation}\label{general RBS-categories}
Let $\mathcal{C}$ be a category with filtrations. For an object $c$ in $\mathcal{C}$, we write
\begin{align*}
\operatorname{RBS}(c)\subseteq \mathscr{M}_{\operatorname{RBS}}^{\operatorname{nu}}(\mathcal{C})
\end{align*}
for the full subcategory spanned by objects admitting a map the one-object list $(c)$. For $c=0$, we replace $(c)$ by the empty list so that $\operatorname{RBS}(0)$ is just the terminal category on the object $\emptyset$. Let $\operatorname{BGL}(c)\subseteq \operatorname{RBS}(c)$ denote the maximal subgroupoid spanned by objects isomorphic to $(c)$.
Finally, write
\begin{align*}
\partial\operatorname{RBS}(c):=\operatorname{RBS}(c)\smallsetminus \operatorname{BGL}(c)
\end{align*}

We call $\operatorname{RBS}(c)$ the \textit{$\operatorname{RBS}$-category of $c$} (in $\mathcal{C}$). We add a subscript if we want to stress the ambient category with filtrations, $\operatorname{RBS}_\mathcal{C}(c)$.
\end{notation}

\begin{observation}\label{decomposition of monoidal RBS}
It is a simple observation that for an \textit{exact} category $\mathcal{C}$, the monoidal $\operatorname{RBS}$-category decomposes as a disjoint union:
\begin{align*}
\mathscr{M}_{\operatorname{RBS}}(\mathcal{C})\simeq \coprod_{c\in C}\operatorname{RBS}(c)
\end{align*}
where $c$ runs through a set $C$ of representatives of isomorphism classes in $\mathcal{C}$. Simply note that an object $(c_1,\ldots,c_n)$ in $\mathscr{M}_{\operatorname{RBS}}(\mathcal{C})$ maps (canonically) to the object given by taking the direct sum of the individual objects, $(c_1\oplus \cdots \oplus c_n)$. Hence, the realisation, $|\mathscr{M}_{\operatorname{RBS}}(\mathcal{C})|$, decomposes into components indexed by (isomorphism classes of) the single object lists and the empty list.

The non-unital monoidal $\operatorname{RBS}$-category decomposes analogously by simply removing the empty list:
\begin{equation*}
\mathscr{M}_{\operatorname{RBS}}^{\operatorname{nu}}(\mathcal{C})\simeq \coprod_{c\in C\smallsetminus 0}\operatorname{RBS}(c). \qedhere
\end{equation*}
\end{observation}

If $\mathcal{C}=\mathcal{P}(R)$ is the exact category of finitely generated projective modules over a ring $R$, then the categories above recover the \textit{reductive Borel--Serre categories} as defined in \cite[\S 5]{ClausenOrsnesJansen} that we now recall.

\begin{notation}
Let $A$ be an associative ring and $M$ a finitely generated projective $A$-module.
\begin{enumerate}
\item A \textit{splittable flag} in $M$ is a chain of inclusions of submodules of $M$,
\begin{align*}
\mathcal{F}= (M_1\subseteq M_2\subseteq \cdots \subseteq M_{d-1}),
\end{align*}
such that each subquotient $M_i/M_{i-1}$ is non-zero and projective (here we set $M_0=0$ and $M_d=M$).
\item The automorphism group $\operatorname{GL}(M)$ acts on the splittable flags by acting on each submodule: for $g\in \operatorname{GL}(M)$ and $\mathcal{F}$ as above,
\begin{align*}
g\mathcal{F}= (gM_1\subseteq gM_2\subseteq \cdots \subseteq gM_{d-1}).
\end{align*}
\item For a splittable flag $\mathcal{F}$ as above, we denote by $\mathcal{P}_{\mathcal{F}}\subset \operatorname{GL}(M)$ the stabiliser of $\mathcal{F}$, i.e. the subgroup of $g$ such that $gM_i=M_i$ for all $i$.
\item For $\mathcal{F}$ as above, we denote by $\mathcal{U}_{\mathcal{F}}\subseteq \mathcal{P}_{\mathcal{F}}$ the subgroup of elements inducing the identity on the associated graded, i.e. on each $M_i/M_{i-1}$.
\item Let $\mathcal{F}=(M_1\subseteq M_2\subseteq \cdots \subseteq M_{d-1})$ and $\mathcal{G}=(N_1\subseteq N_2\subseteq \cdots \subseteq N_{e-1})$ be splittable flags in $M$. We say that $\mathcal{F}$ \textit{refines} $\mathcal{G}$, and write $\mathcal{F}\leq \mathcal{G}$, if the submodules making up $\mathcal{G}$ form a subset of those making up $\mathcal{F}$:
\begin{equation*}
\{N_1,\ldots, N_{e-1}\}\subseteq\{M_1,\ldots,M_{d-1}\}.\qedhere
\end{equation*}
\end{enumerate}
\end{notation}

\begin{definition}\label{reductive borel-serre categories}
Let $A$ be an associative ring and $M$ a finitely generated projective $A$-module. Define a category $\operatorname{RBS}'(M)$ whose objects are splittable flags in $M$ and whose set of morphisms $\mathcal{F}\rightarrow \mathcal{G}$ is
\begin{align*}
\{g\in \operatorname{GL}(M)\mid g\mathcal{F}\leq \mathcal{G}\}/\mathcal{U}_{\mathcal{F}}.
\end{align*}
Composition is given by multiplication in $\operatorname{GL}(M)$ (this is well-defined because $\mathcal{F}\leq \mathcal{G}$ implies $\mathcal{U}_\mathcal{G}\subseteq \mathcal{U}_\mathcal{F}$).
\end{definition}

\begin{observation}\label{compare flag and list version of RBS}
The functor that sends a flag $\mathcal{F}$ to its associated graded $\operatorname{gr}(\mathcal{F})$ defines an equivalence of categories
\begin{align*}
\operatorname{RBS}'(M)\xrightarrow{\simeq} \operatorname{RBS}(M)
\end{align*}
where the left hand side is the category defined above and the right hand side is the category defined in \Cref{general RBS-categories} (see also \cite[Remark 7.18]{ClausenOrsnesJansen}). We will in this paper generally work with the list-incarnation (\Cref{general RBS-categories}) rather than the flag-incarnation of \Cref{reductive borel-serre categories}.

If $\mathcal{C}=\mathcal{F}(R)$ is the exact of finitely generated free modules over $R$, then 
\begin{align*}
\operatorname{RBS}_{\mathcal{F}(R)}(M)\subseteq \operatorname{RBS}(M)\simeq \operatorname{RBS}'(M)
\end{align*}
identifies with the full subcategory spanned by the \textit{free} splittable flags; that is, all submodules \textit{and} quotients must be free.
\end{observation}

\begin{remark} \ 
\begin{enumerate}
\item The motivation behind the chosen terminology is the fact that the category $\operatorname{RBS}'(M)$ is a simple generalisation of the $1$-category that captures the stratified homotopy type of the \textit{reductive Borel--Serre compactification} $\overline{\Gamma\backslash X}{}^{\operatorname{RBS}}$ of a locally symmetric space associated to a neat arithmetic subgroup $\Gamma\subseteq\operatorname{GL}(\Z)$. See \cite{OrsnesJansen} and \cite[\S 4]{ClausenOrsnesJansen}.
\item In line with the geometric origins of these categories, one can interpret $\operatorname{RBS}(M)$ as a ``compactification'' of the full subcategory $\operatorname{BGL}(M)$ (on the one-object list $(M)$). This is also the reason behind the notation $\partial \operatorname{RBS}(M)$ interpreting the complement of $\operatorname{BGL}(M)$ as the \textit{boundary} of $\operatorname{RBS}(M)$.
\item Let $A$ be an associative ring and $\mathcal{P}(A)$ the exact category of finitely generated projective $A$-modules. The decomposition
\begin{align*}
\mathscr{M}_{\operatorname{RBS}}(\mathcal{P}(A))\simeq \coprod_{M\in \mathcal{M}}\operatorname{RBS}(M)
\end{align*}
is analogous to the decomposition of the symmetric monoidal category given by the maximal subgroupoid of $\mathcal{P}(A)$ into a disjoint union of automorphism groups:
\begin{equation*}
\mathcal{P}(A)^\simeq\simeq \coprod_{M\in \mathcal{M}}\operatorname{BGL}(M).\qedhere
\end{equation*}
\end{enumerate}

\end{remark}

\begin{remark}
Note that in \cite[\S 7]{ClausenOrsnesJansen}, we call the above defined monoidal category the \textit{monoidal category of flags and associated gradeds} in $\mathcal{C}$. We opt for a change of terminology because we want to keep in mind the link to the (non-monoidal) reductive Borel--Serre categories. These are the categories that we are ultimately interested in, but we exploit the fact that they assemble to a monoidal category.
\end{remark}

\subsection{The fat monoidal reductive Borel--Serre category}

In this section, we define a ``bigger'' version of the monoidal reductive Borel--Serre category by allowing what one might think of as degenerate objects. At the level of categories this makes it somewhat more difficult to get our hands on, but at the level of realisations it has the effect of simply ``fattening'' up the $\mathbb{E}_1$-space $|\mathscr{M}_{\operatorname{RBS}}(\mathcal{C})|$ providing some flexibility that turns out to be extremely useful.

\begin{definition}\label{monoidal category of filtrations and associated gradeds}
Let $\mathcal{C}$ be a category with filtrations. Let $\accentset{\sim}{\mathscr{M}}_{\operatorname{RBS}}(\mathcal{C})$ be the category whose objects are finite ordered lists of objects in $\mathcal{C}$ (including the empty list $\emptyset$). A morphism
\begin{align*}
(m_i)_{i\in I}\rightarrow (n_j)_{j\in J}
\end{align*}
consists of the following data:
\begin{itemize}
\item an order preserving map $\theta\colon I\rightarrow J$ such that if $j\notin \theta(I)$, then $n_j=0$;
\item for each $j\in \theta(I)$, a $\theta^{-1}(j)$-indexed filtration in $n_j$ with an identification of the associated graded with the (ordered) sublist $(m_i)_{i\in \theta^{-1}(j)}$.
\end{itemize}
Concatenation of lists, $\circledast$, defines a strict monoidal product on $\accentset{\sim}{\mathscr{M}}_{\operatorname{RBS}}(\mathcal{C})$ with monoidal unit given by the empty list. We call this the \textit{fat monoidal $\operatorname{RBS}$-category} of $\mathcal{C}$.
\end{definition}

As for the (slim) monoidal $\operatorname{RBS}$-category, we will not describe the composition in full detail, but refer the reader to \cite[Construction 7.14 and Definition 7.7]{ClausenOrsnesJansen}.

The monoidal $\operatorname{RBS}$-category identifies with the full subcategory of $\accentset{\sim}{\mathscr{M}}_{\operatorname{RBS}}(\mathcal{C})$ on lists of \textit{non-zero} objects:
\begin{align*}
\Upsilon \colon \mathscr{M}_{\operatorname{RBS}}(\mathcal{C}) \hookrightarrow \accentset{\sim}{\mathscr{M}}_{\operatorname{RBS}}(\mathcal{C}).
\end{align*}

In view of the following lemma, we can interpret $\accentset{\sim}{\mathscr{M}}_{\operatorname{RBS}}(\mathcal{C})$ as a fattened up version of $\mathscr{M}_{\operatorname{RBS}}(\mathcal{C})$; hence, the name.

\begin{lemma}\label{fluffed up model of monoidal category of flags and associated gradeds}
The inclusion $\Upsilon\colon  \mathscr{M}_{\operatorname{RBS}}(\mathcal{C}) \hookrightarrow \accentset{\sim}{\mathscr{M}}_{\operatorname{RBS}}(\mathcal{C})$ admits a right adjoint. In particular, after realisation, we have an equivalence of $\mathbb{E}_1$-spaces
\begin{align*}
|\mathscr{M}_{\operatorname{RBS}}(\mathcal{C})| \xrightarrow{\ \simeq \ } |\accentset{\sim}{\mathscr{M}}_{\operatorname{RBS}}(\mathcal{C})|.
\end{align*}
\end{lemma}
\begin{proof}
The right adjoint $\Gamma\colon \accentset{\sim}{\mathscr{M}}_{\operatorname{RBS}}(\mathcal{C})\rightarrow \mathscr{M}_{\operatorname{RBS}}(\mathcal{C})$ is given by removing all zero objects. For a morphism
\begin{align*}
\phi=\mathop{\circledast}_{j\in J}\phi_j\colon (c_i)_{i\in I}\rightarrow (d_j)_{j\in J}
\end{align*}
given by $\rho\colon I\rightarrow J$ and a filtration in each $d_j$ with associated graded $(c_i)_{i\in \rho^{-1}(j)}$; the resulting morphism $\Gamma(\phi)$ is defined by removing any $\phi_j$ for which $d_j=0$ (this makes sense since all subquotients $c_i$, $i\in \rho^{-1}(j)$, must be zero) and taking the underlying flag of each filtration in $d_j\neq 0$ with associated graded the non-zero $c_i$, $i\in \rho^{-1}(j)$.

Then $\Gamma\circ \Upsilon$ is the identity and the counit transformation is given by the unique maps
\begin{align*}
\epsilon \colon (c_i)_{i\in I'}\rightarrow (c_i)_{i\in I}
\end{align*}
given by the inclusion $I'\hookrightarrow I$, where $I'=\{i\in I\mid c_i\neq 0\}\subset I$. We leave it to the reader to check the relevant commutativity.
\end{proof}

\begin{remark} \ 
\begin{enumerate}
\item We could also have defined a non-unital fat monoidal $\operatorname{RBS}$-category by removing the empty list, but note that the realisation of this would be a \textit{quasi-unital} $\mathbb{E}_1$-space: any list of zeros $(0,\ldots,0)$ would be a quasi-unit (\cite[Definition 5.4.3.1]{LurieHA}). Hence, it can be extended to a unital $\mathbb{E}_1$-space in an essentially unique way in view of \cite[Theorem 5.4.3.5]{LurieHA}, so such a definition would not make much difference for our purposes.
\item The difference between $\mathscr{M}_{\operatorname{RBS}}(\mathcal{C}) $ and $\accentset{\sim}{\mathscr{M}}_{\operatorname{RBS}}(\mathcal{C})$ should reflect the difference between the categories $\mathcal{I}_\gg$ and $\mathcal{J}$ of \S \ref{partitioned linearly ordered sets}. We are allowing ``degenerate'' objects to appear: in $\mathcal{J}$ we allow empty sets, in $\accentset{\sim}{\mathscr{M}}_{\operatorname{RBS}}(\mathcal{C})$ we allow zeros. This makes the categories slightly more difficult to work with simply because they are much bigger, but it provides some room for maneuvre which is not a priori easy to see in the ``smaller'' models.
\item Suppose $(m_i)_{i\in I}\rightarrow (n_j)_{j\in J}$ is a morphism in $\accentset{\sim}{\mathscr{M}}_{\operatorname{RBS}}(\mathcal{C})$ given by $\theta\colon I\rightarrow J$. If $n_j=0$, then we must have $m_i=0$ for all $i\in \theta^{-1}(j)$. Moreover, by definition, we have that an object $n_j$ can only by ``skipped'' if it is zero. We already used this in the proof above, but highlight it here to underline that we actually have reasonable control over the added ``non-degeneracy''.\qedhere
\end{enumerate}
\end{remark}

The first upshot of this model is the following functoriality.

\begin{observation}
Let $\phi\colon \mathcal{C}\rightarrow \mathcal{D}$ be a morphism of categories with filtrations, i.e. a functor preserving the short exact sequences. There is an induced morphism of monoidal categories
\begin{align*}
\tilde{\phi}\colon \accentset{\sim}{\mathscr{M}}_{\operatorname{RBS}}(\mathcal{C})\rightarrow \accentset{\sim}{\mathscr{M}}_{\operatorname{RBS}}(\mathcal{D}), \quad (c_i)_{i\in I}\mapsto (\phi(c_i))_{i\in I}.
\end{align*}
On morphisms we likewise apply $\phi$ to the given filtrations and to the identifications of associated gradeds. Thus the fat monoidal $\operatorname{RBS}$-categories define a functor
\begin{align*}
\accentset{\sim}{\mathscr{M}}_{\operatorname{RBS}}\colon \operatorname{Cat}^{\operatorname{filtr}}\rightarrow \operatorname{Mon}(\operatorname{Cat})
\end{align*}
from (essentially small) categories with filtrations to monoidal categories. Applying geometric realisation to this yields a functor
\begin{align*}
\operatorname{Cat}^{\operatorname{filtr}}\rightarrow \operatorname{Mon}(\mathcal{S}),\quad \mathcal{C}\mapsto |\mathscr{M}_{\operatorname{RBS}}(\mathcal{C})|
\end{align*}
by \Cref{fluffed up model of monoidal category of flags and associated gradeds}. Note that the assignment $\mathcal{C}\mapsto \mathscr{M}_{\operatorname{RBS}}(\mathcal{C})$ is \textit{not} functorial at the level of monoidal categories, since an exact functor $\mathcal{C}\rightarrow \mathcal{D}$ may send non-zero objects to zero.
\end{observation}

\section{Partial algebraic K-theory and \texorpdfstring{$\operatorname{RBS}$}{RBS}-categories}\label{partial k-theory}

As was already mentioned, the geometric realisation of the \textit{monoidal $\operatorname{RBS}$-category} is an $\mathbb{E}_1$-space that shares some notable properties with Yuan's partial algebraic K-theory. In this section, we show that these are in fact equivalent $\mathbb{E}_1$-spaces. As a consequence, we have a universal property on the one hand and an explicit tangible monoidal $1$-category on the other.

\subsection{Notation and proof strategy}\label{monoids in question}

Let us very briefly recall the $\mathbb{E}_1$-monoids under consideration and introduce the notation necessary for what's to come. Let $\mathcal{C}$ be an exact category. Recall that the partial algebraic K-theory of $\mathcal{C}$ is the $\mathbb{E}_1$-space
\begin{align*}
K^\partial(\mathcal{C}):=\mathbb{L}(S_\bullet(\mathcal{C})^{\simeq})
\end{align*}
where $S_\bullet(\mathcal{C})$ is Waldhausen's ($1$-categorical) $S_\bullet$-construction and $\mathbb{L}\colon \operatorname{Fun}(\Delta^{\operatorname{op}},\mathcal{S})\rightarrow \operatorname{Mon}(\mathcal{S})$ is left adjoint to the inclusion. Explicitly, $S_\bullet(\mathcal{C})^{\simeq}$ is the simplicial groupoid whose category of $n$-simplices has as objects diagrams in $\mathcal{C}$ as pictured below satisfying that for all $s<t<r$, the sequence $x_{st}\rightarrowtail x_{sr}\twoheadrightarrow x_{tr}$ is exact. The morphisms are just isomorphisms of such diagrams (\cite{Waldhausen}).

\begin{center}
\begin{equation}\label{element in Sdot}
\begin{tikzpicture}[baseline=(current  bounding  box.center)]
\matrix (m) [matrix of math nodes,row sep=1em,column sep=2em]
{
x_{01}  & x_{02} & x_{03} & \cdots\vphantom{_*} & x_{0(n-1)} & x_{0n} \\
		& x_{12} & x_{13} & \cdots\vphantom{_*} & x_{1(n-1)} & x_{1n} \\
		&		 & x_{23} & \cdots\vphantom{_*} & x_{2(n-1)} & x_{2n} \\
& & & \ddots & \vdots & \vdots \\
& & & & x_{(n-2)(n-1)} & x_{(n-1)n} \\
& & & & & x_{(n-1)n} \\
};
\path[>->]
(m-1-1) edge (m-1-2)
(m-1-2) edge (m-1-3)
(m-1-3) edge (m-1-4)
(m-1-4) edge (m-1-5)
(m-1-5) edge (m-1-6)
(m-2-2) edge (m-2-3)
(m-2-3) edge (m-2-4)
(m-2-4) edge (m-2-5)
(m-2-5) edge (m-2-6)
(m-3-3) edge (m-3-4)
(m-3-4) edge (m-3-5)
(m-3-5) edge (m-3-6)
(m-5-5) edge (m-5-6)
;
\path[->>]
(m-1-2) edge (m-2-2)
(m-1-3) edge (m-2-3)
(m-1-5) edge (m-2-5)
(m-1-6) edge (m-2-6)
(m-2-3) edge (m-3-3)
(m-2-5) edge (m-3-5)
(m-2-6) edge (m-3-6)
(m-3-5) edge (m-4-5)
(m-3-6) edge (m-4-6)
(m-4-5) edge (m-5-5)
(m-4-6) edge (m-5-6)
(m-5-6) edge (m-6-6)
;
\end{tikzpicture}
\end{equation}
\end{center}

Write $S:=S_\bullet(\mathcal{C})^\simeq$ and let $\mathbf{S}\colon \mathcal{J}\rightarrow \mathcal{S}$ denote the induced functor defined in \Cref{simplicial space out of J-category}:
\begin{align*}
I_P\mapsto \prod_{p\in P} S(I_p^\pm).
\end{align*}
Writing $I_p=\{1<\ldots<n_p\}$ for $p\in P$, the groupoid $S(I_p^\pm)$ naturally identifies with $S_{n_p}(\mathcal{C})^\simeq$. By \Cref{monoid over J}, the bar construction of the partial algebraic K-theory space $K^\partial(\mathcal{C})$ identifies with the simplicial space
\begin{align*}
(n)^\pm \mapsto \big(\mathop{\operatorname{colim}}_{\mathcal{J}} \mathbf{S} \big)^n
\end{align*}
with structure maps induced by concatenation in $\mathcal{J}$ and inserting the object $\emptyset_\emptyset$ (specified explicitly in \Cref{monoid over J}). Recall that we may calculate a colimit in $\mathcal{S}$ using the \textit{simplicial replacement}.

\begin{notation}\label{simplicial replacement}
Given a diagram $X\colon I\rightarrow \mathcal{S}$, the colimit of $X$ may be calculated as the realisation of the simplicial space $\operatorname{srep}(X)$ with 
\begin{align*}
\operatorname{srep}(X)_n=\coprod_{i_0\leftarrow \cdots \leftarrow i_n} X(i_n)
\end{align*}
taking the coproduct over sequences of composable maps in $I$. The structure maps of $\operatorname{srep}(X)$ are as follows:
\begin{itemize}
\item For $j<n$, the face map $d_j\colon \operatorname{srep}(X)_n\rightarrow \operatorname{srep}(X)_{n-1}$ sends the component $X(i_n)$ indexed by the sequence $\sigma\colon i_0\leftarrow \cdots \leftarrow i_n$ to the identical component $X(i_n)$ indexed by the sequence $d_j(\sigma)$ given by applying the $j$'th face map in the nerve $N_\bullet(I)$.
\item The last face map $d_n\colon \operatorname{srep}(X)_n\rightarrow \operatorname{srep}(X)_{n-1}$ sends the component $X(i_n)$ indexed by the sequence $\sigma\colon i_0\leftarrow \cdots \leftarrow i_n$ to the component $X(i_{n-1})$ indexed by the sequence $d_n(\sigma)\colon i_0\leftarrow \cdots \leftarrow i_{n-1}$ via the map $X(i_n)\rightarrow X(i_{n-1})$ induced by the discarded map $i_n\rightarrow i_{n-1}$.
\item The degeneracy maps $s_j\colon \operatorname{srep}(X)_n\rightarrow \operatorname{srep}(X)_{n+1}$ send a copy of $X(i_n)$ indexed by a sequence $\sigma\colon i_0\leftarrow \cdots \leftarrow i_n$ to the identical component $X(i_n)$ indexed by the sequence $s_j(\sigma)$ obtained by inserting an identity map $i_j\rightarrow i_j$.
\end{itemize}
We refer to \cite{Dugger} for details.
\end{notation}

We will identify $K^\partial(\mathcal{C})$ with $|\mathscr{M}_{\operatorname{RBS}}(\mathcal{C})|$ by identifying the underlying spaces and verifying that these identifications are compatible with the structure maps of the bar constructions. We will work with the \textit{fat} monoidal $\operatorname{RBS}$-category, $\accentset{\sim}{\mathscr{M}}_{\operatorname{RBS}}(\mathcal{C})$, and we compare partial K-theory with the geometric realisation $|\accentset{\sim}{\mathscr{M}}_{\operatorname{RBS}}(\mathcal{C})|$ by passing through the twisted arrow category $\operatorname{Tw}(\accentset{\sim}{\mathscr{M}}_{\operatorname{RBS}}(\mathcal{C}))$. We've already used the twisted arrow category as our category $\mathcal{J}$ is one such but let's recall it here for good measure.

\begin{notation}
Recall the \textit{twisted arrow category} $\operatorname{Tw}(\mathcal{D})$ of a category $\mathcal{D}$: its objects are morphisms in $\mathcal{D}$ and a morphism from $f\colon x\rightarrow y$ to $g\colon x'\rightarrow y'$ is a pair of maps $a\colon x'\rightarrow x$, $b\colon y\rightarrow y'$ such that $g=bfa$. In other words, a morphism from $f$ to $g$ is a factorisation of $g$ through $f$. Recall that the projection functor
\begin{align*}
\operatorname{Tw}(\mathcal{D})^{\operatorname{op}}\rightarrow \mathcal{D},\quad (x\rightarrow y)\mapsto x
\end{align*}
is a $\varinjlim$-equivalence (\cite[Lemma 4.16]{ClausenOrsnesJansen}). It follows that there is a functorial equivalence $|\operatorname{Tw}(\mathcal{D})|\simeq |\mathcal{D}|$ (\cite[Remark 2.20 and Corollary 2.11]{ClausenOrsnesJansen}).
\end{notation}

The comparison is based on the simple observation that a simplex in the $S_\bullet$-construction gives rise to a morphism in $\accentset{\sim}{\mathscr{M}}_{\operatorname{RBS}}(\mathcal{C})$. More explicitly, an object in $\mathbf{S}(I_P)$ is essentially just an ordered list of filtrations together with a choice of associated gradeds; this is exactly the same as a morphism in $\accentset{\sim}{\mathscr{M}}_{\operatorname{RBS}}(\mathcal{C})$. The source of the resulting morphism is the (ordered list of) associated gradeds (the diagonals of the diagrams as in (\ref{element in Sdot})), the target is the (ordered list of) ambient objects in which the filtrations live (the top right corners of the diagrams as in (\ref{element in Sdot})). This observation is made precise in the following subsection.

With the technical input in place, the final identification becomes straightforward: we get an equivalence of simplicial groupoids,
\begin{align*}
\operatorname{srep}(\mathbf{S})_\bullet \longrightarrow N_\bullet i \operatorname{Tw}(\accentset{\sim}{\mathscr{M}}_{\operatorname{RBS}}(\mathcal{C})),
\end{align*}
where on the left we have the simplicial replacement of the functor $\mathbf{S}\colon \mathcal{J}\rightarrow \mathcal{S}$ (see \Cref{simplicial replacement}) and on the right we have the nerve of the twisted arrow category of the monoidal category enhanced to a simplicial groupoid as explained below.

\begin{notation}\label{maximal subgroupoids}
Given a $1$-category $C$, we can enhance its nerve $N_\bullet C$ to a simplicial groupoid $N_\bullet iC$, by considering for each $n\in \N$, the groupoid of sequences $c_0\rightarrow \cdots \rightarrow c_n$, with isomorphisms given by commutative diagrams
\begin{center}
\begin{tikzpicture}
\matrix (m) [matrix of math nodes,row sep=1em,column sep=1em]
{
c_0 & c_1 & \cdots & c_n \\
d_0 & d_1 & \cdots & d_n	\\
};
\path[-stealth]
(m-1-1) edge (m-1-2) edge node[right]{$\scriptstyle\cong$} (m-2-1)
(m-1-2) edge (m-1-3) edge node[right]{$\scriptstyle\cong$} (m-2-2)
(m-1-3) edge (m-1-4)
(m-1-4)  edge node[right]{$\scriptstyle\cong$} (m-2-4)
(m-2-1) edge (m-2-2)
(m-2-2) edge (m-2-3)
(m-2-3) edge (m-2-4)
;
\end{tikzpicture}
\end{center}
Equivalently, $N_\bullet iC$ is the horizontal nerve of the double category with horizontal morphisms given by morphisms in $C$ and vertical morphisms given by the isomorphisms in $C$. By Waldhausen's swallowing lemma, this does not change the homotopy type of the resulting geometric realisation: $|N_\bullet iC|\simeq |N_\bullet C| \simeq |C|$ (\cite[Lemma 1.6.5]{Waldhausen}).
\end{notation}

\begin{remark}
The identification could equally well be made using the smaller monoidal category $\mathscr{M}_{\operatorname{RBS}}(\mathcal{C})$ by simply restricting to the functor $\mathbf{S}^{\operatorname{f}}\colon \mathcal{I}_\gg\rightarrow \mathcal{S}$ picking out the non-degenerate simplices, i.e. the flags (see \Cref{examples: picking out the non-degenerate simplices}). It turns out that we will need the technical observations in their slightly more general form later on (\S \ref{homotopy commutativity}). Therefore we choose to work with the fluffed up monoidal category. See also \Cref{restricting to flags}.
\end{remark}

\subsection{Technical legwork}

In this section, we lay out the technical ingredients needed for the comparison. We present the relevant observations in a sequence of lemmas to make it clear how the different parts relate. First of all, we translate objects in $\mathbf{S}(I_P)$ into morphisms in the fat monoidal $\operatorname{RBS}$-category, $\accentset{\sim}{\mathscr{M}}_{\operatorname{RBS}}(\mathcal{C})$.

\begin{lemma}\label{morphism in monoidal category}
Given $I_P$ in $\mathcal{J}$, an object $x$ in $\mathbf{S}(I_P)$ defines a canonical morphism
\begin{align*}
\phi_x\colon \mathfrak{s}(x)\rightarrow \mathfrak{t}(x)
\end{align*}
in $\accentset{\sim}{\mathscr{M}}_{\operatorname{RBS}}(\mathcal{C})$. Conversely, for any morphism $\phi$ in $\accentset{\sim}{\mathscr{M}}_{\operatorname{RBS}}(\mathcal{C})$, there is a unique $I_P$ in $\mathcal{J}$ and an $x$ in $\mathbf{S}(I_P)$ such that $\phi=\phi_x$. Moreover, $x$ is unique up to unique isomorphism in $\mathbf{S}(I_P)$.
\end{lemma}
\begin{proof}
If $I=\emptyset$, then $x=0$ and $\phi_x$ is the unique map
\begin{align*}
\phi_x\colon \emptyset \rightarrow (0)_{p\in P}
\end{align*}
given by the map $\emptyset\rightarrow P$. Assume that $I\neq \emptyset$. We consider first the case $P=\ast$ and write $I=\{1<\cdots <n\}$. Then the element $x$ is a diagram as in (\ref{element in Sdot}), and the resulting morphism
\begin{align*}
\phi_x\colon \mathfrak{s}(x):=(x_{01},x_{12},\ldots,x_{(n-1)n})\longrightarrow \mathfrak{t}(x):=(x_{0n})
\end{align*}
is given by the filtration
\begin{align*}
x_{01} \hookrightarrow x_{02} \hookrightarrow x_{03} \hookrightarrow \cdots \hookrightarrow x_{0(n-1)} \hookrightarrow x_{0n}
\end{align*}
defined by the upper horizontal sequence of this diagram together with the given identification of the associated graded as the sequence of elements lying on the diagonal.

For general $P$, the element $x=(x^p)_{p\in P}$ is a list of diagrams as above (or $0$) and the morphism $\phi_x\colon \mathfrak{s}(x)\rightarrow \mathfrak{t}(x)$ is simply the concatenation of the maps $\phi_{x^p}$ described above:
\begin{align*}
\phi_x:=\mathop{\circledast}_{p\in P} \phi_{x^p}\colon \mathop{\circledast}_{p\in P} \big(x^p_{01}, x^p_{12}, \ldots, x^p_{(n_p-1)n_p}\big)\longrightarrow \mathop{\circledast}_{p\in P} (x^p_{0n_p})
\end{align*}
where we have identified $I_p=\{1<\ldots<n_p\}$.

Suppose conversely that we have a morphism $\phi\colon (a_i)_{i\in I}\rightarrow (b_p)_{p\in P}$ in $\accentset{\sim}{\mathscr{M}}_{\operatorname{RBS}}(\mathcal{C})$. This is given by an order preserving map $\rho\colon I\rightarrow P$ and for each $p\in P$, the data of a $\rho^{-1}(p)$-indexed filtration $\mathcal{F}^p =(f^p_i)_{i\in \rho^{-1}(p)}$ in $b_p$ together with morphisms $f^p_i\twoheadrightarrow a_i$, $i\in \rho^{-1}(p)$, identifying $(a_i)_{i\in \rho^{-1}(p)}$ with the associated graded of $\mathcal{F}^p$. The map $\rho\colon I \rightarrow P$ defines an object $I_P$ in $\mathcal{J}$ and there is an object
\begin{align*}
x_\phi=(x_\phi^p)_{p\in P}\quad\text{ in }\quad\mathbf{S}(I_P)=\prod_{p\in P}S_{n_p}
\end{align*}
whose $p$'th factor $x_\phi^p$ has as its upper horizontal line (a representative of) the filtration $\mathcal{F}^p$ and whose diagonal elements are $x^p_{(s-1)s}=a_s$ for all $s$. The object $I_P$ is uniquely determined as it is simply part of the data, and the element $x_\phi$ is unique up to unique isomorphism in $\mathbf{S}(I_P)$ as we are just filling in the missing subquotients.
\end{proof}

We now observe what happens to isomorphisms under the association $x\mapsto \phi_x$ described above.

\begin{observation}\label{isomorphisms to isomorphisms}

An isomorphism $x\rightarrow y$ of diagrams as in (\ref{element in Sdot}) is uniquely determined by the isomorphism $x_{0n}\rightarrow y_{0n}$ of the upper right corner elements. Analogously, for general objects $x,y$ in $\mathbf{S}(I_P)$, a morphism $\beta\colon \phi_y\rightarrow \phi_x$ is uniquely determined by the map $\beta_t\colon \mathfrak{t}(y)\rightarrow \mathfrak{t}(x)$. From this we conclude that we have a bijection
\begin{equation*}
\{x\xrightarrow{\cong}y \text{ in } \mathbf{S}(I_P)\} \rightleftarrows \{\phi_y\xrightarrow{\cong} \phi_x \text{ in } \operatorname{Tw}(\accentset{\sim}{\mathscr{M}}_{\operatorname{RBS}}(\mathcal{C}))\}
\end{equation*}
and that this association, $\alpha\mapsto \bar{\alpha}$, is functorial.
\end{observation}

We move on to translating the data of a morphism in $\mathcal{J}$ together with an object in the $S_\bullet$-construction into a morphism in the twisted arrow category. We do this in steps: we treat the case of collapse and splitting maps in $\mathcal{J}$ separately as this better reveals the finer details of the association (and makes for cleaner technical arguments). It will in both cases be a simple application of the following technical observation.

\begin{lemma}\label{factorising a morphism in MC}
Suppose given two order preserving maps $I\xrightarrow{\theta} J\xrightarrow{\rho} K$ and consider the following morphisms in $\mathcal{J}$:
\begin{align*}
\kappa:=(\theta, \operatorname{id}_K)\colon I_K\rightarrow J_K, \quad\text{and}\quad \sigma:=(\operatorname{id}_I, \rho)\colon I_K\rightarrow I_J.
\end{align*}
Let $x$ in $\mathbf{S}(I_K)$ and let $\kappa_*x$, respectively $\sigma_*x$, denote the image of $x$ under $\mathbf{S}(\kappa)$, respectively $\mathbf{S}(\sigma)$. Then the following diagram in $\accentset{\sim}{\mathscr{M}}_{\operatorname{RBS}}(\mathcal{C})$ commutes
\begin{center}
\begin{tikzpicture}
\matrix (m) [matrix of math nodes,row sep=2em,column sep=3em]
{
\mathfrak{t}(x) & \mathfrak{t}(\kappa_*x)  & \mathfrak{s}(\kappa_*x) \\
\mathfrak{s}(x) & \mathfrak{s}(\sigma_*x) & \mathfrak{t}(\sigma_*x)\\
};
\path[-stealth]
(m-2-1) edge node[left]{$\phi_x$} (m-1-1)
(m-2-2) edge node[below]{$\phi_{\sigma_*x}$} (m-2-3)
(m-1-3) edge node[above]{$\phi_{\kappa_*x}$} (m-1-2)
;
\path[-]
(m-1-1) edge[double equal sign distance] (m-1-2)
(m-2-1) edge[double equal sign distance] (m-2-2)
(m-2-3) edge[double equal sign distance] (m-1-3)
;
\end{tikzpicture}
\end{center}
In other words, $\phi_x=\phi_{\kappa_*x}\circ \phi_{\sigma_*x}$.
\end{lemma}
\begin{proof}
If $I=\emptyset$, then $x=0$ and $\kappa_*x$, respectively $\sigma_*x$ are filtrations $0\rightarrowtail \cdots \rightarrowtail 0$ indexed over $J$, respectively $K$. The maps $\phi_x$, $\phi_{\kappa_*x}$ and $\phi_{\sigma_*x}$ are thus uniquely determined by the underlying map of posets and the claim follows. Assume now that $I\neq \emptyset$. We'll assume for notational ease that $K=\ast$ (the general case follows by concatenation over $K$). Writing $I=\{1<\cdots <n\}$, the element $x$ is as described in diagram (\ref{element in Sdot}).

Write $J=\{1<\cdots <m\}$ and define
\begin{align*}
\bar{\theta}\colon \{0\}\cup J\rightarrow \{0\}\cup I,\qquad j\mapsto\begin{cases}
0 & j=0 \\
\mathop{\max}_{l\leq j} \theta^{-1}(l) & j\in J \\
\end{cases}
\end{align*}
where on both sides $0$ denotes an adjoined minimal element (see also \S \ref{Lawvere theory}). Then $\kappa_*x$ is given by the diagram obtained from (\ref{element in Sdot}) by picking out the elements $x_{\bar{\theta}(j)\bar{\theta}(k)}$ for all $j<k$ in $\{0\}\cup J$. By definition, we have $\bar{\theta}(m)=n$ and hence the element in the top right hand corner of the diagrams of $x$ and $\kappa_*x$ coincide, i.e. $\mathfrak{t}(x)=(x_{0n})=\mathfrak{t}(\kappa_*x)$. The morphism 
\begin{align*}
\phi_{\kappa_*x}\colon (x_{\bar{\theta}(0)\bar{\theta}(1)},x_{\bar{\theta}(1)\bar{\theta}(2)},\ldots,x_{\bar{\theta}(m-1)\bar{\theta}(m)})\longrightarrow (x_{0n})
\end{align*}
is given by the filtration
\begin{align*}
x_{0\bar{\theta}(1)}\hookrightarrow x_{0\bar{\theta}(2)}\hookrightarrow \ldots\hookrightarrow x_{0\bar{\theta}(m)}=x_{0n}
\end{align*}
with the identification of the associated graded given by the appropriate maps in diagram (\ref{element in Sdot}); in other words, it is the subfiltration of the upper horizontal line of $x$ picking out the elements indexed via $\bar{\theta}$ with associated graded given by the chosen subquotients.

The element $\sigma_*x$ in $\mathbf{S}(I_J)=\prod_{j\in J} S_{|\theta^{-1}(j)|}$ is given by the subdiagrams 
\begin{align*}
(\sigma_*x)^j_{st}=x_{st}, \quad  s<t\text{ in } \{\bar{\theta}(j-1)<\cdots < \bar{\theta}(j)\}.
\end{align*}
In particular, we see that
\begin{align*}
\mathfrak{s}(\sigma_*x)&=(x_{01},x_{12},\ldots,x_{(n-1)n})= \mathfrak{s}(x), \\
\mathfrak{t}(\sigma_*x)&=(x_{\bar{\theta}(0)\bar{\theta}(1)},x_{\bar{\theta}(1)\bar{\theta}(2)},\ldots,x_{\bar{\theta}(m-1)\bar{\theta}(m)})=\mathfrak{s}(\kappa_*x),
\end{align*}
and the morphism
\begin{align*}
\phi_{\sigma_*x}\colon (x_{01},x_{12},\ldots,x_{(n-1)n}) \longrightarrow (x_{\bar{\theta}(0)\bar{\theta}(1)},x_{\bar{\theta}(1)\bar{\theta}(2)},\ldots,x_{\bar{\theta}(m-1)\bar{\theta}(m)})
\end{align*}
is given by the filtrations
\begin{align*}
x_{\bar{\theta}(j-1)(\bar{\theta}(j-1)+1)}\hookrightarrow x_{\bar{\theta}(j-1)(\bar{\theta}(j)+2)} \hookrightarrow \cdots \hookrightarrow x_{\bar{\theta}(j-1)(\bar{\theta}(j)-1)} \hookrightarrow x_{\bar{\theta}(j-1)\bar{\theta}(j)}
\end{align*}
for $j=1,\ldots,m$ together with the identification of the associated gradeds given by the diagram (\ref{element in Sdot}). It is now straightforward to check that $\phi_x=\phi_{\kappa_*x}\circ \phi_{\sigma_*x}$.
\end{proof}

Let us write out a small example to illustrate the factorisation in the lemma above.

\begin{example}\label{example of factorisation}
Consider the maps
\begin{align*}
I=\{1<2<3<4<5<6\} \xrightarrow{\theta} J=\{1<3<6\} \rightarrow K=\ast
\end{align*}
given by  $\theta^{-1}(1)=\{1\}$, $\theta^{-1}(3)=\{2<3\}$, and $\theta^{-1}(6)=\{4<5<6\}$. Then the map $\bar{\theta}\colon \{0<1<3<6\}\rightarrow \{0<1<\cdots <6\}$ is just the inclusion $i\mapsto i$. Let $x$ be an element in $\mathbf{S}(I_\ast)=S_6$ as below. Using the notation of the previous lemma, the element $\kappa_*x$ is the subdiagram given by the red nodes in the diagram above and $\sigma_*x$ consists of the three subdiagrams singled out by the boxed nodes.

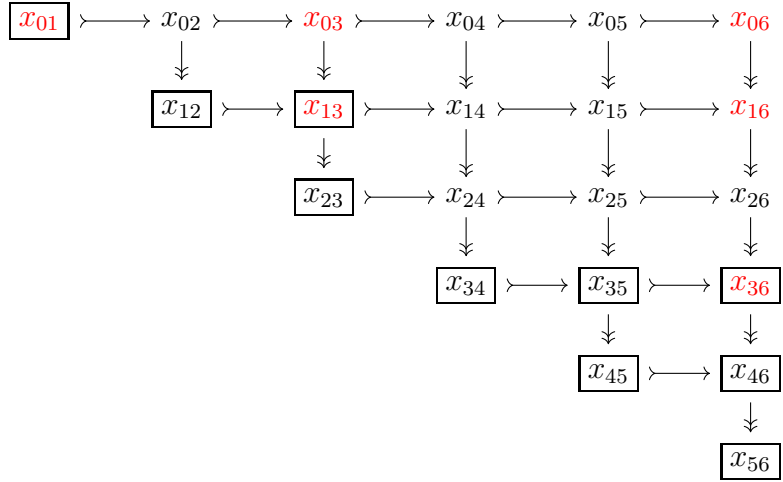
\begin{center}
\begin{tikzpicture}
\matrix (m) [matrix of math nodes,row sep=1em,column sep=2em]
{
\boxed{\color{red} x_{01}}  & x_{02} & \color{red} x_{03} & x_{04} & x_{05} & \color{red} x_{06} \\
		& \boxed{x_{12}} & \boxed{\color{red} x_{13}} & x_{14} & x_{15} & \color{red} x_{16} \\
		&		 & \boxed{x_{23}} & x_{24} & x_{25} & x_{26} \\
		&		 &		  & \boxed{x_{34}} & \boxed{x_{35}} & \boxed{\color{red} x_{36}} \\
		&		 &		  &		   & \boxed{x_{45}} & \boxed{x_{46}} \\
		&		 &		  &		   &	    & \boxed{x_{56}} \\
};
\path[>->]
(m-1-1) edge (m-1-2)
(m-1-2) edge (m-1-3)
(m-1-3) edge (m-1-4)
(m-1-4) edge (m-1-5)
(m-1-5) edge (m-1-6)
(m-2-2) edge (m-2-3)
(m-2-3) edge (m-2-4)
(m-2-4) edge (m-2-5)
(m-2-5) edge (m-2-6)
(m-3-3) edge (m-3-4)
(m-3-4) edge (m-3-5)
(m-3-5) edge (m-3-6)
(m-4-4) edge (m-4-5)
(m-4-5) edge (m-4-6)
(m-5-5) edge (m-5-6)
;
\path[->>]
(m-1-2) edge (m-2-2)
(m-1-3) edge (m-2-3)
(m-1-4) edge (m-2-4)
(m-1-5) edge (m-2-5)
(m-1-6) edge (m-2-6)
(m-2-3) edge (m-3-3)
(m-2-4) edge (m-3-4)
(m-2-5) edge (m-3-5)
(m-2-6) edge (m-3-6)
(m-3-4) edge (m-4-4)
(m-3-5) edge (m-4-5)
(m-3-6) edge (m-4-6)
(m-4-5) edge (m-5-5)
(m-4-6) edge (m-5-6)
(m-5-6) edge (m-6-6)
;
\end{tikzpicture}
\end{center}
Pulling the flags of $\sigma_*x$ up along the three epimorphisms
\begin{align*}
x_{01}\twoheadrightarrow x_{01},\qquad x_{03}\twoheadrightarrow x_{13},\qquad x_{06}\twoheadrightarrow x_{36},
\end{align*}
exactly recovers the upper horizontal line, that is, the filtration specified by the object $x$. This is the claimed commutativity: $\phi_x=\phi_{\kappa_*x}\circ \phi_{\sigma_*x}$.
\end{example}

We apply this to collapse maps in $\mathcal{J}$.

\begin{lemma}
Let $\kappa=(\theta,\operatorname{id})\colon I_P\rightarrow J_P$ be a collapse map in $\mathcal{J}$ and $x$ an object of $\mathbf{S}(I_P)$. Denote by $\kappa_*x$ the image of $x$ under the map $\mathbf{S}(\kappa)$. There exists a canonical morphism $\mathfrak{s}_{\kappa,x}\colon \mathfrak{s}(x) \rightarrow \mathfrak{s}(\kappa_*x)$ in $\accentset{\sim}{\mathscr{M}}_{\operatorname{RBS}}(\mathcal{C})$ such that the following diagram commutes.
\begin{center}
\begin{tikzpicture}
\matrix (m) [matrix of math nodes,row sep=2em,column sep=3em]
{
\mathfrak{t}(x) & \mathfrak{t}(\kappa_*x)  \\
\mathfrak{s}(x) & \mathfrak{s}(\kappa_*x) \\
};
\path[-stealth]
(m-2-1) edge node[left]{$\phi_x$} (m-1-1)
(m-2-2) edge node[right]{$\phi_{\kappa_*x}$} (m-1-2)
(m-2-1) edge node[below]{$\mathfrak{s}_{\kappa,x}$} (m-2-2)
;
\path[-]
(m-1-2) edge[double equal sign distance] (m-1-1)
;
\end{tikzpicture}
\end{center}
\end{lemma}
\begin{proof}
Write out $s\colon I\rightarrow P$ and $r\colon J\rightarrow P$.
Consider the object $\theta\colon I\rightarrow J$ in $\mathcal{J}$ and the morphism $\sigma:=(\operatorname{id}_I, r)\colon I_P\rightarrow I_J$ in $\mathcal{J}$. The morphism
\begin{align*}
\mathfrak{s}_{\kappa,x} := \phi_{\sigma_*x}\colon \mathfrak{s}(x)\rightarrow \mathfrak{t}(\kappa_*x)
\end{align*}
fits into the diagram by \Cref{factorising a morphism in MC}.
\end{proof}

Now for the splitting morphisms.

\begin{lemma}
Let $\sigma=(\operatorname{id},\rho)\colon I_P\rightarrow I_Q$ be a splitting map in $\mathcal{J}$ and $x$ an object of $\mathbf{S}(I_P)$. Denote by $\sigma_*x$ the image of $x$ under $\mathbf{S}(\sigma)$. There is a canonical morphism $\mathfrak{t}_{\kappa,x}\colon \mathfrak{t}(\sigma_*x)\rightarrow \mathfrak{t}(x)$ in $\accentset{\sim}{\mathscr{M}}_{\operatorname{RBS}}(\mathcal{C})$ such that the following diagram commutes.
\begin{center}
\begin{tikzpicture}
\matrix (m) [matrix of math nodes,row sep=2em,column sep=3em]
{
\mathfrak{t}(x) & \mathfrak{t}(\sigma_*x)  \\
\mathfrak{s}(x) & \mathfrak{s}(\sigma_*x) \\
};
\path[-stealth]
(m-2-1) edge node[left]{$\phi_x$} (m-1-1)
(m-2-2) edge node[right]{$\phi_{\sigma_*x}$} (m-1-2)
(m-1-2) edge node[above]{$\mathfrak{t}_{\sigma,x}$} (m-1-1)
;
\path[-]
(m-2-2) edge[double equal sign distance] (m-2-1)
;
\end{tikzpicture}
\end{center}
\end{lemma}
\begin{proof}
Write out $s\colon I\rightarrow P$ and $r\colon I\rightarrow Q$. Consider the object $\rho\colon Q\rightarrow P$ in $\mathcal{I}$ and the morphism $\kappa:=(r,\operatorname{id}_P)\colon I_P\rightarrow Q_P$. The morphism
\begin{align*}
\mathfrak{t}_{\sigma,x}:=\phi_{\kappa_*x}\colon \mathfrak{t}(\sigma_*x)\rightarrow \mathfrak{t}(x)
\end{align*}
fits into the diagram by \Cref{factorising a morphism in MC}.
\end{proof}

\begin{corollary}\label{morphism in twisted arrow category}
Let $\tau=(\theta,\rho)\colon I_P\rightarrow J_Q$ be a morphism in $\mathcal{J}$ and $x$ an object of $\mathbf{S}(I_P)$. There is a canonical commutative diagram in $\accentset{\sim}{\mathscr{M}}_{\operatorname{RBS}}(\mathcal{C})$ as below where $\tau_*x$ denotes the image of $x$ under the map $\mathbf{S}(\tau)$.
\begin{center}
\begin{tikzpicture}
\matrix (m) [matrix of math nodes,row sep=2em,column sep=3em]
{
\mathfrak{t}(x) & \mathfrak{t}(\tau_*x)  \\
\mathfrak{s}(x) & \mathfrak{s}(\tau_*x) \\
};
\path[-stealth]
(m-2-1) edge node[left]{$\phi_x$} (m-1-1)
(m-2-2) edge node[right]{$\phi_{\tau_*x}$} (m-1-2)
(m-2-1) edge node[below]{$\mathfrak{s}_{\tau,x}$} (m-2-2)
(m-1-2) edge node[above]{$\mathfrak{t}_{\tau,x}$} (m-1-1)
;
\end{tikzpicture}
\end{center}
In other words, $\tau$ and $x$ define a canonical morphism $\hat{\tau}\colon \phi_{\tau_*x}\rightarrow\phi_x$ in the twisted arrow category $\operatorname{Tw}(\accentset{\sim}{\mathscr{M}}_{\operatorname{RBS}}(\mathcal{C}))$. Moreover, this association is functorial: if $\nu\colon J_Q\rightarrow K_W$ is another morphism in $\mathcal{J}$, then the following diagram commutes in $\operatorname{Tw}(\accentset{\sim}{\mathscr{M}}_{\operatorname{RBS}}(\mathcal{C}))$.
\begin{center}
\begin{tikzpicture}
\matrix (m) [matrix of math nodes,row sep=2em,column sep=3em]
{
\phi_x & \phi_{\tau_*x} \\
 & \phi_{\nu_*\tau_*x} \\
};
\path[-stealth]
(m-1-2) edge node[above]{$\hat{\tau}$} (m-1-1)
(m-2-2) edge node[below left]{$\widehat{\nu\circ \tau}$} (m-1-1) edge node[right]{$\hat{\nu}$} (m-1-2)
;
\end{tikzpicture}
\end{center}
\end{corollary}
\begin{proof}
To define $\hat{\tau}$, simply factorise $\tau$ as a collapse map followed by a splitting map, \begin{align*}
I_P\xrightarrow{(\theta,\operatorname{id})} J_P\xrightarrow{(\operatorname{id},\rho)} J_Q,
\end{align*}
and apply the two previous lemmas. Tracing through the definitions, it is not difficult to see that applying the two previous lemmas to the factorisation $I_P\xrightarrow{(\operatorname{id},\rho)}I_Q\xrightarrow{(\theta, \operatorname{id})}J_Q$ instead yields the same maps $\mathfrak{s}_{\tau,x}$ and $\mathfrak{t}_{\tau,x}$.

The claim about functoriality is a straightforward application of \Cref{factorising a morphism in MC}. Write $\tau=(\theta, \rho)$ and $\nu=(\theta',\rho')$ and write out the objects $s\colon I\rightarrow P$, $r\colon J\rightarrow Q$, and $t\colon K\rightarrow W$. We have to show that $\mathfrak{s}_{\nu\circ \tau}=\mathfrak{s}_{\nu}\circ \mathfrak{s}_{\tau}$ and $\mathfrak{t}_{\nu\circ\tau}= \mathfrak{t}_{\tau}\circ \mathfrak{t}_{\nu}$ (with the correct dependency on the elements $x$ and $\tau_*x$ that we omit for notational ease). We have
\begin{align*}
\mathfrak{s}_{\tau}=\phi_{\sigma_*x},\qquad \mathfrak{s}_{\nu}=\phi_{\sigma'_*\tau_*x}\qquad\text{and}\qquad \mathfrak{s}_{\nu\circ\tau}=\phi_{\sigma''_*x}
\end{align*}
for the splitting maps
\begin{align*}
\sigma=(\operatorname{id}_I,\rho\circ r)\colon I_P\rightarrow I_J,\quad\sigma'=(\operatorname{id}_J,\rho'\circ t)\colon J_Q\rightarrow J_K,\quad\sigma''=(\operatorname{id}_I, \rho\circ\rho'\circ t)\colon I_P\rightarrow I_K.
\end{align*}
Noting that
\begin{align*}
\sigma= (\operatorname{id}_I,\theta')\circ \sigma''\qquad\text{and}\qquad \sigma'\circ\tau=(\theta,\operatorname{id}_K)\circ \sigma'',
\end{align*}
the desired identification $\phi_{\sigma''_*x} =\phi_{\sigma'_*\tau_*x}\circ\phi_{\sigma_*x}$ follows by applying \Cref{factorising a morphism in MC} to the sequence of maps $I\xrightarrow{\theta}J\xrightarrow{\theta'} K$ and the element $y:=\sigma''_*x$. For the claim about the $\mathfrak{t}$-morphisms, we note that
\begin{align*}
\mathfrak{t}_{\tau}=\phi_{\kappa_*x},\qquad \mathfrak{t}_{\nu}=\phi_{\kappa'_*\tau_*x}\qquad\text{and}\qquad \mathfrak{t}_{\nu\circ \tau}=\phi_{\kappa''_*x}
\end{align*}
for the collapse maps
\begin{align*}
\kappa=(r\circ\theta, \operatorname{id}_P)\colon I_P\rightarrow Q_P,\quad\kappa'=(t\circ \theta',\operatorname{id}_Q)\colon J_Q\rightarrow W_Q,\quad\kappa''=(t\circ\theta'\circ\theta,\operatorname{id}_P)\colon I_P\rightarrow W_P.
\end{align*}
Then note that
\begin{align*}
\kappa=(\rho',\operatorname{id}_P)\circ \kappa'',\qquad\text{and}\qquad \kappa'\circ \tau=(\operatorname{id}_W, \rho)\circ \kappa''
\end{align*}
and apply \Cref{factorising a morphism in MC} to the sequence of maps $W\xrightarrow{\rho'}Q\xrightarrow{\rho} P$ and the element $z=\kappa''_*x$.
\end{proof}

\begin{remark}\label{factorisation diagrammatically}
Let's write out the factorisation $\phi_x=\mathfrak{t}_{\tau,x}\circ \phi_{\tau_*x}\circ \mathfrak{s}_{\tau,x}$ diagrammatically to give the reader a better picture of what's going on. Let $\tau$ be a morphism in $\mathcal{J}$ as below
\begin{center}
\begin{tikzpicture}
\matrix (m) [matrix of math nodes,row sep=2em,column sep=3em]
{
I & J \\
P=\ast & Q \\
};
\path[-stealth]
(m-1-1) edge node[above]{$\theta$} (m-1-2)
(m-1-1) edge node[left]{$s$} (m-2-1)
(m-1-2) edge node[right]{$r$} (m-2-2)
(m-2-2) edge node[below]{$\rho$} (m-2-1)
;
\end{tikzpicture}
\end{center}
and let $x\in \mathbf{S}(I_P)\simeq S_{|I|}$, i.e. $x$ is a diagram $x_{ij}$ for $i<j$ in $0\cup I$ as in (\ref{element in Sdot}). Tracing through the definitions, one sees that
\begin{itemize}
\item the morphism $\mathfrak{t}_{\tau,x}$ is given by the single subdiagram $x_{\overline{r\theta}(i)\overline{r\theta}(j)}$ with $i<j$ in $0\cup Q$;
\item the morphism $\phi_{\tau_*x}$ is given by the ordered collection of subdiagrams $x_{\bar{\theta}(i)\bar{\theta}(j)}$ with $i<j$ in $\{\bar{r}(q-1)<\cdots <\bar{r}(q)\}\subset J$, $q\in Q$;
\item the morphism $\mathfrak{s}_{\tau,x}$ is given by the ordered collection of subdiagrams $x_{ij}$ with $i<j$ in $\{\bar{\theta}(l-1)<\cdots <\bar{\theta}(l)\}\subset I$, $l\in J$.
\end{itemize}

Let's make this even more explicit: suppose $\tau$ is morphism
\begin{center}
\begin{tikzpicture}
\matrix (m) [matrix of math nodes,row sep=2em,column sep=3em]
{
I=\{1<2<3<4<5<6<7<8\} & J=\{2<5<6<8\} \\
P=\{8\} & Q=\{5<8\} \\
};
\path[-stealth]
(m-1-1) edge node[above]{$\theta$} (m-1-2)
(m-1-1) edge node[left]{$s$} (m-2-1)
(m-1-2) edge node[right]{$r$} (m-2-2)
(m-2-2) edge node[below]{$\rho$} (m-2-1)
;
\end{tikzpicture}
\end{center}
whose complementary diagram is given by the inclusions
\begin{center}
\begin{tikzpicture}
\matrix (m) [matrix of math nodes,row sep=2em,column sep=3em]
{
\{0<1<2<3<4<5<6<7<8\} & \{0<2<5<6<8\} \\
\{0<8\} & \{0<5<8\} \\
};
\path[right hook-stealth]
(m-2-1) edge node[left]{$\bar{s}$} (m-1-1)
(m-2-2) edge node[right]{$\bar{r}$} (m-1-2)
(m-2-1) edge node[below]{$\bar{\rho}$} (m-2-2)
;
\path[left hook-stealth]
(m-1-2) edge node[above]{$\bar{\theta}$} (m-1-1)
;
\end{tikzpicture}
\end{center}

Writing out the diagram $x$ as below, we see that
\begin{itemize}
\item the morphism $\mathfrak{t}_{\tau,x}$ is given by the subdiagram picked out by the blue nodes;
\item the morphism $\phi_{\tau_*x}$ is given by the two subdiagrams picked out by the boxed nodes;
\item the morphism $\mathfrak{s}_{\tau,x}$ is given by the four subdiagrams picked out by the red nodes.
\end{itemize}

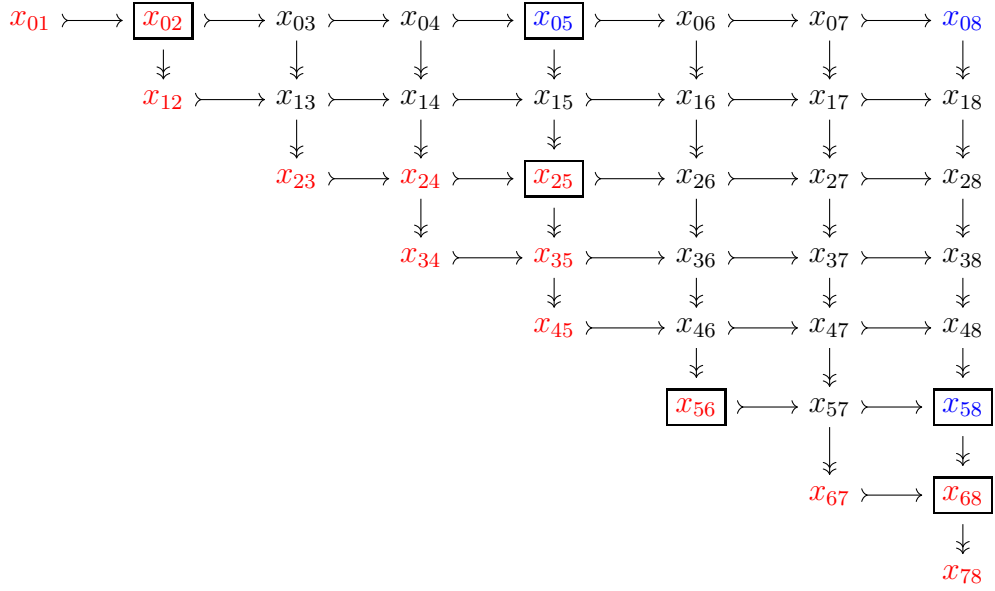
\begin{center}
\begin{tikzpicture}
\matrix (m) [matrix of math nodes,row sep=1em,column sep=2em]
{
\color{red} x_{01} &  \boxed{ \color{red} x_{02}} & x_{03} & x_{04} & \boxed{ \color{blue}x_{05}} & x_{06} & x_{07} & \color{blue} x_{08} \\
		& \color{red} x_{12} & x_{13} & x_{14} & x_{15} & x_{16} & x_{17} & x_{18} \\
		&		 & \color{red} x_{23} & \color{red} x_{24} & \boxed{ \color{red} x_{25}} & x_{26} & x_{27} & x_{28} \\
		&		 &		  & \color{red} x_{34} & \color{red} x_{35} & x_{36} & x_{37} & x_{38} \\
		&		 &		  &		   & \color{red} x_{45} & x_{46} & x_{47} & x_{48} \\
		&		 &		  &		   &	    & \boxed{ \color{red} x_{56}} & x_{57} & \boxed{ \color{blue}x_{58} }\\
		&		 &		  &		   &	    &		 & \color{red} x_{67} &\boxed{ \color{red}  x_{68}} \\
		&		 &		  &		   &	    & 		 &		  &\color{red}  x_{78} \\
};
\path[>->]
(m-1-1) edge (m-1-2)
(m-1-2) edge (m-1-3)
(m-1-3) edge (m-1-4)
(m-1-4) edge (m-1-5)
(m-1-5) edge (m-1-6)
(m-1-6) edge (m-1-7)
(m-1-7) edge (m-1-8)
(m-2-2) edge (m-2-3)
(m-2-3) edge (m-2-4)
(m-2-4) edge (m-2-5)
(m-2-5) edge (m-2-6)
(m-2-6) edge (m-2-7)
(m-2-7) edge (m-2-8)
(m-3-3) edge (m-3-4)
(m-3-4) edge (m-3-5)
(m-3-5) edge (m-3-6)
(m-3-6) edge (m-3-7)
(m-3-7) edge (m-3-8)
(m-4-4) edge (m-4-5)
(m-4-5) edge (m-4-6)
(m-4-6) edge (m-4-7)
(m-4-7) edge (m-4-8)
(m-5-5) edge (m-5-6)
(m-5-6) edge (m-5-7)
(m-5-7) edge (m-5-8)
(m-6-6) edge (m-6-7)
(m-6-7) edge (m-6-8)
(m-7-7) edge (m-7-8)
;
\path[->>]
(m-1-2) edge (m-2-2)
(m-1-3) edge (m-2-3)
(m-1-4) edge (m-2-4)
(m-1-5) edge (m-2-5)
(m-1-6) edge (m-2-6)
(m-1-7) edge (m-2-7)
(m-1-8) edge (m-2-8)
(m-2-3) edge (m-3-3)
(m-2-4) edge (m-3-4)
(m-2-5) edge (m-3-5)
(m-2-6) edge (m-3-6)
(m-2-7) edge (m-3-7)
(m-2-8) edge (m-3-8)
(m-3-4) edge (m-4-4)
(m-3-5) edge (m-4-5)
(m-3-6) edge (m-4-6)
(m-3-7) edge (m-4-7)
(m-3-8) edge (m-4-8)
(m-4-5) edge (m-5-5)
(m-4-6) edge (m-5-6)
(m-4-7) edge (m-5-7)
(m-4-8) edge (m-5-8)
(m-5-6) edge (m-6-6)
(m-5-7) edge (m-6-7)
(m-5-8) edge (m-6-8)
(m-6-7) edge (m-7-7)
(m-6-8) edge (m-7-8)
(m-7-8) edge (m-8-8)
;
\end{tikzpicture}
\end{center}

Essentially, we're just decomposing the diagram $x$ as in \Cref{example of factorisation} but with one more level. The diagram $x$ is recovered by pulling the red diagrams up along the epimorphisms of the boxed diagrams and thereafter pulling the resulting diagrams up along the epimorphisms of the blue diagram --- this is exactly the equality $\phi_x=\mathfrak{t}_{\tau,x}\circ \phi_{\tau_*x}\circ \mathfrak{s}_{\tau,x}$.
\end{remark}

As a converse to \Cref{morphism in twisted arrow category}, we have the following lemma saying that any morphism in the twisted arrow category is uniquely isomorphic to a morphism of the above form.

\begin{lemma}\label{converse to morphism in twisted arrow category}
Let $I_P$, $J_Q$ in $\mathcal{J}$, $x$ in $\mathbf{S}(I_P)$ and $y$ in $\mathbf{S}(J_Q)$, and suppose that $\mu\colon \phi_y\rightarrow \phi_x$ is a morphism in $\operatorname{Tw}(\accentset{\sim}{\mathscr{M}}_{\operatorname{RBS}}(\mathcal{C}))$. There is a unique morphism $\tau\colon I_P\rightarrow J_Q$ in $\mathcal{J}$ and a unique isomorphism $\beta\colon \tau_*x \xrightarrow{\cong} y$ in $\mathbf{S}(J_Q)$ such that the following diagram in $\operatorname{Tw}(\accentset{\sim}{\mathscr{M}}_{\operatorname{RBS}}(\mathcal{C}))$ commutes.
\begin{center}
\begin{tikzpicture}
\matrix (m) [matrix of math nodes,row sep=2em,column sep=2em]
{
\phi_y &  & \phi_{\tau_*x}  \\
& \phi_x & \\
};
\path[-stealth]
(m-1-1) edge node[below left]{$\mu$} (m-2-2)
(m-1-3) edge node[below right]{$\hat{\tau}$} (m-2-2)
(m-1-1) edge node[above]{$\bar{\beta}$} node[below]{$\cong$} (m-1-3)
;
\end{tikzpicture}
\end{center}
\end{lemma}
\begin{proof}
Let's write out the desired diagram in $\accentset{\sim}{\mathscr{M}}_{\operatorname{RBS}}(\mathcal{C})$:
\begin{center}
\begin{tikzpicture}
\matrix (m) [matrix of math nodes,row sep=2em,column sep=3em]
{
\mathfrak{t}(x) & \mathfrak{t}(y) & \mathfrak{t}(\tau_*x)  \\
\mathfrak{s}(x) & \mathfrak{s}(y) & \mathfrak{s}(\tau_*x) \\
};
\path[-stealth]
(m-2-1) edge node[left]{$\phi_x$} (m-1-1)
(m-2-2) edge node[right]{$\phi_y$} (m-1-2)
(m-2-3) edge node[right]{$\phi_{\tau_*x}$} (m-1-3)
(m-2-1) edge node[below]{$s$} (m-2-2)
(m-1-2) edge node[above]{$t$} (m-1-1)
(m-2-1) edge[bend right] node[below]{$\mathfrak{s}_\tau$} (m-2-3)
(m-1-3) edge[bend right] node[above]{$\mathfrak{t}_\tau$} (m-1-1)
(m-2-3) edge[dashed] node[below]{$\scriptstyle\beta_s$} node[above]{$\scriptstyle\cong$} (m-2-2)
(m-1-2) edge[dashed] node[above]{$\scriptstyle\beta_t$} node[below]{$\scriptstyle\cong$} (m-1-3)
;
\end{tikzpicture}
\end{center}

Here $t$ and $s$ are the morphisms defining $\mu\colon \phi_y\rightarrow \phi_x$. First of all note that the morphism $\tau=(\theta,\rho)$ is simply part of the data defining $\mu$: indeed, $s$ is given by an order preserving map $\theta\colon I\rightarrow J$ together with a collection of filtrations, and likewise $t$ is given by an order preserving map $\rho\colon Q\rightarrow P$ together with a collection of filtrations. These commute with the partitioning maps (i.e. define a morphism $I_P\rightarrow J_Q$ in $\mathcal{J}$) because $t\circ \phi_y\circ s = \phi_x$.

The isomorphism $\bar{\beta}\colon \phi_y\rightarrow \phi_{\tau_*x}$ (and hence, in view of \Cref{isomorphisms to isomorphisms}, the isomorphism $\beta\colon\tau_*x\rightarrow y$) arises by spelling out the identification
\begin{align*}
t\circ \phi_y\circ s = \phi_x = \mathfrak{t}_\tau \circ \phi_{\tau_*x}\circ \mathfrak{s}_\tau
\end{align*}
in terms of filtrations with associated gradeds. We leave the details to the reader (\Cref{factorisation diagrammatically} might be helpful).
\end{proof}

\begin{observation}\label{compatibility: isos and morphisms}
Let $\tau\colon I_P\rightarrow J_Q$ be a morphism in $\mathcal{J}$ and let $\alpha\colon x\rightarrow y$ be an isomorphism in $\mathbf{S}(I_P)$. Writing $\tau_*\alpha:=\mathbf{S}(\tau)(\alpha)\colon \tau_*x\rightarrow \tau_*y$, the following diagram in $\operatorname{Tw}(\accentset{\sim}{\mathscr{M}}_{\operatorname{RBS}}(\mathcal{C}))$ commutes:
\begin{center}
\begin{tikzpicture}
\matrix (m) [matrix of math nodes,row sep=2em,column sep=2em]
{
\phi_{\tau_*x} & \phi_x  \\
\phi_{\tau_*y} & \phi_y\\
};
\path[-stealth]
(m-1-1) edge node[left]{$\overline{\tau_*\alpha}$} node[right]{$\scriptstyle\cong$} (m-2-1)
(m-1-1) edge node[above]{$\hat{\tau}$} (m-1-2)
(m-2-1) edge node[below]{$\hat{\tau}$} (m-2-2)
(m-1-2) edge node[right]{$\bar{\alpha}$} node[left]{$\scriptstyle\cong$} (m-2-2)
;
\end{tikzpicture}
\end{center}
This provides the compatibility of the operations $\bar{\scriptstyle(-)}$ and $\hat{\scriptstyle(-)}$ that we'll need.
\end{observation}

\subsection{Comparison}

Having taken care of the technical input, the final identification becomes a relatively simple case of comparing explicitly given simplicial groupoids. Hopefully the main ideas of and intuition behind this identification stand out more clearly in having isolated the technical (and notationally heavier) details in the previous section.

\begin{theorem}\label{equivalence of E1-spaces}
For any exact category $\mathcal{C}$, there is an equivalence of $\mathbb{E}_1$-spaces
\begin{align*}
K^\partial(\mathcal{C})\simeq |\mathscr{M}_{\operatorname{RBS}}(\mathcal{C})|.
\end{align*}
\end{theorem}
\begin{proof}
We will exhibit an equivalence of bar constructions $\mathbb{B}K^\partial (\mathcal{C})\xrightarrow{\simeq} \mathbb{B}|\mathscr{M}_{\operatorname{RBS}}(\mathcal{C})|$. To do this, we exploit the unravelling of $K^\partial$ as in \Cref{monoid over J} and use the simplicial replacement to exhibit equivalences
\begin{align*}
\big(\mathop{\operatorname{colim}}_{\mathcal{J}}\mathbf{S}\big)^n \xrightarrow{\ \simeq\ } |\accentset{\sim}{\mathscr{M}}_{\operatorname{RBS}}(\mathcal{C})|^n, \quad \text{for all } n\in \N,
\end{align*}
commuting with the relevant face and degeneracy maps. First of all we define an equivalence of simplicial groupoids
\begin{align*}
\Psi\colon \operatorname{srep}(\mathbf{S})_\bullet\xrightarrow{\ \simeq\ } N_\bullet i\operatorname{Tw}(\accentset{\sim}{\mathscr{M}}_{\operatorname{RBS}}(\mathcal{C}))
\end{align*}
where on the left hand side we have the simplicial replacement of the functor $\mathbf{S}\colon \mathcal{J}\rightarrow \mathcal{S}$ and on the right hand side we consider the nerve of the twisted arrow category of $\accentset{\sim}{\mathscr{M}}_{\operatorname{RBS}}(\mathcal{C})$, enhanced to a simplicial groupoid as described in \Cref{maximal subgroupoids}. We will omit the ordered sets indexing the partitions for notational ease: that is, we denote an object $I_P$ in $\mathcal{J}$ simply by $I$. For each $n\in \N$, define a functor
\begin{align*}
\Psi_n\colon \operatorname{srep}(\mathbf{S})_n\longrightarrow N_n i\operatorname{Tw}(\accentset{\sim}{\mathscr{M}}_{\operatorname{RBS}}(\mathcal{C}))
\end{align*}
as follows: send an object
\begin{align*}
\big(I_0\xleftarrow{\tau_1} I_1\xleftarrow{\tau_2} \cdots \xleftarrow{\tau_n} I_n, x_n\in \mathbf{S}(I_n)\big)
\end{align*}
to the following sequence in $\operatorname{Tw}(\accentset{\sim}{\mathscr{M}}_{\operatorname{RBS}}(\mathcal{C}))$:
\begin{align*}
\phi_{x_0}\xrightarrow{\hat{\tau}_1}\phi_{x_1}\xrightarrow{\hat{\tau}_2}\cdots \xrightarrow{\hat{\tau}_n} \phi_{x_n}
\end{align*}
where we inductively set $x_{i-1}:=(\tau_i)_*x_i$ and where $\phi_{x_i}$ and $\hat{\tau}_i$ are the canonical objects and morphisms of $\operatorname{Tw}(\accentset{\sim}{\mathscr{M}}_{\operatorname{RBS}}(\mathcal{C}))$ as described in \Cref{morphism in monoidal category} and \Cref{morphism in twisted arrow category}.

Given an isomorphism $\alpha_n\colon x_n\rightarrow y_n$ in the component $\mathbf{S}(I_n)$ associated to the sequence
\begin{align*}
I_0\xleftarrow{\tau_1} I_1\xleftarrow{\tau_2} \cdots \xleftarrow{\tau_n} I_n,
\end{align*}
we inductively set $\alpha_{i-1}:=\mathbf{S}(\tau_i)(\alpha_i)\colon x_{i-1}\rightarrow y_{i-1}$ in $\mathbf{S}(I_i)$. Applying Observations \ref{isomorphisms to isomorphisms} and \ref{compatibility: isos and morphisms}, send $\alpha$ to the induced diagram in $\operatorname{Tw}(\accentset{\sim}{\mathscr{M}}_{\operatorname{RBS}}(\mathcal{C}))$:
\begin{center}
\begin{tikzpicture}
\matrix (m) [matrix of math nodes,row sep=2em,column sep=2em]
{
\phi_{x_0} & \phi_{x_1} & \cdots{\vphantom{_{y_i}}} & \phi_{x_n}  \\
\phi_{y_0} & \phi_{y_1} & \cdots{\vphantom{_{y_i}}} & \phi_{y_n}  \\
};
\path[-stealth]
(m-2-1) edge node[left]{$\bar{\alpha}_0$} node[right]{$\scriptstyle\cong$} (m-1-1)
(m-2-2) edge node[left]{$\bar{\alpha}_1$} node[right]{$\scriptstyle\cong$} (m-1-2)
(m-2-4) edge node[left]{$\bar{\alpha}_n$} node[right]{$\scriptstyle\cong$} (m-1-4)
(m-1-3) edge (m-1-4)
(m-1-2) edge (m-1-3)
(m-1-1) edge (m-1-2)
(m-2-3) edge (m-2-4)
(m-2-2) edge (m-2-3)
(m-2-1) edge (m-2-2)
;
\end{tikzpicture}
\end{center}
This functor is an equivalence: it is essentially surjective by \Cref{converse to morphism in twisted arrow category} and it is fully faithful by \Cref{isomorphisms to isomorphisms}. It is straightforward to see that this commutes with the structure maps by the functoriality $\widehat{\nu\circ \kappa}=\hat{\kappa}\circ \hat{\nu}$ (\Cref{morphism in twisted arrow category}). This identifies the underlying spaces of the $\mathbb{E}_1$-monoids under consideration,
\begin{align*}
\Psi\colon \mathop{\operatorname{colim}}_{\mathcal{J}}\mathbf{S} \xrightarrow{\ \simeq\ } |\accentset{\sim}{\mathscr{M}}_{\operatorname{RBS}}(\mathcal{C})|.
\end{align*}
It remains to be seen that this equivalence respects the $\mathbb{E}_1$-structure. To this end we consider the bar constructions of the monoids in question. The bar construction of $K^\partial(\mathcal{C})=\mathbb{L}S$ can be represented by a bisimplical groupoid via the simplicial replacement:
\begin{align*}
(n)^\pm\mapsto \operatorname{srep}(\mathbf{S})^n.
\end{align*}
Tracing through the definitions, we see that the face maps
\begin{align*}
d_i\colon \operatorname{srep}(\mathbf{S})^n\longrightarrow \operatorname{srep}(\mathbf{S})^{n-1}, \quad 0\leq i\leq n,
\end{align*}
of this construction are as follows: for $0<i<n$, $d_i$ sends an object 
\begin{align*}
X=(X^j)_{j=1,\ldots,n}=(I^j_0\leftarrow \cdots \leftarrow I^j_\bullet, x^j\in \mathbf{S}^{\operatorname{f}}(I^j_\bullet))_{j=1,\ldots,n}
\end{align*}
to the tuple whose $j$'th factor is
\begin{align*}
d_i(X)^j=\begin{cases}
(I^j_0\leftarrow \cdots \leftarrow I^j_\bullet,\  x^j\in \mathbf{S}(I^j_\bullet)) & j<i \\
(I^i_0I^{i+1}_0\leftarrow \cdots \leftarrow I^i_\bullet I^{i+1}_\bullet,\ (x^i,x^{i+1})\in \mathbf{S}(I^i_\bullet I^{i+1}_\bullet)) & j=i \\
(I^{j+1}_0\leftarrow \cdots \leftarrow I^{j+1}_\bullet,\ x^{j+1}\in \mathbf{S}(I^{j+1}_\bullet)) & j>i. \\
\end{cases}
\end{align*}
The first and last face maps are: $d_0(X)=(X^j)_{j=2,\ldots,n}$ and $d_n(X)=(X^j)_{j=1,\ldots,n-1}$.

The degeneracy maps
\begin{align*}
s_i\colon \operatorname{srep}(\mathbf{S})^n\longrightarrow \operatorname{srep}(\mathbf{S})^{n+1}, \quad 0\leq i\leq n
\end{align*}
send an object $(I^j_0\leftarrow \cdots \leftarrow I^j_\bullet, x^j\in \mathbf{S}^{\operatorname{f}}(I^j_\bullet))_{i=1,\ldots,n}$ to the tuple whose $j$'th factor is
\begin{align*}
s_i(X)^j=\begin{cases}
(I^j_0\leftarrow \cdots \leftarrow I^j_\bullet,\ x^j\in \mathbf{S}(I^j_\bullet)) & j<i+1 \\
(\emptyset_\emptyset\leftarrow \cdots \leftarrow \emptyset_\emptyset,\ 0 \in \mathbf{S}(\emptyset_\emptyset)) & j=i+1 \\
(I^{j-1}_0\leftarrow \cdots \leftarrow I^{j-1}_\bullet,\ x^{j-1}\in \mathbf{S}(I^{j-1}_\bullet)) & j>i+1 \\
\end{cases}
\end{align*}

On the other side, we have the bar construction of the monoid $|\mathscr{M}_{\operatorname{RBS}}(\mathcal{C})|$ represented by a bisimplical groupoid
\begin{align*}
(n)^\pm\rightarrow (N_\bullet i\operatorname{Tw}(\accentset{\sim}{\mathscr{M}}_{\operatorname{RBS}}(\mathcal{C})))^n
\end{align*}
with face and degeneracy maps induced by the inherited monoidal structure of $\operatorname{Tw}(\accentset{\sim}{\mathscr{M}}_{\operatorname{RBS}}(\mathcal{C}))$. More explicitly, the face map $d_i$, $0<i<n$, sends a tuple $(\bar{\phi}^j)_{j=1,\ldots,n}$ of sequences
\begin{align*}
\bar{\phi}^j = (\phi^j_0\rightarrow \phi_1^j\rightarrow \cdots \rightarrow \phi_\bullet^j)
\end{align*}
to the tuple whose $j$'th factor is
\begin{align*}
d_i(\bar{\phi})^j=\begin{cases}
\bar{\phi}^j & j<i \\
\bar{\phi}^i\circledast \bar{\phi}^{i+1} & j=i \\
\bar{\phi}^{j+1} & j>i \\
\end{cases}
\end{align*}
where $\bar{\phi}^i\circledast \bar{\phi}^{i+1}$ is the sequence
\begin{align*}
\phi^i_0\circledast \phi^{i+1}_0\rightarrow \phi^i_1\circledast \phi^{i+1}_1\rightarrow \cdots \rightarrow \phi^i_\bullet\circledast \phi^{i+1}_\bullet.
\end{align*}
The first and last face maps are given by: $d_0(\bar{\phi}^j)=(\bar{\phi}^j)_{j=2,\ldots,n}$ and $d_n(\bar{\phi}^j)=(\bar{\phi}^j)_{j=1,\ldots,n-1}$.

The degeneracy map $s_i$, $0\leq i\leq n$, sends the tuple $(\bar{\phi}^j)_{j=1,\ldots,n}$ from above to the tuple whose $j$'th factor is
\begin{align*}
s_i(\bar{\phi})^j=\begin{cases}
\bar{\phi}^j & j<i+1 \\
0 & j=i+1 \\
\bar{\phi}^{j-1} & j>i+1 \\
\end{cases},
\end{align*}
where $0$ denotes the constant sequence on the monoidal unit
\begin{align*}
\operatorname{id}_\emptyset \rightarrow \operatorname{id}_\emptyset\rightarrow \cdots \rightarrow \operatorname{id}_\emptyset.
\end{align*}

The equivalence $\Psi$ defined above gives rise to levelwise equivalences
\begin{align*}
\Psi^n\colon \operatorname{srep}(\mathbf{S})^n\xrightarrow{\ \simeq\ } \left(N_\bullet i\operatorname{Tw}(\accentset{\sim}{\mathscr{M}}_{\operatorname{RBS}}(\mathcal{C}))\right)^n,\quad \text{for all }n\in \N,
\end{align*}
and it is straightforward to verify that these maps respect the structure maps identified above. In other words, the resulting maps
\begin{align*}
\big(\mathop{\operatorname{colim}}_{\mathcal{J}}\mathbf{S}\big)^n \xrightarrow{\ \simeq\ } |\accentset{\sim}{\mathscr{M}}_{\operatorname{RBS}}(\mathcal{C})|^n,\quad n\in \N,
\end{align*}
respect the face and degeneracy maps of the bar constructions and we have exhibited the desired equivalence, $K^\partial(\mathcal{C})\simeq |\mathscr{M}_{\operatorname{RBS}}(\mathcal{C})|$, of $\mathbb{E}_1$-spaces.
\end{proof}

\begin{observation}\label{restricting to flags}
The equivalence of simplicial groupoids
\begin{align*}
\Psi^\bullet\colon \operatorname{srep}(\mathbf{S})^\bullet\xrightarrow{\ \simeq\ } \left(N i\operatorname{Tw}(\accentset{\sim}{\mathscr{M}}_{\operatorname{RBS}}(\mathcal{C}))\right)^\bullet
\end{align*}
given in the proof above restricts to an equivalence
\begin{align*}
\Psi^\bullet\colon \operatorname{srep}(\mathbf{S}^{\operatorname{f}})^\bullet\xrightarrow{\ \simeq\ } \left(N i\operatorname{Tw}(\mathscr{M}_{\operatorname{RBS}}(\mathcal{C}))\right)^\bullet.
\end{align*}
On the left hand side, we have the simplicial replacement of the functor $\mathbf{S}^{\operatorname{f}}\colon \mathcal{I}_\gg\rightarrow \mathcal{S}$ picking out the subgroupoid of flags in Waldhausen's $S_\bullet$-construction (\Cref{examples: picking out the non-degenerate simplices}); on the right hand side we have the twisted arrow category of the (slim) monoidal $\operatorname{RBS}$-category. In view of \Cref{examples: picking out the non-degenerate simplices} and \Cref{fluffed up model of monoidal category of flags and associated gradeds}, we get the same $\mathbb{E}_1$-spaces after realisation.
\end{observation}

The identification
\begin{align*}
K^\partial(\mathcal{C})\simeq |\mathscr{M}_{\operatorname{RBS}}(\mathcal{C})|.
\end{align*}
furnishes us with a very explicit model for partial algebraic K-theory (of exact categories, not in full generality!). This is otherwise defined in terms of a universal property, so this identification provides a more hands on object to study. In particular the combination of having a universal property on the one hand and an explicit tangible monoidal category on the other may prove useful. We have for example the following simple observation that partial K-theory inherits the decomposition into components of $\mathscr{M}_{\operatorname{RBS}}(\mathcal{C})$ (\Cref{decomposition of monoidal RBS}). In particular, the partial algebraic K-theory $K^\partial(A):=K^\partial(\mathcal{P}(A))$ of an associative ring $A$ decomposes into a disjoint union of reductive Borel--Serre categories. This recovers the decomposition of $K^\partial(\F_p)$ established in \cite[\S 5.2]{Yuan}.

\begin{corollary}
Let $\mathcal{C}$ be an exact category. Then we have an equivalence of spaces
\begin{align*}
K^\partial(\mathcal{C})\simeq \coprod_{[c]\in \pi_0(\mathcal{C}^\simeq)} |\operatorname{RBS}(c)|
\end{align*}
where $\operatorname{RBS}(c)$ is the full subcategory defined in \Cref{general RBS-categories}. In particular, for $\mathcal{C}=\mathcal{P}(A)$, we have an equivalence of spaces
\begin{align*}
K^\partial(A)\simeq \coprod_{M\in \mathcal{M}} |\operatorname{RBS}(M)|
\end{align*}
where $M$ runs over a set $\mathcal{M}$ of representatives of isomorphism classes of finitely generated projective $A$-modules.
\end{corollary}

The main motivation behind this comparison is the remarkable fact that even though partial algebraic K-theory has a universal property as an $\mathbb{E}_1$-space, it naturally admits the structure of an $\mathbb{E}_\infty$-space (\cite[Corollary 4.9.1]{Yuan}). The monoidal category $\mathscr{M}_{\operatorname{RBS}}(\mathcal{C})$ is \textit{only} monoidal --- not symmetric, not even braided --- and hence, its realisation is a priori only an $\mathbb{E}_1$-space. The above identification reveals, however, that it naturally admits the structure of an $\mathbb{E}_\infty$-space! We will establish this explicitly in the next section and explore some immediate consequences.

\section{Homotopy commutativity, explicitly}\label{homotopy commutativity}

The identification
\begin{align*}
K^\partial(\mathcal{C})\simeq |\mathscr{M}_{\operatorname{RBS}}(\mathcal{C})|
\end{align*}
revealed to us that the $\mathbb{E}_1$-space obtained as the realisation of the monoidal $\operatorname{RBS}$-category is in fact $\mathbb{E}_\infty$. As mentioned, this was in many ways the main motivation behind this comparison. In an attempt to better reveal and understand the underlying dynamics, we establish this homotopy commutativity explicitly in this section. The proof is heavily inspired by Yuan's proof that partial algebraic K-theory is $\mathbb{E}_\infty$ and by our general analysis of partial algebraic K-theory leading up to the identification of the previous section. Moreover, it turns out that we'll be needing the exact same technical ingredients that went into the identification above. We go on to explore what this extra structure can do for us in our investigation of the reductive Borel--Serre categories.

\subsection{\texorpdfstring{$\mathbb{E}_\infty$}{E-infinity}-structure: an Eckmann-Hilton argument}

In this section, we show that the assignment
\begin{align*}
\mathcal{C}\mapsto |\mathscr{M}_{\operatorname{RBS}}(\mathcal{C})|
\end{align*}
preserves finite products. As an immediate consequence, we see that if $\mathcal{C}$ is an exact category, then $|\mathscr{M}_{\operatorname{RBS}}(\mathcal{C})|$ admits an $\mathbb{E}_\infty$-structure that refines the $\mathbb{E}_1$-structure.

\begin{proposition}\label{product preserving}
The functor
\begin{align*}
\operatorname{Cat}^{\operatorname{filtr}}\rightarrow \operatorname{Mon}(\mathcal{S}),\quad \mathcal{C}\mapsto |\mathscr{M}_{\operatorname{RBS}}(\mathcal{C})|
\end{align*}
preserves finite products.
\end{proposition}
\begin{proof}
We will show that the natural map
\begin{align*}
\Phi=\Phi_{\mathcal{C},\mathcal{D}}\colon \accentset{\sim}{\mathscr{M}}_{\operatorname{RBS}}(\mathcal{C}\times \mathcal{D})\rightarrow \accentset{\sim}{\mathscr{M}}_{\operatorname{RBS}}(\mathcal{C})\times \accentset{\sim}{\mathscr{M}}_{\operatorname{RBS}}(\mathcal{D})
\end{align*}
induced by the projections from the product $\mathcal{C}\times \mathcal{D}$ is a homotopy equivalence on realisations. We will do this by showing that the induced map of twisted arrow categories induces a homotopy equivalence on realisations. For any $\mathcal{C}$, we have a functor $\accentset{\sim}{\mathscr{M}}_{\operatorname{RBS}}(\mathcal{C})\rightarrow \operatorname{Ord}$ sending an object $(c_i)_{i\in I}$ to the underlying linearly ordered set $I$, and a morphism $(c_i)_{i\in I}\rightarrow (d_j)_{j\in J}$ to the underlying map $I\rightarrow J$ of linearly ordered sets. Taking twisted arrow categories yields a commutative diagram.
\begin{center}
\begin{tikzpicture}
\matrix (m) [matrix of math nodes,row sep=2em,column sep=3em]
{
\operatorname{Tw}\big(\accentset{\sim}{\mathscr{M}}_{\operatorname{RBS}}(\mathcal{C}\times \mathcal{D})\big) & \operatorname{Tw}\big(\accentset{\sim}{\mathscr{M}}_{\operatorname{RBS}}(\mathcal{C})\big)\times \operatorname{Tw}\big(\accentset{\sim}{\mathscr{M}}_{\operatorname{RBS}}(\mathcal{D})\big) \\
\mathcal{J}^{\operatorname{op}} & \mathcal{J}^{\operatorname{op}} \times \mathcal{J}^{\operatorname{op}}  \\
};
\path[-stealth]
(m-1-1) edge node[above]{$\scriptstyle\operatorname{Tw}(\Phi)$} (m-1-2)
(m-2-1) edge node[below]{$\bigtriangleup$} (m-2-2)
(m-1-1) edge node[left]{$\rho$} (m-2-1)
(m-1-2) edge node[right]{$\rho\times\rho$} (m-2-2)
;
\end{tikzpicture}
\end{center}

We will show that the top horizontal map induces a homotopy equivalence on geometric realisations using Quillen's Theorem A; more precisely, we'll show that the left fibre of this map is equivalent to the left fibre of the lower horizontal map which we know to have contractible geometric realisation since $\mathcal{J}$ is sifted (\Cref{J is sifted}).

Let $\phi\colon (c_i)_{i\in I}\rightarrow ({c'}_{\! p})_{p\in P}$ be a morphism in $\accentset{\sim}{\mathscr{M}}_{\operatorname{RBS}}(\mathcal{C})$ and $\psi\colon (d_j)_{j\in J}\rightarrow ({d'}_{\! q})_{q\in Q}$ a morphism in $\accentset{\sim}{\mathscr{M}}_{\operatorname{RBS}}(\mathcal{D})$. Write $\Phi_{/(\phi,\psi)}$ for the left fibre of $\operatorname{Tw}(\Phi)$ over the object $(\phi,\psi)$, and write $\bigtriangleup_{(I_P,J_Q)/}$ for the right fibre of the diagonal map $\bigtriangleup\colon \mathcal{J}\rightarrow \mathcal{J}\times \mathcal{J}$ over the underlying partitioned linearly ordered sets $(I_P,J_Q)$. We will show that the induced map
\begin{align*}
\Pi\colon \Phi_{/(\phi,\psi)}\longrightarrow (\bigtriangleup_{(I_P,J_Q)/})^{\operatorname{op}}
\end{align*}
is an equivalence by defining an inverse functor
\begin{align*}
\Lambda\colon(\bigtriangleup_{(I_P,J_Q)/})^{\operatorname{op}}\longrightarrow \Phi_{/(\phi,\psi)}.
\end{align*}

By \Cref{morphism in monoidal category}, we may write $\phi=\phi_x$ and $\psi=\phi_y$ for some $x\in \mathbf{S}_{\mathcal{C}}(I_P)$ and $y\in \mathbf{S}_{\mathcal{D}}(J_Q)$ (we use the subscript to keep track of the ambient category with filtrations). We define $\Lambda$ as follows. It sends an object
\begin{align*}
I_P\xrightarrow{\tau} K_R\xleftarrow{\nu} J_Q
\end{align*}
to the following object
\begin{align*}
(\phi_{\tau_*x\times \nu_*y},\  
\phi_{\tau_*x}\xrightarrow{\hat{\tau}}\phi_x,\ \phi_{\nu_*y}\xrightarrow{\hat{\nu}}\phi_y)
\end{align*}
using the notation of \Cref{morphism in twisted arrow category}. Here, $\tau_*x\times \nu_*y\in \mathbf{S}_{\mathcal{C}\times \mathcal{D}}(K_R)$ is the object given by taking the product of the relevant diagrams in $S_{|K_r|}$ under the isomorphism
\begin{align*}
S_\bullet(\mathcal{C})\times S_\bullet(\mathcal{D})\xrightarrow{\cong} S_\bullet(\mathcal{C}\times \mathcal{D})
\end{align*}
for each $r\in R$ (the key here is that we can take the product of diagrams of the \textit{same} size). Consider a morphism in $\bigtriangleup_{(I_P,J_Q)/}$ as below.
\begin{center}
\begin{tikzpicture}
\matrix (m) [matrix of math nodes,row sep=1em,column sep=3em]
{
& K_R & \\
I_P & & J_Q \\
& L_W & \\
};
\path[-stealth]
(m-2-1) edge node[above left]{$\tau$} (m-1-2) edge node[below left]{$\tau'$} (m-3-2)
(m-2-3) edge node[above right]{$\nu$} (m-1-2) edge node[below right]{$\nu'$} (m-3-2)
(m-1-2) edge node[right]{$\gamma$} (m-3-2)
;
\end{tikzpicture}
\end{center}
The functor $\Lambda$ sends this to the morphism
\begin{align*}
(\phi_{\tau_*x\times \nu_*y},\  
\hat{\tau},\ \hat{\nu})\longrightarrow (\phi_{\tau'_*x\times \nu'_*y},\  
\hat{\tau}',\ \hat{\nu}')
\end{align*}
in $\Phi_{/(\phi,\psi)}$ given by the map
\begin{align*}
\hat{\gamma}\colon \phi_{\tau'_*x\times \nu'_*y}=\phi_{\gamma_*(\tau_*x\times \nu_*x)}\longrightarrow\phi_{\tau_*x\times \nu_*y}
\end{align*}
in $\accentset{\sim}{\mathscr{M}}_{\operatorname{RBS}}(\mathcal{C}\times \mathcal{D})$. This defines a map in $\Phi_{/(\phi,\psi)}$ by functoriality of the association $\hat{\scriptstyle(-)}$ (\Cref{morphism in twisted arrow category}). Clearly, $\Pi\circ \Lambda=\operatorname{id}$. Conversely, the unique isomorphisms
\begin{center}
\begin{tikzpicture}
\matrix (m) [matrix of math nodes,row sep=1em,column sep=3em]
{
& \phi_{z\times w} & \\
\phi_x & & \phi_y \\
& \phi_{\tau_*x\times \nu_*y} & \\
};
\path[-stealth]
(m-1-2) edge (m-2-1) edge (m-2-3) 
(m-3-2) edge node[below left]{$\hat{\tau}$} (m-2-1) edge node[below right]{$\hat{\nu}$} (m-2-3)
(m-3-2) edge node[right]{$\overline{\scriptstyle\alpha\times \beta}$} node[left]{$\scriptstyle\cong$} (m-1-2)
;
\end{tikzpicture}
\end{center}
supplied by \Cref{converse to morphism in twisted arrow category} exhibit a natural isomorphism $\Lambda\circ \Pi\simeq \operatorname{id}$.
\end{proof}

By what we may call an $\infty$-categorical Eckmann-Hilton argument, it follows immediately from this that the $\mathbb{E}_1$-space $|\mathscr{M}_{\operatorname{RBS}}(\mathcal{C})|$ is in fact $\mathbb{E}_\infty$ (recovering the corollary to the identification $K^\partial(\mathcal{C})\simeq |\mathscr{M}_{\operatorname{RBS}}(\mathcal{C})|$). 

\begin{corollary}\label{E-infinity structure}
For any exact category $\mathcal{C}$, the $\mathbb{E}_1$-space $|\mathscr{M}_{\operatorname{RBS}}(\mathcal{C})|$ naturally admits the structure of an $\mathbb{E}_\infty$-space refining the $\mathbb{E}_1$-structure.
\end{corollary}
\begin{proof}
Indeed, direct sum exhibits $\mathcal{C}$ as an $\mathbb{E}_\infty$-algebra object in $\operatorname{Cat}^{\operatorname{filtr}}$ and we may use \Cref{product preserving} above to transport this to an $\mathbb{E}_\infty$-structure on $|\mathscr{M}_{\operatorname{RBS}}(\mathcal{C})|$. This refines the $\mathbb{E}_1$-structure by the Dunn Additivity Theorem (\cite[Theorem 5.1.2.2]{LurieHA}).
\end{proof}

\subsection{The group completion theorem}

Allow us to briefly share with the reader the situation we were in before having established the $\mathbb{E}_\infty$-structure on $|\mathscr{M}_{\operatorname{RBS}}(\mathcal{C})|$. In \cite{ClausenOrsnesJansen}, we show that the $\operatorname{RBS}$-categories stabilise to the algebraic K-theory space via group completion; more precisely, we show that for any exact category $\mathcal{C}$, we have a homotopy equivalence:
\begin{flalign*}
& & K(\mathcal{C})\simeq \Omega B|\mathscr{M}_{\operatorname{RBS}}(\mathcal{C})| \qquad\qquad  \text{(\cite[Theorem 7.38]{ClausenOrsnesJansen}).}
\end{flalign*}

For an associative ring $A$, the monoidal category decomposes into a disjoint union of reductive Borel--Serre categories:
\begin{align*}
\mathscr{M}_{\operatorname{RBS}}(A):=\mathscr{M}_{\operatorname{RBS}}(\mathcal{P}(A))\simeq \coprod_{M\in \mathcal{M}}\operatorname{RBS}(M)
\end{align*}
where $\mathcal{P}(A)$ is the exact category of finitely generated projective $A$-modules and $\mathcal{M}$ is a set of representatives of isomorphism classes of such. For any finitely generated projective $A$-module $M$, we have a stabilisation map
\begin{align*}
\operatorname{RBS}(M)\xrightarrow{\, (-,A)\, } \operatorname{RBS}(M\oplus A), \quad (M_1,\ldots,M_d)\mapsto (M_1,\ldots,M_d,A),
\end{align*}
and thus we can consider the more naive stabilisation procedure of taking the colimit along these stabilisation maps:
\begin{align*}
\operatorname{RBS}(M)\rightarrow \operatorname{RBS}(M\oplus A)\rightarrow \cdots \rightarrow  \operatorname{RBS}(M\oplus A^n) \rightarrow \operatorname{RBS}(M\oplus A^{n+1})\rightarrow \cdots
\end{align*}

The group completion theorem provides a general tool for comparing these two stabilisation procedures: group completion $\Omega B (-)$ on the one hand and the colimits along stabilisation maps on the other (\cite{McDuffSegal}, \cite{RandalWilliams}, \cite{Nikolaus}). To make use of it, however, requires some kind of commutativity, even if only a bare minimum (see e.g. \cite{Gritschacher}). Having a priori no homotopy commutativity in $|\mathscr{M}_{\operatorname{RBS}}(\mathcal{C})|$, we could not apply it to compare $K(A)\simeq \Omega B| \mathscr{M}_{\operatorname{RBS}}(A)|$ with the colimit
\begin{align*}
\operatorname{RBS}_\infty(A):=\mathop{\operatorname{colim}}_{n\rightarrow \infty}\operatorname{RBS}(A^n).
\end{align*}

Now that we have established that $\mathscr{M}_{\operatorname{RBS}}(\mathcal{C})$ admits a natural $\mathbb{E}_\infty$-structure, we \textit{can} apply the group completion theorem. In order to analyse the plus-constructions that appear, we recall that a finitely generated projective module $M$ is \textit{split Noetherian} if every increasing chain of splittable submodules of $M$ stabilises (\cite[Definition 1.1]{ClausenOrsnesJansen}). This is a technical condition that ensures that we have reasonable control over the poset of splittable flags and the action of the automorphism group on this poset (see \cite[Lemma 5.4]{ClausenOrsnesJansen}). Moreover note that if either:
\begin{enumerate}
\item $A$ is Noetherian, or
\item $A$ is commutative and $\operatorname{Spec}(A)$ has only finitely many connected components
\end{enumerate}
then every finitely generated projective $A$-module $M$ is split Noetherian (\cite[Lemma 5.5]{ClausenOrsnesJansen}).

\begin{corollary}\label{applying the group completion theorem}
For any associative ring $A$, we have
\begin{align*}
K(A)\simeq \Omega B|\mathscr{M}_{\operatorname{RBS}}(A)|\simeq K_0(A)\times |\operatorname{RBS}_\infty(A)|^+
\end{align*}
If $A$ is such that all finitely generated projective modules are split Noetherian, then $|\operatorname{RBS}_\infty(A)|$ is its own plus-construction and
\begin{align*}
|\operatorname{RBS}_\infty(A)|\simeq \operatorname{BGL}_\infty(A)^+.
\end{align*}
\end{corollary}
\begin{proof}
Applying \cite[Theorem 1.1]{RandalWilliams}, we can form the telescope along the stabilisation maps
\begin{align*}
|\mathscr{M}_{\operatorname{RBS}}(A)|_\infty :=\operatorname{colim}(|\mathscr{M}_{\operatorname{RBS}}(A)|\xrightarrow {\,\cdot A\,}|\mathscr{M}_{\operatorname{RBS}}(A)|\xrightarrow {\,\cdot A\,} |\mathscr{M}_{\operatorname{RBS}}(A)|\xrightarrow {\,\cdot A\,} \cdots\, )
\end{align*}
and by \cite[Corollary 1.2]{RandalWilliams}, we have an equivalence
\begin{align*}
\Omega B|\mathscr{M}_{\operatorname{RBS}}(A)| \simeq |\mathscr{M}_{\operatorname{RBS}}(A)|_\infty^+
\end{align*}
where $|\mathscr{M}_{\operatorname{RBS}}(A)|_\infty^+$ is the result of applying Quillen's +-construction to each component of $|\mathscr{M}_{\operatorname{RBS}}(A)|_\infty$ with respect to the commutator subgroup of its fundamental group (not to be confused with the unitalisation functor of \S \ref{non-unital monoids}).

If $A$ is such that all finitely generated projective modules are split Noetherian, then we know from \cite[Theorem 5.9]{ClausenOrsnesJansen} that the fundamental group of the $\operatorname{RBS}$-category is
\begin{align*}
\pi_1(\operatorname{RBS}(A^n))\cong \operatorname{GL}_n(A)/\operatorname{E}(A^n),\quad n\geq 1,
\end{align*}
where $\operatorname{E}(A^n)\subset \operatorname{GL}_n(A)$ is the subgroup generated by those automorphisms of $A^n$ which induce the identity on the associated graded of \textit{some} splittable flag of submodules (see \cite[\S 1]{ClausenOrsnesJansen}). Since elements in $\operatorname{E}(A^n)$ map to zero in $K_1(A)$ and the subgroup $\operatorname{E}_n(A)$ generated by elementary matrices is always contained in $\operatorname{E}(A^n)$, it follows that the fundamental group of the component $|\operatorname{RBS}_\infty(A)|$ is the first K-group,
\begin{align*}
\pi_1(|\operatorname{RBS}_\infty(A)|) \cong K_1(A).
\end{align*}
As such it is abelian and the component $|\operatorname{RBS}_\infty(A)|$ is its own $+$-construction.
\end{proof}

\begin{remark}\label{get rid of split Noetherian hypothesis}
It is somewhat annoying that the technical hypothesis that finitely generated projective modules should be split Noetherian shows up here. For an arbitrary associative ring $A$, the only thing missing in order to deduce that
\begin{align*}
\pi_1(|\operatorname{RBS}_\infty(A)|) \cong K_1(A).
\end{align*}
and hence that $|\operatorname{RBS}_\infty(A)|$ is its own plus-construction, is that the map
\begin{align*}
\operatorname{GL}_\infty(A) \rightarrow \pi_1(|\operatorname{RBS}_\infty(A)|)
\end{align*}
should be surjective (see the proof of \cite[Theorem 5.9]{ClausenOrsnesJansen}). It does not seem completely unreasonable to believe that this might be true in full generality (even unstably). The proof of \cite[Theorem 5.9]{ClausenOrsnesJansen} uses the split Noetherian hypothesis for an inductive argument, so one would have to prove it more directly.
\end{remark}

\section{The monoidal rank filtration}\label{monoidal rank filtration}

The rank function on finitely generated projective modules (when well-defined) gives rise to a natural filtration on the monoidal $\operatorname{RBS}$-category. This is analogous to Quillen's rank filtration of the K-theory space via the Q-construction (\cite{Quillen73a}, see also \Cref{rank filtration for commutative ring}). One can interpret the filtration of the monoidal $\operatorname{RBS}$-category as somewhat more fundamental than Quillen's for the following two reasons:
\begin{enumerate}
\item we recover Quillen's rank filtration via group completion (in fact, our filtration deloops to Quillen's filtration at the level of the $Q$-construction (\cite[Theorem 7.38]{ClausenOrsnesJansen}));
\item it is more highly structured as it is a filtration of \textit{monoidal} categories.
\end{enumerate}

In this section, we study this monoidal rank filtration. We show that its associated graded is extremely simple: it is given by free monoidal categories on the (non-monoidal) $\operatorname{RBS}$-categories modulo their boundary. We unravel this even further by introducing Tits complexes of finitely generated projective modules, a direct generalisation of the usual Tits buildings.

\subsection{The filtration}\label{filtration}

Let $\mathcal{C}$ be an exact category and suppose we have a \textit{grading}, that is, a monoidal functor
\begin{align*}
r\colon \mathcal{C}\rightarrow \N
\end{align*}
where $\mathcal{C}$ is equipped with the monoidal structure given by the biproduct. For any $N\geq 0$, let $\mathcal{C}_{\leq N}\subseteq \mathcal{C}$ denote the full subcategory spanned by objects $c$ of rank at most $N$, $r(c)\leq N$. This is no longer an exact category as it is not additive, but it is still a category with filtrations and as such, we can still construct a monoidal $\operatorname{RBS}$-category
\begin{align*}
\mathscr{M}_{\operatorname{RBS}}(\mathcal{C}_{\leq N}), \quad N\geq 0.
\end{align*}
This, indeed, is the main reason for defining $\mathscr{M}_{\operatorname{RBS}}(-)$ in this generality.

\begin{definition}
In the situation described above, the resulting filtration,
\begin{align*}
\mathscr{M}_{\operatorname{RBS}}(\mathcal{C}_{\leq 0})\subset \mathscr{M}_{\operatorname{RBS}}(\mathcal{C}_{\leq 1}) \subset\cdots \subset \mathscr{M}_{\operatorname{RBS}}(\mathcal{C}_{\leq N})\subset \mathscr{M}_{\operatorname{RBS}}(\mathcal{C}_{\leq N+1}) \subset \cdots
 \subset \mathscr{M}_{\operatorname{RBS}}(\mathcal{C}),
\end{align*}
is called the \textit{(monoidal) rank filtration} of $\mathscr{M}_{\operatorname{RBS}}(\mathcal{C})$.

Analogously,
\begin{align*}
\mathscr{M}_{\operatorname{RBS}}^{\operatorname{nu}}(\mathcal{C}_{\leq 0})\subset \mathscr{M}_{\operatorname{RBS}}^{\operatorname{nu}}(\mathcal{C}_{\leq 1}) \subset\cdots \subset \mathscr{M}_{\operatorname{RBS}}^{\operatorname{nu}}(\mathcal{C}_{\leq N})\subset \mathscr{M}_{\operatorname{RBS}}^{\operatorname{nu}}(\mathcal{C}_{\leq N+1}) \subset \cdots
 \subset \mathscr{M}_{\operatorname{RBS}}^{\operatorname{nu}}(\mathcal{C})
\end{align*}
is called the \textit{non-unital (monoidal) rank filtration} of $\mathscr{M}_{\operatorname{RBS}}^{\operatorname{nu}}(\mathcal{C})$.
\end{definition}

\begin{example}\label{rank filtration for commutative ring}
Let $R$ be a commutative ring whose spectrum is connected. Then the rank of a finitely generated projective $R$-module is well-defined and gives a grading
\begin{align*}
\mathcal{P}(R)\rightarrow \N, \quad M\mapsto \operatorname{rk}(M).
\end{align*}
By \cite[Theorem 7.38]{ClausenOrsnesJansen}, we have an identification
\begin{align*}
B|\mathscr{M}_{\operatorname{RBS}}(\mathcal{P}(R)_{\leq N})|\simeq |Q(\mathcal{P}(R)_{\leq N})|
\end{align*}
of the classifying space of the $\mathbb{E}_1$-space $|\mathscr{M}_{\operatorname{RBS}}(\mathcal{P}(R)_{\leq N})|$ with Quillen's Q-construction. It follows that after group completion, we recover Quillen's rank filtration of $K(R)$ (\cite{Quillen73a}). Upon geometric realisation, we thus have a filtration of $\mathbb{E}_1$-spaces,
\begin{align*}
\cdots \subset |\mathscr{M}_{\operatorname{RBS}}(\mathcal{P}(R)_{\leq N})|\subset |\mathscr{M}_{\operatorname{RBS}}(\mathcal{P}(R)_{\leq N+1})| \subset \cdots
 \subset |\mathscr{M}_{\operatorname{RBS}}(R)|,
\end{align*}
that group completes to Quillen's rank filtration.

Banús has recently proved that this filtration satisfies homological stability, an analogue of Quillen's result on the rank filtration of the Q-construction (\cite[Corollary 4.2.3]{Damia}, \cite[Corollary to Theorem 3]{Quillen73a}): more precisely, the map
\begin{align*}
H_d(|\mathscr{M}_{\operatorname{RBS}}(\mathcal{P}(R)_{\leq N})|;\Z)\longrightarrow H_d(|\mathscr{M}_{\operatorname{RBS}}(\mathcal{P}(R)_{\leq N+1})|; \Z)
\end{align*}
is an isomorphism for $d\leq N-2$ and an epimorphism for $d\leq N-1$ (see also \Cref{questions}).

Recall that Rognes introduces spectrum-level rank filtrations of algebraic K-theory (\cite{Rognes}): this generalises Quillen's rank filtration to rank filtrations of deloopings of the K-theory space. We can interpret the rank filtration of the monoidal $\operatorname{RBS}$-category as a generalisation in the other direction giving a rank filtration that deloops to Quillen's.
\end{example}

\begin{example}\label{free rank filtration}
Suppose $A$ is an associative ring and consider the category $\mathcal{F}(A)$ of finitely generated free $A$-modules. This inherits the structure of an exact category from $\mathcal{P}(A)$. More precisely, an admissible epimorphism is a projection $A^n\twoheadrightarrow A^m$ whose kernel is also \textit{free}, and an admissible monomorphism is an inclusion $A^n\hookrightarrow A^m$ that splits off a \textit{free} module.

This gives rise to a monoidal category $\mathscr{M}_{\operatorname{RBS}}(\mathcal{F}(A))$. If $A$ has invariant basis number, then the rank defines a grading on $\mathcal{F}(A)$ and we can consider the resulting rank filtration of $\mathscr{M}_{\operatorname{RBS}}(\mathcal{F}(A))$.
\end{example}

\begin{notation}\label{RBS-pieces}
Let $\mathcal{C}$ be an exact category with a grading $r\colon \mathcal{C}\rightarrow \N$. We can define an induced grading of the monoidal $\operatorname{RBS}$-category
\begin{align*}
\rho\colon \mathscr{M}_{\operatorname{RBS}}(\mathcal{C})\rightarrow \N, \qquad c=(c_j)_{j\in J}\,\mapsto\, \rho(c):=r(\oplus_{j\in J} c_j)=\sum_{j\in J} r(c_j).
\end{align*}
We call $\rho(c)$ the \textit{total rank} of the object $c$. This is a monoidal functor and for each $N$, we write 
\begin{align*}
\operatorname{RBS}_N(\mathcal{C}):=\rho^{-1}(N)\subset \mathscr{M}_{\operatorname{RBS}}(\mathcal{C})
\end{align*}
for the full subcategory spanned by the objects of total rank $N$. We denote by
\begin{align*}
\partial \operatorname{RBS}_N(\mathcal{C})\subset \operatorname{RBS}_N(\mathcal{C})
\end{align*}
the further full subcategory spanned by the objects $(c_j)_{j\in J}$ with $|J|>1$; in other words, we remove the single object lists on rank $N$ objects in $\mathcal{C}$.
\end{notation}

\begin{observation}\label{RBSN vs reductive borel-serre categories}
Let $\mathcal{C}$ be an exact category equipped with a grading. The unital and non-unital monoidal $\operatorname{RBS}$-categories decompose as disjoint unions
\begin{align*}
\mathscr{M}_{\operatorname{RBS}}(\mathcal{C}) \simeq \coprod_{N\geq 0} \operatorname{RBS}_N(\mathcal{C}),\qquad \mathscr{M}_{\operatorname{RBS}}^{\operatorname{nu}}(\mathcal{C}) \simeq \coprod_{N>0} \operatorname{RBS}_N(\mathcal{C}).
\end{align*}
We also have decompositions
\begin{align*}
\operatorname{RBS}_N(\mathcal{C})\ \simeq \coprod_{c\in C_N} \operatorname{RBS}(c),\qquad \partial\operatorname{RBS}_N(\mathcal{C})\ \simeq \coprod_{c\in C_N} \partial\operatorname{RBS}(c),
\end{align*}
where $c$ runs through a set $C_N$ of representatives of isomorphism classes of rank $N$ objects in $\mathcal{C}$, and where $\operatorname{RBS}(c)$ is the \textit{reductive Borel--Serre category} of $c$ (see \Cref{general RBS-categories}).
\end{observation}

\begin{remark}
To avoid confusion, let us briefly remark that the monoidal rank filtration cannot be recovered from the above grading as the inclusion
\begin{align*}
\rho^{-1}(\N_{\leq N})\subsetneq \mathscr{M}_{\operatorname{RBS}}(\mathcal{C}_{\leq N})
\end{align*}
is proper. Indeed, the monoidal subcategory $\mathscr{M}_{\operatorname{RBS}}(\mathcal{C}_{\leq N})\subset \mathscr{M}_{\operatorname{RBS}}(\mathcal{C})$ contains objects of arbitrarily large total rank; we only require that the individual objects in $\mathcal{C}$ making up the lists have rank less than $N$. For example, considering the exact category of finitely generated $k$-vector spaces, the monoidal category $\mathscr{M}_{\operatorname{RBS}}(\mathcal{P}(k)_{\leq 1})$ has as objects finite (but arbitrarily long) lists of rank $1$ vector spaces:
\begin{align*}
(k,k,\ldots,k),
\end{align*}
whereas $\rho^{-1}(\N_{\leq 1})$ is the full subcategory spanned by $\emptyset$ and $(k)$.

Let's compare the monoidal rank filtration with the natural (non-monoidal) filtration of $\mathcal{C}^\simeq$, the category of objects in $\mathcal{C}$ and isomorphisms between them. Consider for example the exact category $\mathcal{P}(R)$ of finitely generated projective modules over a commutative ring $R$ with connected spectrum. Then
\begin{align*}
\mathcal{P}(R)^\simeq \simeq \coprod_{M\in \mathcal{M}} \operatorname{BGL}(M)
\end{align*}
inherits a grading $\rho\colon \mathcal{P}(R)^\simeq\rightarrow \N$ and can be filtered by
\begin{align*}
\mathcal{P}(R)^\simeq_{\leq N}:=\rho^{-1}(\N_{\leq N})\ \simeq \coprod_{M\in \mathcal{M}_{\leq N}} \operatorname{BGL}(M).
\end{align*}
Here the $N$'th piece of the filtration only captures information about modules of rank $\leq N$. The $N$'th piece of the monoidal rank filtration of $\mathscr{M}_{\operatorname{RBS}}(\mathcal{P}(R))$, however, captures information about modules of arbitrarily large rank since any object $(c_j)_{j\in J}$ admits a canonical map
\begin{align*}
(c_j)_{j\in J}\rightarrow (\oplus_{j\in J}c_j).
\end{align*}
More precisely, the category $\mathscr{M}_{\operatorname{RBS}}(\mathcal{P}(R)_{\leq N})$ captures the splittings of finitely generated projective modules into pieces of rank less than $N$. For example, $\mathscr{M}_{\operatorname{RBS}}(\mathcal{P}(k)_{\leq 1})$ captures all \textit{complete flags} in $k^n$ for varying $n$, that is, flags that cannot be further refined.
\end{remark}

\subsection{The associated graded}\label{associated graded}

Let $\mathcal{C}$ be an exact category equipped with a grading $r\colon \mathcal{C}\rightarrow \N$ and consider the monoidal rank filtration of $\mathscr{M}_{\operatorname{RBS}}(\mathcal{C})$ as defined in the previous section. For notational ease, we write
\begin{align*}
\mathscr{M}:=\mathscr{M}_{\operatorname{RBS}}(\mathcal{C}) \quad\text{and}\quad \mathscr{M}_{\leq N}:=\mathscr{M}_{\operatorname{RBS}}(\mathcal{C}_{\leq N})\quad\text{for any }N\geq 0,
\end{align*}
for the monoidal $\operatorname{RBS}$-categories, and analogously for the non-unital versions:
\begin{align*}
\mathscr{M}^{\operatorname{nu}}:=\mathscr{M}_{\operatorname{RBS}}^{\operatorname{nu}}(\mathcal{C}) \quad\text{and}\quad \mathscr{M}^{\operatorname{nu}}_{\leq N}:=\mathscr{M}_{\operatorname{RBS}}^{\operatorname{nu}}(\mathcal{C}_{\leq N})\quad\text{for any }N\geq 0,
\end{align*}
Likewise, for any $N\geq 0$, we write
\begin{align*}
\operatorname{RBS}_N:=\operatorname{RBS}_N(\mathcal{C}),\quad\text{and}\quad \partial\operatorname{RBS}_N:=\partial\operatorname{RBS}_N(\mathcal{C})
\end{align*}
for the subcategories of \Cref{RBS-pieces}. The aim of this section is to prove the following result.

\begin{theorem}\label{E1-pushout diagram}
For any $N\geq 0$, the following diagram is a pushout in $\operatorname{Mon}(\operatorname{Cat}_\infty)$:
\begin{center}
\begin{tikzpicture}
\matrix (m) [matrix of math nodes,row sep=2em,column sep=2em,nodes={anchor=center}]
{
\operatorname{Free}(\partial \operatorname{RBS}_N) & \operatorname{Free}(\operatorname{RBS}_N) \\
(\mathscr{M}_{\leq N-1})^\otimes & (\mathscr{M}_{\leq N})^\otimes \\
};
\path[-stealth]
(m-1-1) edge (m-1-2) edge (m-2-1)
(m-2-1) edge (m-2-2)
(m-1-2) edge (m-2-2)
;
\end{tikzpicture}
\end{center}
In particular, we have an equivalence
\begin{align*}
(\mathscr{M}_{\leq N})^\otimes/(\mathscr{M}_{\leq N-1})^\otimes \xleftarrow{\ \simeq\ } \operatorname{Free}(\operatorname{RBS}_N/\partial \operatorname{RBS}_N)
\end{align*}
identifying the pushout in $\operatorname{Mon}(\operatorname{Cat}_\infty)$ on the left with the free monoidal $\infty$-category on the pushout $\operatorname{RBS}_N/\partial \operatorname{RBS}_N$ in $\operatorname{Cat}_\infty$.
\end{theorem}

In fact, we will prove the non-unital version of this theorem; the theorem above then follows from the fact that the unitalisation functor $(-)^+$ is a left adjoint. The unital version can be proved directly using the exact same strategy, but we'll use the non-unital version in the subsequent section, so we opt for settling them in this manner. Thus we set out to prove the following theorem.

\begin{theorem}\label{non-unital E1-pushout diagram}
For any $N> 0$, the following diagram is a pushout in $\operatorname{Mon}^{\operatorname{nu}}(\operatorname{Cat}_\infty)$:
\begin{center}
\begin{tikzpicture}
\matrix (m) [matrix of math nodes,row sep=2em,column sep=2em,nodes={anchor=center}]
{
\operatorname{Free}^{\operatorname{nu}}(\partial \operatorname{RBS}_N) & \operatorname{Free}^{\operatorname{nu}}(\operatorname{RBS}_N) \\
(\mathscr{M}_{\leq N-1}^{\operatorname{nu}})^\otimes & (\mathscr{M}_{\leq N}^{\operatorname{nu}})^\otimes \\
};
\path[-stealth]
(m-1-1) edge (m-1-2) edge (m-2-1)
(m-2-1) edge (m-2-2)
(m-1-2) edge (m-2-2)
;
\end{tikzpicture}
\end{center}
\end{theorem}

We will need some decorated variants of the categories of partitioned linearly ordered sets introduced in \S \ref{partitioned linearly ordered sets}. Given an object in $\mathscr{M}^{\operatorname{nu}}_{\leq N}$, we want to keep track of the different ways in which this object can be partitioned such that each ``substring'' belongs to either $\mathscr{M}^{\operatorname{nu}}_{\leq N-1}$ or $\operatorname{Free}^{\operatorname{nu}}(\operatorname{RBS}_N)$. In fact, we need to keep track of such partitions but working with all finite products $(\mathscr{M}^{\operatorname{nu}}_{\leq N})^n$ as specified in the definition below.

\begin{notation}
In the following, we will as usual denote an object in $\mathscr{M}^{\operatorname{nu}}$ by $c=(c_i)_{i\in I}$. If on the other hand, we are considering an $I$-tuple of objects in $\mathscr{M}^{\operatorname{nu}}$, we write $(c^i)_{i\in I}\in (\mathscr{M}^{\operatorname{nu}})^I$, in which case each $c^i$ is of the form $c^i=(c^i_j)_{j\in J_i}\in \mathscr{M}^{\operatorname{nu}}$.
\end{notation}

Recall that $\mathcal{I}_\gg^\circ=\mathcal{I}_\gg\smallsetminus \emptyset_\emptyset$ is the result of removing the isolated object.

\begin{definition}
Let $N>0$ and let $c=(c_j)_{j\in J}$ be an object in $\mathscr{M}^{\operatorname{nu}}_{\leq N}$.
\begin{enumerate}
\item Let $s\colon I\rightarrow P$ be an object in $\mathcal{I}^\circ_\gg$ and let $\pi\colon J\rightarrow I$ be a surjective order preserving map. We say that the sequence $J\xrightarrow{\pi}I\xrightarrow{s} P$ is an \textit{$N$-admissible partition of $c$ (over $I_P=(I_p)_{p\in P}$)} if the following holds: for each $p\in P$, the tuple
\begin{align*}
((c_j)_{j\in \pi^{-1}(i)})_{i\in I_p} \in (\mathscr{M}^{\operatorname{nu}}_{\leq N})^{I_p}
\end{align*}
belongs to either $(\mathscr{M}^{\operatorname{nu}}_{\leq N-1})^{I_p}$ or $\operatorname{Free}^{\operatorname{nu}}(\operatorname{RBS}_N)^{I_p}$. We will often denote such a sequence by $J\rightarrow I_P$.
\item Let $s\colon I\rightarrow P$ be an object in $\mathcal{I}_\gg^\circ$. We say that $I_P$ \textit{admits an $N$-admissible extension with respect to $c$} if there exists a $\pi\colon J\rightarrow I$ such that the resulting sequence $J\rightarrow I\rightarrow P$ is an $N$-admissible partition of $c$. We denote the set of $N$-admissible partitions of $c$ over $I_P$ by
\begin{align*}
\operatorname{ext}_c^N(I_P):=\{J\rightarrow I_P \ N\text{-admissible partition of }c\}.
\end{align*}
\item Let $\mathcal{I}^N_\gg(c)\subseteq \mathcal{I}^\circ_\gg$ denote the full subcategory of objects $I_P$ admitting an $N$-admissible extension with respect to $c$.
\item Let $\widetilde{\mathcal{I}}^N_\gg(c)$ denote the category whose objects are the $N$-admissible partitions of $c$ and where a morphism
\begin{align*}
(J\xrightarrow{\pi} I_P)\rightarrow (J\xrightarrow{\pi'} I'_{P'})
\end{align*}
is given by a morphism $I_P\rightarrow I'_{P'}$ in $\mathcal{I}_\gg$ such that the defining map $\theta\colon I\rightarrow I'$ satisfies $\theta\circ \pi=\pi'$.
\item Let $\Upsilon\colon \widetilde{\mathcal{I}}^N_\gg(c)\rightarrow\mathcal{I}^N_\gg(c)$ denote the projection functor $(J\rightarrow I_P)\mapsto I_P$.
\item Finally, let $\mathcal{I}_\gg^{J,N}(c)\subseteq \mathcal{I}^N_\gg(c)$, respectively $\widetilde{\mathcal{I}}_\gg^{J,N}(c)\subseteq \widetilde{\mathcal{I}}^N_\gg(c)$, denote the full subcategories on objects of the form $J_P$, respectively $J\rightarrow J_P$; that is, requiring $I=J$.\qedhere
\end{enumerate}
\end{definition}

The above definitions give rise to a commutative diagram of $1$-categories
\begin{center}
\begin{tikzpicture}
\matrix (m) [matrix of math nodes,row sep=2em,column sep=2em,nodes={anchor=center}]
{
 & \widetilde{\mathcal{I}}_\gg^{J,N}(c) & \widetilde{\mathcal{I}}^N_\gg(c) & \\
\{J_J\}  & \mathcal{I}_\gg^{J,N}(c) & \mathcal{I}^N_\gg(c) & \mathcal{I}^\circ_\gg\\
};
\path[right hook-stealth]
(m-1-2) edge (m-1-3)
(m-2-1) edge (m-2-2)
(m-2-2) edge (m-2-3)
(m-2-3) edge (m-2-4)
;
\path[-stealth]
(m-1-2) edge (m-2-2)
(m-1-3) edge node[right]{$\Upsilon$} (m-2-3)
;
\end{tikzpicture}
\end{center}

We make some simple observations.

\begin{observation}\label{observation about partition categories wrt c}
Let $N\geq 0$ and let $c=(c_j)_{j\in J}$ be an object in $\mathscr{M}^{\operatorname{nu}}_{\leq N}$.
\begin{enumerate}
\item The object $J_J$ is terminal in $\mathcal{I}_\gg^{J,N}(c)$. In particular, the inclusion $\{J_J\}\hookrightarrow \mathcal{I}_\gg^{J,N}(c)$ is a $\varinjlim$-equivalence.
\item The functor $\Upsilon$ restricts to an equivalence $\widetilde{\mathcal{I}}_\gg^{J,N}(c)\xrightarrow{\simeq} \mathcal{I}_\gg^{J,N}(c)$.
\item The inclusion $\widetilde{\mathcal{I}}_\gg^{J,N}(c) \hookrightarrow \widetilde{\mathcal{I}}^N_\gg(c)$ admits a right adjoint
\begin{align*}
(J\xrightarrow{\pi} I\xrightarrow{s} P)\mapsto (J\xrightarrow{\operatorname{id}} J\xrightarrow{s\circ \pi}P).
\end{align*}
\item Combining the above, we see that $|\widetilde{\mathcal{I}}^N_\gg(c)|\simeq \ast$.\qedhere
\end{enumerate}
\end{observation}

We will be applying proper descent to analyse colimits in the proof of \Cref{E1-pushout diagram} (see \cite[Theorem 2.29]{ClausenOrsnesJansen}). To do this, we'll need the following lemma.

\begin{lemma}\label{concatenation functor is proper}
For any surjective order preserving map $\theta\colon I\rightarrow J$, $J\neq \emptyset$, the concatenation functor
\begin{align*}
(\mathscr{M}^{\operatorname{nu}}_{\leq N})^I\rightarrow (\mathscr{M}^{\operatorname{nu}}_{\leq N})^J, \quad (c^i)_{i\in I}\mapsto (\circledast_{i\in \theta^{-1}(j)}c^i)_{j\in J},
\end{align*}
is proper.
\end{lemma}
\begin{proof}
Since the functor splits as a product over $J$, it suffices to show that for each $j\in J$, the concatenation functor
\begin{align*}
(\mathscr{M}_{\leq N}^{\operatorname{nu}})^{\theta^{-1}(j)}\rightarrow (\mathscr{M}_{\leq N}^{\operatorname{nu}})^{\{j\}}
\end{align*}
is proper. Hence, we may assume that $J=\ast$ and thus $\theta\colon I\rightarrow \ast$ is the unique map to the terminal object. Fix an object $d=(d_l)_{l\in L}$ in $\mathscr{M}_{\leq N}^{\operatorname{nu}}$ and consider the inclusion of the fibre into the right fibre:
\begin{align*}
F_d\rightarrow F_{d/}.
\end{align*}
We claim that this functor admits a right adjoint. This implies that it is a $\varprojlim$-equivalence (\cite[Example 2.21]{ClausenOrsnesJansen}) and hence that the concatenation functor is proper (\cite[Definition 2.22]{ClausenOrsnesJansen}). Given an object $((c^i)_{i\in I}, \phi \colon d\rightarrow \circledast_{i\in I}c^i)$ in $F_{d/}$ with $c^i=(c^i_h)_{h\in H^i}$, write $H:=\circledast_{i\in I}H^i$ and let $\pi\colon L\rightarrow H$ be the order preserving map given by the morphism $\phi$. The right adjoint sends this object to
\begin{align*}
((d^i)_{i\in I}, \operatorname{id} \colon d\rightarrow \circledast_{i\in I}d^i)\qquad \text{with} \quad d^i=(d_l)_{l\in \pi^{-1}(H^i)}.
\end{align*}
In other words, we just split the object $d$ into a tuple over $I$ as specified by the data given by the morphism $\phi$. We leave the remaining details to the reader.
\end{proof}

We are now ready to prove \Cref{E1-pushout diagram}: we will prove that the given diagram is a pushout in $\operatorname{Mon}^{\operatorname{nu}}(\operatorname{Cat}_\infty)$ by determining an explicit formula for the pushout in question and then proving that the induced comparison map is an equivalence. We obtain the explicit formula by  considering the diagram in $\operatorname{Fun}(\Delta^{\operatorname{op}}_i, \operatorname{Cat}_\infty)$ via the inclusion
\begin{align*}
\mathbb{B}^{\operatorname{nu}}\colon \operatorname{Mon}^{\operatorname{nu}}(\operatorname{Cat}_\infty)\hookrightarrow \operatorname{Fun}(\Delta^{\operatorname{op}}_i, \operatorname{Cat}_\infty)
\end{align*}
and then mapping it back into $\operatorname{Mon}^{\operatorname{nu}}(\operatorname{Cat}_\infty)$ via the left adjoint $\mathbb{L}^{\operatorname{nu}}$ to $\mathbb{B}^{\operatorname{nu}}$ for which we have a concrete computable expression (\Cref{non-unital monoid over Igg}). 

\begin{proof}[Proof of \Cref{E1-pushout diagram}]
We begin by determining the pushout in $\operatorname{Fun}(\Delta^{\operatorname{op}}_i, \operatorname{Cat}_\infty)$.
Note first of all that the concatenation functor
\begin{align*}
\coprod_{k>0}(\operatorname{RBS}_N)^k\rightarrow \mathscr{M}_{\leq N}^{\operatorname{nu}}
\end{align*}
identifies the underlying category of $\operatorname{Free}^{\operatorname{nu}}(\operatorname{RBS}_N)$ as a left closed full subcategory of $\mathscr{M}_{\leq N}^{\operatorname{nu}}$ (since a partition of an object into substrings of total rank $N$ must necessarily be unique). Likewise, $\coprod_{k> 0}(\partial\operatorname{RBS}_N)^k$ and $\mathscr{M}_{\leq N-1}^{\operatorname{nu}}$ are left closed full subcategories of $\mathscr{M}_{\leq N}^{\operatorname{nu}}$. Let $n\geq 0$, consider the $n$'fold products of these categories and the resulting pushout in $\operatorname{Cat}_\infty$:
\begin{center}
\begin{tikzpicture}
\matrix (m) [matrix of math nodes,row sep=2em,column sep=2em,nodes={anchor=center}]
{
\big(\coprod_{k>0}(\partial\operatorname{RBS}_N)^k \big)^n & \big(\coprod_{k> 0}(\operatorname{RBS}_N)^k \big)^n \\
(\mathscr{M}_{\leq N-1}^{\operatorname{nu}})^n & P(n) \\
};
\path[-stealth]
(m-1-1) edge (m-1-2) edge (m-2-1)
(m-2-1) edge (m-2-2)
(m-1-2) edge (m-2-2)
;
\end{tikzpicture}
\end{center}
It follows by descent for left closed covers (\cite[Corollary 2.33]{ClausenOrsnesJansen}) that this pushout identifies with the union of these categories in $(\mathscr{M}_{\leq N}^{\operatorname{nu}})^n$. Let $P$ denote the resulting simplicial $\infty$-category. Then the following diagram is a pushout diagram in $\operatorname{Fun}(\Delta^{\operatorname{op}}_i, \operatorname{Cat}_\infty)$.
\begin{center}
\begin{tikzpicture}
\matrix (m) [matrix of math nodes,row sep=2em,column sep=2em,nodes={anchor=center}]
{
\mathbb{B}^{\operatorname{nu}}\operatorname{Free}^{\operatorname{nu}}(\partial \operatorname{RBS}_N) & \mathbb{B^{\operatorname{nu}}}\operatorname{Free}^{\operatorname{nu}}(\operatorname{RBS}_N) \\
\mathbb{B}^{\operatorname{nu}}(\mathscr{M}_{\leq N-1}^{\operatorname{nu}})^\otimes & P \\
};
\path[-stealth]
(m-1-1) edge (m-1-2) edge (m-2-1)
(m-2-1) edge (m-2-2)
(m-1-2) edge (m-2-2)
;
\end{tikzpicture}
\end{center}

Applying $\mathbb{L}^{\operatorname{nu}}$ to this yields a pushout diagram in $\operatorname{Mon}^{\operatorname{nu}}(\operatorname{Cat}_\infty)$ as below since $\mathbb{L}^{\operatorname{nu}}$ is left adjoint and $\mathbb{L}^{\operatorname{nu}}\mathbb{B}^{\operatorname{nu}}\simeq \operatorname{id}$.
\begin{center}
\begin{tikzpicture}
\matrix (m) [matrix of math nodes,row sep=2em,column sep=2em,nodes={anchor=center}]
{
\operatorname{Free}^{\operatorname{nu}}(\partial \operatorname{RBS}_N) & \operatorname{Free}^{\operatorname{nu}}(\operatorname{RBS}_N) \\
(\mathscr{M}_{\leq N-1}^{\operatorname{nu}})^\otimes & \mathbb{L}^{\operatorname{nu}}P \\
};
\path[-stealth]
(m-1-1) edge (m-1-2) edge (m-2-1)
(m-2-1) edge (m-2-2)
(m-1-2) edge (m-2-2)
;
\end{tikzpicture}
\end{center}

To show that the induced comparison map of monoidal $\infty$-categories,
\begin{align*}
\mathbb{L}^{\operatorname{nu}}P\longrightarrow (\mathscr{M}_{\leq N}^{\operatorname{nu}})^\otimes,
\end{align*}
is an equivalence, we send it back into $\operatorname{Fun}(\Delta^{\operatorname{op}}_i, \operatorname{Cat}_\infty)$ via the inclusion $\mathbb{B}^{\operatorname{nu}}$. By \Cref{non-unital monoid over Igg}, the comparison map $\mathbb{B}^{\operatorname{nu}}\mathbb{L}^{\operatorname{nu}}P\rightarrow \mathbb{B}^{\operatorname{nu}}(\mathscr{M}_{\leq N}^{\operatorname{nu}})^\otimes$ is determined by the map
\begin{align*}
\mathop{\operatorname{colim}}_{\mathcal{I}_\gg^\circ} \mathbf{P}\rightarrow \mathscr{M}_{\leq N}^{\operatorname{nu}}
\end{align*}
of underlying $\infty$-categories induced by the concatenation functors
\begin{align*}
\mathbf{P}(I_P)=\prod_{p\in P} P(I_p^\pm) \longrightarrow \mathscr{M}_{\leq N}^{\operatorname{nu}},\quad \left((c^i)_{i\in I_p}\right)_{p\in P}\mapsto \circledast_{p\in P}\circledast_{i \in I_p} c^i.
\end{align*}
We will show that this is an equivalence using proper descent (\cite[Theorem 2.29]{ClausenOrsnesJansen}). To this end consider the cocone diagram $(\mathcal{I}^\circ_\gg)^\triangleright\rightarrow\operatorname{Cat}_\infty$ given by $\mathbf{P}$ with cocone point $\mathscr{M}_{\leq N}^{\operatorname{nu}}$. First of all note that for any morphism $\tau=(\theta,\rho)\colon I_P\rightarrow J_Q$, we have a commutative diagram as below, where the lower horizontal map is the concatenation functor induced by the map $\theta\colon I \rightarrow J$. The lower horizontal concatenation functor is proper by \Cref{concatenation functor is proper}.
\begin{center}
\begin{tikzpicture}
\matrix (m) [matrix of math nodes,row sep=2em,column sep=2em,nodes={anchor=center}]
{
\mathbf{P}(I_P) & \mathbf{P}(J_Q) \\
\prod_{p\in P} P(I_p^\pm) & \prod_{q\in Q} P(J_q^\pm) \\
(\mathscr{M}_{\leq N}^{\operatorname{nu}})^I & (\mathscr{M}_{\leq N}^{\operatorname{nu}})^J \\
};
\path[-stealth]
(m-1-1) edge (m-1-2)
(m-2-1) edge (m-2-2)
(m-1-2) edge (m-2-2)
(m-3-1) edge (m-3-2)
;
\path[-]
(m-1-1) edge[double equal sign distance] (m-2-1)
(m-1-2) edge[double equal sign distance] (m-2-2)
;
\path[right hook-stealth]
(m-2-1) edge (m-3-1)
(m-2-2) edge (m-3-2)
;
\end{tikzpicture}
\end{center}
Since $\mathbf{P}(I_P)\subseteq (\mathscr{M}_{\leq N}^{\operatorname{nu}})^I$ is a left closed full subcategory, the inclusion is proper (\cite[Example 2.24 (3)]{ClausenOrsnesJansen}). Hence, the composite $\mathbf{P}(I_P)\rightarrow (\mathscr{M}_{\leq N}^{\operatorname{nu}})^J$ is also proper, and since $\mathbf{P}(J_Q)\subseteq (\mathscr{M}_{\leq N}^{\operatorname{nu}})^J$ is a full subcategory, it follows directly that the functor $\mathbf{P}(I_P) \rightarrow \mathbf{P}(J_Q)$ is proper. Likewise, we see that for any $I_P$ in $\mathcal{I}_\gg$, the concatenation functor
\begin{align*}
\mathbf{P}(I_P)\rightarrow \mathscr{M}_{\leq N}^{\operatorname{nu}}
\end{align*}
to the cocone point is also proper.

It remains to be seen that for any object $c$ in $\mathscr{M}_{\leq N}^{\operatorname{nu}}$, the pullback
\begin{align*}
(\mathcal{I}^\circ_\gg)^\triangleright \rightarrow \operatorname{Cat}_\infty,\qquad I_P\mapsto \mathbf{P}(I_P)\mathop{\times}_{\mathscr{M}_{\leq N}^{\operatorname{nu}}} \{c\},\quad \triangleright \mapsto \{c\}
\end{align*}
is a colimit diagram. Fix $c=(c_j)_{j\in J}$ in $\mathscr{M}_{\leq N}^{\operatorname{nu}}$. It is straightforward to verify that the resulting diagram $\mathcal{I}^\circ_\gg\rightarrow \operatorname{Cat}_\infty$ sends $I_P$ to the set of $N$-admissible partitions of $c$ over $I_P$:
\begin{align*}
I_P\ \mapsto\ \operatorname{ext}_c^N(I_P):=\{J\rightarrow I_P \ N\text{-admissible partition of }c\}.
\end{align*}
This in turn identifies with the left Kan extension of the terminal functor $\ast\colon \widetilde{\mathcal{I}}^N_\gg(c)\rightarrow \operatorname{Cat}_\infty$ along the composite
\begin{align*}
\widetilde{\mathcal{I}}^N_\gg(c)\xrightarrow{\Upsilon} \mathcal{I}^N_\gg(c)\hookrightarrow \mathcal{I}^\circ_\gg.
\end{align*}
Indeed, the comma category over $I_P$ is empty if $I_P$ does not belong to $\mathcal{I}^N_\gg(c)$ in which case the functor evaluates to the empty category $\emptyset$. If $I_P$ does belong to $\mathcal{I}^N_\gg(c)$, then the left Kan extension evaluates to
\begin{align*}
\mathop{\operatorname{colim}}_{\Upsilon_{/I_P}}\ast \simeq |\Upsilon_{/I_P}|\simeq \operatorname{ext}_c^N(I_P)
\end{align*} 
where the first equivalence is a general fact (see e.g. \cite[Corollary 2.10 (4)]{ClausenOrsnesJansen}) and the second follows by observing that each component in $\Upsilon_{/I_P}$ has a terminal object and that these terminal objects are exactly the $N$-admissible partitions of $c$ over $I_P$. It follows that
\begin{align*}
\mathop{\operatorname{colim}}_{I_P\in \mathcal{I}^\circ_\gg} \big(\mathbf{P}(I_P)\mathop{\times}_{\mathscr{M}_{\leq N}^{\operatorname{nu}}} \{c\}\big)\simeq \mathop{\operatorname{colim}}_{I_P\in \mathcal{I}^\circ_\gg}\operatorname{ext}_c^N(I_P) \simeq \mathop{\operatorname{colim}}_{\widetilde{\mathcal{I}}^N_\gg(c)}\ast \simeq |\widetilde{\mathcal{I}}^N_\gg(c)|
\end{align*}
and by \Cref{observation about partition categories wrt c}, we have $|\widetilde{\mathcal{I}}^N_\gg(c)|\simeq \ast$ as desired.

We conclude by proper descent (\cite[Theorem 2.29]{ClausenOrsnesJansen}), that the map
\begin{align*}
\mathop{\operatorname{colim}}_{\mathcal{I}^\circ_\gg} \mathbf{P}\rightarrow \mathscr{M}_{\leq N}^{\operatorname{nu}}
\end{align*}
is indeed an equivalence, and hence that the map of monoidal $\infty$-categories
\begin{align*}
\mathbb{L}^{\operatorname{nu}}P\longrightarrow (\mathscr{M}_{\leq N}^{\operatorname{nu}})^\otimes
\end{align*}
is an equivalence identifying $(\mathscr{M}_{\leq N}^{\operatorname{nu}})^\otimes$ as the desired pushout.
\end{proof}

As an immediate corollary, we get an analogous pushout squares of monoids in spaces. We will state the non-unital version, but the unital variant reads analogously.

\begin{corollary}\label{pushout of non-unital E1-spaces}
For any $N\geq 0$, the following diagram is a pushout in $\operatorname{Mon}^{\operatorname{nu}}(\mathcal{S})$:
\begin{center}
\begin{tikzpicture}
\matrix (m) [matrix of math nodes,row sep=2em,column sep=2em,nodes={anchor=center}]
{
\operatorname{Free}^{\operatorname{nu}}(\vert\partial \operatorname{RBS}_N\vert) & \operatorname{Free}^{\operatorname{nu}}(\vert\operatorname{RBS}_N\vert) \\
\vert\mathscr{M}^{\operatorname{nu}}_{\leq N-1}\vert & \vert\mathscr{M}^{\operatorname{nu}}_{\leq N}\vert \\
};
\path[-stealth]
(m-1-1) edge (m-1-2) edge (m-2-1)
(m-2-1) edge (m-2-2)
(m-1-2) edge (m-2-2)
;
\end{tikzpicture}
\end{center}
In particular, we have an equivalence
\begin{align*}
|\mathscr{M}^{\operatorname{nu}}_{\leq N}|/|\mathscr{M}^{\operatorname{nu}}_{\leq N-1}| \xleftarrow{\ \simeq\ } \operatorname{Free}^{\operatorname{nu}}(|\operatorname{RBS}_N|/|\partial \operatorname{RBS}_N|)
\end{align*}
identifying the pushout in $\operatorname{Mon}^{\operatorname{nu}}(\mathcal{S})$ on the left with the free (non-unital) $\mathbb{E}_1$-space on the homotopy quotient $|\operatorname{RBS}_N|/|\partial \operatorname{RBS}_N|$.
\end{corollary}
\begin{proof}
Taking geometric realisation is a left adjoint and it commutes with taking free monoidal objects (\Cref{observations about free objects}).
\end{proof}

\subsection{Tits complexes}

The associated graded of the monoidal rank filtration as identified in the previous section can be unravelled even further when we consider the exact category of finitely generated projective modules over a ring. To this end we introduce the Tits complex of a finitely generated projective module.

\begin{definition}\label{Tits complex}
Let $A$ be an associative ring and let $M$ be a projective $A$-module and let $\mathcal{P}(M)$ denote the poset of non-zero splittable submodules of $M$. We define the \textit{Tits complex} of $M$ to be the realisation of the poset of proper non-zero splittable submodules of $M$
\begin{align*}
\mathcal{T}(M):=|\mathcal{P}(M)\smallsetminus M|
\end{align*}
with the natural action of $\operatorname{GL(M)}$. We write $\mathcal{T}_A(M)$ if we need to stress the underlying ring.
\end{definition}

\begin{observation}\label{flags or submodules}
Let $A$ be an associative ring and let $M$ be a projective $A$-module and let $\mathcal{F}(M)$ denote the poset of splittable flags in $M$. Note that the Tits complex $\mathcal{T}(M)$ is equivalent to the realisation of the poset of non-empty splittable flags in $M$:
\begin{align*}
\mathcal{T}(M)\simeq |\mathcal{F}(M)\smallsetminus (0\subsetneq M)|.
\end{align*}
Indeed, $\mathcal{F}(M)$ is simply the subdivision of $\mathcal{P}(M)$, so in terms of realisations, passing to flags means passing to the barycentric subdivision.
\end{observation}

\begin{remark}\label{Scalamandre}
In \cite{Scalamandre}, Scalamandre introduces a slightly smaller generalisation of Tits buildings for a general commutative ring $R$ by considering direct summands $V$ of $R^n$ such that both $V$ and $R^n/V$ are free and proves a Solomon-Tits theorem for rings of stable rank 2 and for Dedekind domains of arithmetic type: in these cases his Tits complex is $(n-2)$-spherical. We will write $\mathcal{T}^f(W)$ for this complex where $W$ is a finitely generated free module over a commutative ring $R$ and we'll call it the \textit{free Tits complex}. We want to work with the slightly larger version $\mathcal{T}(M)$ of all direct summands as it seems to fit better into our general setup.
\end{remark}

\begin{example}\label{examples}\ 
\begin{enumerate}
\item If $R$ is such that finitely generated projective modules are free then the Tits complex agrees with the free Tits complex, e.g. if $R$ is a PID or a local commutative ring.
\item Let $K$ be a field and $V$ a $K$-vector space of rank $n$. Then $\mathcal{T}(V)=\mathcal{T}^f(V)=T_n(K)$ is the usual Tits building over $K$.
\item Let $R$ be a Dedekind domain with field of fractions $K$, and let $M$ be a finitely generated projective $R$-module. Extension by scalars $-\otimes_RK$ defines an isomorphism
\begin{align*}
\mathcal{T}_R(M)\xrightarrow{\cong} \mathcal{T}_K(M\otimes_R K)
\end{align*}
with inverse given by intersecting with $M\cong M\otimes_R R\subseteq M\otimes_R K$ --- this is due to the fact that over a Dedekind domain a finitely generated module is projective if and only if it is torsion free (see also \cite[Remark 2.8]{MillerPatztWilsonYasaki}).\qedhere 

\end{enumerate}
\end{example}

\begin{lemma}\label{cofibre as homotopy quotient of tits complex}
Let $A$ be an associative ring and let $M$ be a finitely generated projective $A$-module. Then the cofibre of the inclusion
\begin{align*}
|\partial \operatorname{RBS}(M)|\hookrightarrow |\operatorname{RBS}(M)|
\end{align*}
identifies with the (homotopy) quotient of the once-suspended Tits complex of $M$:
\begin{align*}
|\operatorname{RBS}(M)|/|\partial\operatorname{RBS}(M)|\simeq \big(\Sigma|\mathcal{T}(M)| \big)/\!\!/\operatorname{GL}(M).
\end{align*}
\end{lemma}
\begin{proof}
This is done for a finite vector space over a finite field in the course of the proof of \cite[Theorem 5.18]{ClausenOrsnesJansen},  but the identification is completely general: by \cite[Corollary 5.8]{ClausenOrsnesJansen}, we have a pushout square in $\mathcal{S}$ as below (note that the poset of flags is denoted by $\mathcal{P}$ there and that the split Noetherian hypothesis on $M$ is not actually needed for this part of the result).
\begin{center}
\begin{tikzpicture}
\matrix (m) [matrix of math nodes,row sep=2em,column sep=2em,nodes={anchor=center}]
{
\vert\mathcal{F}(M)\smallsetminus (0\subsetneq M)\vert /\!\!/\operatorname{GL}(M) & \vert\mathcal{F}(M)\vert /\!\!/\operatorname{GL}(M) \\
\vert\partial \operatorname{RBS}(M)\vert & \vert\operatorname{RBS}(M)\vert \\
};
\path[-stealth]
(m-1-1) edge (m-1-2) edge (m-2-1)
(m-2-1) edge (m-2-2)
(m-1-2) edge (m-2-2)
;
\end{tikzpicture}
\end{center}
Thus we have identifications
\begin{align*}
|\operatorname{RBS}(M)|/|\partial\operatorname{RBS}(M)|&\simeq (|\mathcal{F}(M)|/|\mathcal{F}(M)\smallsetminus (0\subsetneq M)|)/\!\!/\operatorname{GL}(M) \\
&\simeq \Sigma|\mathcal{F}(M)\smallsetminus (0\subsetneq M)|/\!\!/\operatorname{GL}(M)
\end{align*}
which by \Cref{flags or submodules} gives us the desired identification.
\end{proof}

Combined with \Cref{pushout of non-unital E1-spaces}, we see that for a commutative ring $R$ with connected spectrum, the cofibre of the map
\begin{align*}
|\mathscr{M}_{\operatorname{RBS}}(\mathcal{P}(R)_{\leq N-1})|\longrightarrow |\mathscr{M}_{\operatorname{RBS}}(\mathcal{P}(R)_{\leq N})|
\end{align*}
in $\operatorname{Mon}(\mathcal{S})$ identifies with the free $\mathbb{E}_1$-space
\begin{align*}
\operatorname{Free}\bigg(\coprod_{M\in \mathcal{M}_N}\big(\Sigma|\mathcal{T}(M)| \big)/\!\!/\operatorname{GL}(M)\bigg)
\end{align*}
on the homotopy quotients of the once-suspended Tits complexes, where $M$ runs through a set of representatives of finitely generated projective $R$-modules of rank $N$. Likewise for the non-unital versions. We will use this identification in the following section where we approach the question of homological stability: we'll restrict our attention to rings $R$ for which the homology of the Tits complexes is concentrated in a single degree in which case we get an analogue of the usual Steinberg module.

\begin{definition}\label{standard connectivity estimate}
Let $R$ be a commutative ring with connected spectrum. 
\begin{enumerate}
\item  We say that $R$ satisfies the \textit{standard connectivity estimate with respect to the Tits complex} if for any finitely generated projective $R$-module $M$ of rank $n$, the reduced homology of the Tits complex $\mathcal{T}(M)$ is concentrated in degree $n-2$.
\item In this case, we define the \textit{Steinberg module (with coefficients in $\Bbbk$)} of a finitely generated projective $R$-module $M$ of rank $n$ to be the homology group
\begin{align*}
\operatorname{St}(M;\Bbbk):=\tilde{H}_{n-2}(\mathcal{T}(M); \Bbbk)
\end{align*}
with the induced $\operatorname{GL}(M)$-action.\qedhere
\end{enumerate}
\end{definition}

\begin{example} \ 
\begin{itemize}
\item If $R$ is a field, then it satifies the standard connectivity estimate with respect to the Tits complex by the Solomon-Tits Theorem, which says that the usual Tits building over field $T_n(K)=\mathcal{T}(K^n)$ is homotopy equivalent to a wedge of $(n-2)$-spheres (see e.g. \cite{Quillen72}).
\item If $R$ is a Dedekind domain (e.g. a PID), then it satifies the standard connectivity estimate with respect to the Tits complex by combining the observations of \Cref{examples} with the Solomon-Tits Theorem.
\item If $R$ is a local commutative ring, then it satisfies the standard connectivity estimate with respect to the Tits complex by combining the observations of \Cref{examples} with Scalamandre's Solomon-Tits Theorem (\cite[Theorem A]{Scalamandre}).\qedhere
\end{itemize}
\end{example}

\section{\texorpdfstring{$\mathbb{E}_1$}{E1}-homology and homological stability}\label{homological stability}

In order to investigate the homological stability behaviour of the $\operatorname{RBS}$-categories, we calculate its $\mathbb{E}_1$-homology. This is easily done given the analysis of the monoidal rank filtration of the previous section. Given the existence of $\mathbb{E}_\infty$-structure on $|\mathscr{M}_{\operatorname{RBS}}(R)|$ (in fact, we'll only be using $\mathbb{E}_2$) and the calculation of its $\mathbb{E}_1$-homology, we can appeal to the framework of cellular $\mathbb{E}_k$-algebras established by Galatius--Kupers--Randall-Williams (\cite{GalatiusKupersRandalWilliams}). We'll pick out these results at the end of this section. Even this somewhat limited investigation reveals that the $\operatorname{RBS}$-categories satisfy a priori better homological stability properties than the general linear groups.

\subsection{Square zero extension and indecomposables}

We recall the invariants in question. First of all, we make some very general observations. Let $\mathcal{C}$ be a symmetric monoidal $\infty$-category and let $\mathcal{O}^\otimes$ be an $\infty$-operad whose underlying space is a contractible Kan complex, we can identify $\mathcal{O}$-algebra objects in $\mathcal{C}$ with $\mathcal{O}$-algebra objects in the associated $\mathcal{O}$-monoidal $\infty$-category given by the fibre product $\mathcal{D}^\otimes:=\mathcal{C}^\otimes \times_{\operatorname{Fin}_\ast}\mathcal{O}^\otimes$:
\begin{align*}
\operatorname{Alg}_{\mathcal{O}}(\mathcal{C})\simeq \operatorname{Alg}_{/\mathcal{O}}(\mathcal{D}).
\end{align*}
Hence, the forgetful functor
\begin{align*}
u\colon \operatorname{Alg}_{\mathcal{O}}(\mathcal{C})\rightarrow \mathcal{C}
\end{align*}
admits a left adjoint $f$ (\cite[Example 3.1.3.6]{LurieHA}) and this adjunction is monadic by the Monadicity Theorem (\cite[Theorem 4.7.3.5]{LurieHA}) combined with \cite[Lemma 3.2.2.6 and Proposition 3.2.3.1]{LurieHA}.

\begin{notation}\ 
\begin{enumerate}
\item Let $\mathcal{S}^\N:=\operatorname{Fun}(\N,\mathcal{S})$ denote the symmetric monoidal $\infty$-category given by equipping the functor $\infty$-category with the Day convolution structure (\cite[Example 2.2.6.17]{LurieHA}).
\item For $k\geq 1$, consider the non-unitary $\mathbb{E}_k$-operad $\mathbb{E}_k^{\operatorname{nu}}$ (\cite[Definition 5.4.4.1]{LurieHA}) and its associated monad $T_{\mathbb{E}_k^{\operatorname{nu}}}$ on $\mathcal{S}^\N$.
\item Note that the adjunction $\mathcal{S}^\N_*\leftrightarrows \mathcal{S}^\N$ with left adjoint given by by adding a disjoint basepoint at each level is also monadic (\cite[Theorem 4.7.3.5]{LurieHA}) and denote by $(-)_+$ the associated monad on $\mathcal{S}^\N$, $X(n)\mapsto X(n)_+$, $n\in \N$. (See also \cite[p.825]{LurieHA}).\qedhere
\end{enumerate}
\end{notation}

\begin{remark}
We follow the convention of \cite{GalatiusKupersRandalWilliams} and \cite{HorelKrannichKupers}, calling an $\infty$-operad \textit{non-unitary} if it factors over $\operatorname{Surj}_*\hookrightarrow \operatorname{Fin}_*$ (the subcategory containing all objects but only the surjective maps). This does not align with \cite{LurieHA} where such an $\infty$-operad is called \textit{non-unital}, but we hope no confusion will occur. See also \cite[Remark 1.2]{HorelKrannichKupers}.
\end{remark}

There is a canonical map of monads $T_{\mathbb{E}_k^{\operatorname{nu}}}\rightarrow (-)_+$ sending all $n$-ary operations, $n\geq 2$, to the basepoint (\cite{GalatiusKupersRandalWilliams}). The resulting change-of-monad adjunction reads
\begin{align*}
\operatorname{Alg}_{\mathbb{E}_k}^{\operatorname{nu}}(\mathcal{S}^\N)\mathop{\rightleftarrows}^{Q^k}_{Z^k} \mathcal{S}_*^\N
\end{align*}
(see \cite[Proposition 4.6.2.17]{LurieHA}). We call $Q^k$ the $\mathbb{E}_k$-\textit{indecomposables functor}.

\begin{remark}
One should interpret the right adjoint $Z^k$ as a trivial algebra functor. Informally, $Z^k$ equips an $\N$-graded pointed space $X$ with the algebra structure such that all $n$-ary operations, $n>1$, factor through the basepoints.
\end{remark}

We have the following definition (see \cite[Definition 10.7]{GalatiusKupersRandalWilliams}).

\begin{definition}
Let $X$ be a non-unitary $\mathbb{E}_k$-algebra in $\mathcal{S}^\N$. The $\mathbb{E}_k$-homology of $X$ with coefficients in a commutative ring $\Bbbk$ is the bigraded ring
\begin{align*}
H^{\mathbb{E}_k}_{n,d}:=\tilde{H}_{n,d}(Q^kX;\Bbbk)=\tilde{H}_d(Q^kX(n);\Bbbk)
\end{align*}
obtained by taking the degreewise reduced singular homology of the $\mathbb{E}_k$-indecomposables $Q^kX$ in $\mathcal{S}_\ast^\N$.
\end{definition}

It is easy to see what the indecomposables of a free non-unitary $\N$-graded $\mathbb{E}_k$-algebra is.

\begin{observation}\label{Indecomposables of free things}
We have a diagram of adjunctions where the lower line is the forgetful functor with left adjoint given by disjoint basepoint.
\begin{center}
\begin{tikzpicture}
\matrix (m) [matrix of math nodes,row sep=3em,column sep=3em,nodes={anchor=center}]
{
 & \operatorname{Alg}_{\mathbb{E}_k}^{\operatorname{nu}}(\mathcal{S}^\N) \\
 \mathcal{S}^\N_\ast & \mathcal{S}^\N \\
};
\path[-stealth]
(m-2-1) edge node[below]{$\scriptstyle Z^k$} (m-1-2)
(m-1-2.201) edge node[above]{$\scriptstyle Q^k$} node[right, rotate=-53]{$\scriptstyle\!\!\dashv$} (m-2-1.82)
(m-2-1.20) edge (m-2-2.160)
(m-2-2.200) edge node[below]{$\scriptstyle{(-)}_+$} node[right, rotate=90]{$\scriptstyle\!\!\dashv$} (m-2-1.340)
(m-1-2.250) edge node[left]{$\scriptstyle u$} node[right]{$\scriptstyle\!\!\dashv$} (m-2-2.115)
(m-2-2.65) edge node[right]{$\scriptstyle \operatorname{Free}^{\operatorname{nu},k}$}(m-1-2.291)
;
\end{tikzpicture}
\end{center}

Since the right adjoints clearly commute, we must have
\begin{align*}
Q^k\circ \operatorname{Free}^{\operatorname{nu},k}\simeq (-)_+.
\end{align*}

It follows directly from this that the $\mathbb{E}_k$-homology of a free non-unitary $\N$-graded $\mathbb{E}_k$-space is
\begin{equation*}
H^{\mathbb{E}_k}_{n,d}(\operatorname{Free}^{\operatorname{nu},k}(X);\Bbbk)\cong H_{n,d}(X; \Bbbk).\qedhere
\end{equation*}
\end{observation}

The following spectral sequence is just the usual spectral sequence associated to a filtered space (see e.g. \cite{HatcherSS}) and it follows directly from the fact that $Q^k$ preserves colimits. This is essentially just Theorem 10.15 in \cite{GalatiusKupersRandalWilliams}.

\begin{proposition}\label{general spectral sequence}
Let $X$ be an $\N$-graded non-unitary $\mathbb{E}_k$-space and $\Bbbk$ a commutative ring. Suppose we have a sequence of $\N$-graded non-unitary $\mathbb{E}_k$-spaces
\begin{align*}
\emptyset=X_{-1}\rightarrow X_0\rightarrow X_1 \rightarrow \cdots \rightarrow X_i\rightarrow X_{i+1}\rightarrow \cdots
\end{align*}
with $X\simeq \operatorname{colim} X_i$. Then for every $n\in \N$, there is a spectral sequence
\begin{align*}
E^1_{pq}(n)=H_{n,p+q}(Q^kX_p,Q^kX_{p-1};\Bbbk), \quad p,q\in \Z,\ p\geq 0,
\end{align*}
whose limit term, if it exists, is $H_{n,p+q}^{\mathbb{E}_k}(X;\Bbbk)$.
\end{proposition}

This spectral sequence will be particularly useful when we can replace the cofibre of each $X_i\rightarrow X_{i+1}$ by a cofibre of free non-unitary $\mathbb{E}_k$-spaces. In the case of the monoidal $\operatorname{RBS}$-category, we'll invoke an $\N$-graded version of \Cref{pushout of non-unital E1-spaces}.

\begin{observation}\label{monoidal grothendieck}
By \cite{Ramzi}, the usual straightening-unstraightening equivalence can be lifted to an equivalence of symmetric monoidal $\infty$-categories:
\begin{align*}
\mathcal{S}^\N\simeq \mathcal{S}_{/\N}
\end{align*}
where the comma category on the right hand side is equipped with the symmetric monoidal structure exhibited in \cite[\S 2.2.2]{LurieHA}. It follows that for any $\infty$-operad $\mathcal{O}$, we have an equivalence
\begin{align*}
\operatorname{Alg}_{\mathcal{O}}(\mathcal{S}^{\N})\simeq \operatorname{Alg}_{\mathcal{O}}(\mathcal{S}_{/\N})\xrightarrow{\simeq} \operatorname{Alg}_{\mathcal{O}}(\mathcal{S})_{/\N}
\end{align*}
where the second equivalence follows from \cite[Lemma 12.2]{AntolinCamarenaBarthel} restricting $\N$ to an $\mathcal{O}$-algebra object in $\mathcal{S}$. Hence, for any $k\geq 0$, the adjunction $Q^k\dashv Z^k$ takes the form
\begin{align*}
\operatorname{Alg}_{\mathbb{E}_k}^{\operatorname{nu}}(\mathcal{S})_{/\N}\ \mathop{\rightleftarrows} \ \mathcal{S}_\ast^{\N}.
\end{align*}
In particular, for $k=1$, it reads
\begin{equation*}
\operatorname{Mon}^{\operatorname{nu}}(\mathcal{S})_{/\N}\ \mathop{\rightleftarrows} \ \mathcal{S}_\ast^{\N}.\qedhere
\end{equation*}
\end{observation}

\subsection{\texorpdfstring{$\mathbb{E}_1$}{E1}-homology of the monoidal \texorpdfstring{$\operatorname{RBS}$}{RBS}-categories}

We now restrict our attention to $\mathbb{E}_1$-spaces and consider the indecomposables functor
\begin{align*}
Q=Q^1\colon \operatorname{Mon}^{\operatorname{nu}}(\mathcal{S})_{/\N}\rightarrow \mathcal{S}^\N_\ast
\end{align*}
as recalled in the previous section. We will calculate the $\mathbb{E}_1$-homology of the monoidal $\operatorname{RBS}$-categories. First of all we exhibit their natural grading over $\N$.

\begin{notation}
Let $\mathcal{C}$ be an exact category equipped with a grading, $\mathcal{C}\rightarrow \N$. The (non-unital) monoidal $\operatorname{RBS}$-category equipped with the induced rank function
\begin{align*}
\rho\colon \mathscr{M}_{\operatorname{RBS}}^{\operatorname{nu}}(\mathcal{C})\rightarrow \N
\end{align*}
defines a non-unitary $\mathbb{E}_1$-space over $\N$. By \Cref{monoidal grothendieck}, we may equivalently view it as a (non-unitary) $\mathbb{E}_1$-algebra object in $\mathcal{S}^\N$:
\begin{align*}
n\mapsto \rho^{-1}(n)=|\operatorname{RBS}_n(\mathcal{C})|,\quad n>0,\qquad 0\mapsto \emptyset.
\end{align*}

By simply restricting the rank function, the monoidal rank filtration
\begin{align*}
\mathscr{M}_{\operatorname{RBS}}^{\operatorname{nu}}(\mathcal{C}_{\leq 0})\subset \mathscr{M}_{\operatorname{RBS}}^{\operatorname{nu}}(\mathcal{C}_{\leq 1}) \subset\cdots \subset \mathscr{M}_{\operatorname{RBS}}^{\operatorname{nu}}(\mathcal{C}_{\leq N})\subset \mathscr{M}_{\operatorname{RBS}}^{\operatorname{nu}}(\mathcal{C}_{\leq N+1}) \subset \cdots
 \subset \mathscr{M}_{\operatorname{RBS}}^{\operatorname{nu}}(\mathcal{C})
\end{align*}
gives rise to a filtration of $|\mathscr{M}_{\operatorname{RBS}}^{\operatorname{nu}}(\mathcal{C})|$ in $\operatorname{Alg}_{\mathbb{E}_1}^{\operatorname{nu}}(\mathcal{S}^\N)\simeq \operatorname{Mon}^{\operatorname{nu}}(\mathcal{S})_{/\N}$. For any $N>0$, we also consider $|\operatorname{RBS}_N(\mathcal{C})|$ as an $\N$-graded space by concentrating it in degree $N$ and likewise for its boundary. From now on, we view all these as $\N$-graded objects, but omit the morphisms to $\N$ for notational ease.
\end{notation}

We can lift the identification of the associated graded of the monoidal rank filtration to the $\N$-graded setting (as in \S \ref{associated graded}, we'll omit the $\mathcal{C}$ to ease notation).

\begin{proposition}\label{pushout in graded E1-spaces}
For any $N\geq 0$, the following diagram is a pushout in $\operatorname{Alg}_{\mathbb{E}_1}^{\operatorname{nu}}(\mathcal{S}^\N)$:
\begin{center}
\begin{tikzpicture}
\matrix (m) [matrix of math nodes,row sep=2em,column sep=2em,nodes={anchor=center}]
{
\operatorname{Free}^{\operatorname{nu}}(\vert\partial \operatorname{RBS}_N\vert) & \operatorname{Free}^{\operatorname{nu}}(\vert\operatorname{RBS}_N\vert) \\
\vert\mathscr{M}^{\operatorname{nu}}_{\leq N-1}\vert & \vert\mathscr{M}^{\operatorname{nu}}_{\leq N}\vert \\
};
\path[-stealth]
(m-1-1) edge (m-1-2) edge (m-2-1)
(m-2-1) edge (m-2-2)
(m-1-2) edge (m-2-2)
;
\end{tikzpicture}
\end{center}
In particular, we have an equivalence
\begin{align*}
|\mathscr{M}^{\operatorname{nu}}_{\leq N}|/|\mathscr{M}^{\operatorname{nu}}_{\leq N-1}| \xleftarrow{\simeq} \operatorname{Free}^{\operatorname{nu}}(|\operatorname{RBS}_N|/|\partial \operatorname{RBS}_N|)
\end{align*}
identifying the pushout in $\operatorname{Alg}_{\mathbb{E}_1}^{\operatorname{nu}}(\mathcal{S}^\N)$ on the left with the free $\mathbb{E}_1$-space on the homotopy quotient $|\operatorname{RBS}_N|/|\partial \operatorname{RBS}_N|$ concentrated in degree $N$.
\end{proposition}
\begin{proof}
As observed in \Cref{monoidal grothendieck}, $\operatorname{Alg}_{\mathbb{E}_1}^{\operatorname{nu}}(\mathcal{S}^\N)\simeq \operatorname{Mon}^{\operatorname{nu}}(\mathcal{S})_{/\N}$. Since the projection functor $ \operatorname{Mon}^{\operatorname{nu}}(\mathcal{S})_{/\N}\rightarrow  \operatorname{Mon}^{\operatorname{nu}}(\mathcal{S})$ reflects colimits, the result follows directly from \Cref{pushout of non-unital E1-spaces}.
\end{proof}

As an immediate consequence of this identification, we get an expression for the $\mathbb{E}_1$-homology of $|\mathscr{M}_{\operatorname{RBS}}^{\operatorname{nu}}(\mathcal{C})|$.

\begin{theorem}\label{E1-homology for general exact categories}
Let $\mathcal{C}$ be an exact category equipped with a grading and let $\Bbbk$ be a commutative ring. Then the $\mathbb{E}_1$-homology of the (non-unital) monoidal $\operatorname{RBS}$-category in $\mathcal{C}$ is
\begin{align*}
H_{n,d}^{\mathbb{E}_1}(|\mathscr{M}_{\operatorname{RBS}}^{\operatorname{nu}}(\mathcal{C})|;\Bbbk)\cong \tilde{H}_d(|\operatorname{RBS}_n|/|\partial \operatorname{RBS}_n|; \Bbbk) \cong
\bigoplus_{c\in C_n} \tilde{H}_d(|\operatorname{RBS}(c)|/|\partial \operatorname{RBS}(c)|;\Bbbk)
\end{align*}
where $c$ runs through a set $C_n$ of representatives of isomorphism classes of rank $n$ objects in $\mathcal{C}$ and $\operatorname{RBS}(c)$ is the $\operatorname{RBS}$-category of $c$.
\end{theorem}
\begin{proof}
Plugging the monoidal rank filtration into \Cref{general spectral sequence}, we have for every $n\in \N$, a spectral sequence
\begin{align*}
E^1_{pq}(n)=\begin{cases}
\bigoplus_{M\in \mathcal{M}_p} \tilde{H}_{p+q}(|\operatorname{RBS}_n|/|\partial \operatorname{RBS}_n|;\Bbbk) & p=n \\
0 & p\neq n \\
\end{cases}
\end{align*}
abutting to $H_{n,p+q}^{\mathbb{E}_1}(|\mathscr{M}_{\operatorname{RBS}}^{\operatorname{nu}}(\mathcal{C})|;\Bbbk)$. As the $E_1$-page is concentrated in a single column, we can read off the first isomorphism. The second isomorphism follows from the decomposition $\operatorname{RBS}_N \simeq \coprod_{c\in C_N} \operatorname{RBS}(c)$ (\Cref{RBSN vs reductive borel-serre categories}).
\end{proof}

\begin{theorem}\label{E1-homology for general ring}
Let $R$ be a commutative ring with connected spectrum and let $\Bbbk$ be a commutative ring. Then
\begin{align*}
H_{n,d}^{\mathbb{E}_1}(|\mathscr{M}_{\operatorname{RBS}}^{\operatorname{nu}}(R)|;\Bbbk)\cong \bigoplus_{M\in \mathcal{M}_n} \tilde{H}_{d-1}(|\mathcal{T}(M)|/\!\!/\operatorname{GL}(M);\Bbbk).
\end{align*}
If $R$ satisfies the standard connectivity estimate, then
\begin{align*}
H^{\mathbb{E}_1}_{n,d}(|\mathscr{M}_{\operatorname{RBS}}^{\operatorname{nu}}(R)|;\Bbbk)\cong \bigoplus_{M\in \mathcal{M}_n}H_{d-n+1}(GL(M); \operatorname{St}(M;\Bbbk));
\end{align*}
in particular, $H^{\mathbb{E}_1}_{n,d}(|\mathscr{M}_{\operatorname{RBS}}^{\operatorname{nu}}(R)|;\Bbbk)=0$ 
for all $d< n-1$.
\end{theorem}
\begin{proof}
For the first statement, combine \Cref{E1-homology for general exact categories} and \Cref{cofibre as homotopy quotient of tits complex}. The second follows by combining this with the isomorphisms
\begin{align*}
\tilde{H}_{d-1}(|\mathcal{T}(M)|/\!\!/\operatorname{GL}(M);\Bbbk)\cong H_{d-n+1}(\operatorname{GL}(M);\operatorname{St}(M;\Bbbk)).
\end{align*}
arising from the homotopy orbit spectral sequence.
\end{proof}

\begin{remark}
These results should be compared with the analogous calculation of the $\mathbb{E}_1$-homology of $\coprod \operatorname{BGL}(M)$ in \cite[\S 17.2]{GalatiusKupersRandalWilliams}:
\begin{align*}
H^{\mathbb{E}_1}_{n,d}\bigg(\coprod_{M\in \mathcal{M}_n} \operatorname{BGL}(M);\Bbbk\bigg)\cong \bigoplus_{M\in \mathcal{M}_N}H_{d-n+1}(\operatorname{GL}(M); \operatorname{St}^{\operatorname{split}}(M;\Bbbk))
\end{align*}
where $\operatorname{St}^{\operatorname{split}}(M;\Bbbk)$ is the $\mathbb{E}_1$-Steinberg module, that is, the homology of the $\mathbb{E}_1$-splitting complex. For $R$ a Dedekind domain, this agrees with the Charney's \textit{split} building (\cite{Charney80}, \cite[\S 18.2]{GalatiusKupersRandalWilliams}).

One advantage to working with the $\operatorname{RBS}$-categories rather than the general linear groups in this setting is that the Tits building seem to be a more natural object to study than the \textit{split} Tits building. Indeed, in terms of concrete practical advantages, the Tits buildings have been studied intensely for decades and there are many results on their connectivity, on the Steinberg modules and their coinvariants etc. (e.g. \cite{LeeSzczarba,AshRudolph,ChurchFarbPutman}).
\end{remark}

\subsection{Homological stability for \texorpdfstring{$\operatorname{RBS}$}{RBS}-categories}

In this section, we explore the question of homological stability of the $\operatorname{RBS}$-categories. We know now that the monoidal $\operatorname{RBS}$-category gives rise to an $\mathbb{E}_\infty$-space and we have, moreover, calculated its $\mathbb{E}_1$-homology; we review some immediate consequences of this. There's much more to be done in this direction; for now, however, we simply pick out the results that follow free of charge as it were.

Let $R$ be a commutative ring with connected spectrum. In view of \Cref{applying the group completion theorem}, we already know what the homology groups must stabilise to:
\begin{align*}
H_*(\operatorname{RBS}_\infty(R),\Z)\cong H_*(\operatorname{GL}_\infty(R),\Z).
\end{align*}
Our hope is that the $\operatorname{RBS}$-categories satisfy better slopes of homological stability than the general linear groups. We already knew this to be the case for finite fields (\cite[Theorem 1.4]{ClausenOrsnesJansen}). The results of this section support this naive hope although the difference is not as striking as for finite fields (see \Cref{compare with GL}).

\medskip

First of all, we apply the generic homological stability result of \cite{GalatiusKupersRandalWilliams}.

\begin{theorem}\label{PID homological stability}
Let $R$ be a commutative ring with connected spectrum over which finitely generated projective modules are free and which satisfies the standard connectivity estimate, e.g. $R$ a PID or a local commutative ring. Then the map
\begin{align*}
H_d(\operatorname{RBS}(R^{n-1}); \Z)\longrightarrow H_d(\operatorname{RBS}(R^n); \Z)
\end{align*}
is surjective for $2d\leq n-1$ and an isomorphism for $2d< n-1$.
\end{theorem}
\begin{proof}
The proof is a direct application of \cite[Theorem 18.3]{GalatiusKupersRandalWilliams} and it is completely analogous to the proof of Theorem 18.1 loc.cit. The notation and tools here differ slightly from the rest of the present paper --- we try to stick to that of \cite{GalatiusKupersRandalWilliams} to make the application clearer. As in \cite[\S 17.1]{GalatiusKupersRandalWilliams}, one constructs a non-unitary $\mathbb{E}_2$-algebra $\mathbf{R}$ in $\operatorname{sSet}^\N$ whose underlying functor is
\begin{align*}
n\mapsto \begin{cases}
\emptyset & n=0 \\
|\operatorname{RBS}(R^n)| & n>0 \\
\end{cases}
\end{align*}

Consider the associated $\mathbb{E}_2$-algebra in $\N$-graded free abelian groups $\mathbf{R}_\Z:=\Z\mathbf{R}$. We have
\begin{align*}
H_{n,d}^{\mathbb{E}_1}(\mathbf{R}_\Z)=H_{n,d}^{\mathbb{E}_1}(|\mathscr{M}_{\operatorname{RBS}}^{\operatorname{nu}}(R)|;\Z)=0,\quad d< n-1,
\end{align*}
and by \cite[Theorem 14.4]{GalatiusKupersRandalWilliams}, we can transfer this vanishing line from $\mathbb{E}_1$- to $\mathbb{E}_2$-homology. Moreover, we have
\begin{align*}
H_{\ast,0}(\mathbf{R}_\Z^+)\cong H_{\ast,0}(|\mathscr{M}_{\operatorname{RBS}}(R)|;\Z)\cong \Z[\sigma]
\end{align*}
where $\sigma$ is the unique class in bidegree $(1,0)$ (corresponding to the object $(R)$ in $|\mathscr{M}_{\operatorname{RBS}}(R)|$). Then Theorem 18.3 applies to show that
\begin{align*}
H_{n,d}(\overline{\mathbf{R}}_\Z/\sigma)=0,\quad \text{for }2d\leq n-1
\end{align*}
where  $\overline{\mathbf{R}}_\Z/\sigma$ is a certain left $\overline{\mathbf{R}}_\Z$-module whose underlying homotopy type is essentially that of the homotopy cofibre of the stabilisation map
\begin{align*}
|\mathscr{M}_{\operatorname{RBS}}(R)|[-1]\xrightarrow{\,(-,R)\,} |\mathscr{M}_{\operatorname{RBS}}(R)|
\end{align*}
where the first instance of $|\mathscr{M}_{\operatorname{RBS}}(R)|$ is shifted one degree down. Hence,
\begin{align*}
H_d(|\operatorname{RBS}(R^n)|,|\operatorname{RBS}(R^{n-1})|;\Z)\cong H_{n,d}(\overline{\mathbf{R}}_\Z/\sigma)=0 ,\quad \text{for }2d\leq n-1
\end{align*}
which finishes the proof.
\end{proof}

\begin{lemma}\label{H1 of RBS}
Let $A$ be an associative ring, $M$ a split Noetherian finitely generated projective module, and $\Bbbk$ a commutative ring. Then
\begin{align*}
H_1(\operatorname{RBS}(M); \Bbbk)\cong (\operatorname{GL}(M)/\operatorname{E}(M))^{\operatorname{ab}}\otimes \Bbbk
\end{align*}
where $\operatorname{E}(M)\subset \operatorname{GL}(M)$ is the subgroup generated by automorphisms of $M$ inducing the identity on the associated graded of \textit{some} splittable flag in $M$ (see \cite[Theorem 5.9]{ClausenOrsnesJansen}).
\end{lemma}
\begin{proof}
This follows directly from the calculation $\pi_1(|\operatorname{RBS}(M)|)\cong \operatorname{GL}(M)/\operatorname{E}(M)$ (\cite[Theorem 5.9]{ClausenOrsnesJansen}) and the universal coefficient theorem.
\end{proof}

In certain situations, the calculation above can be used to improve the slope of stability established in \Cref{PID homological stability} by comparing the subgroup $\operatorname{E}(R^2)$ with the special linear group $\operatorname{SL}_2(R)\subseteq \operatorname{GL}_2(R)$ as in the following proposition.

\begin{proposition}\label{improve stability to 2/3}
Let $\Bbbk$ be a commutative ring. If in the situation of \Cref{PID homological stability}, the map $R^\times =\operatorname{GL}_1(R)\rightarrow \operatorname{GL}_2(R)$ induces a surjection
\begin{align*}
R^\times\otimes \Bbbk \rightarrow \big(\operatorname{GL}_2(R)/\operatorname{E}(R^2)\big)^{\operatorname{ab}}\otimes \Bbbk,
\end{align*}
then the map
\begin{align*}
H_d(\operatorname{RBS}(R^{n-1}); \Bbbk)\longrightarrow H_d(\operatorname{RBS}(R^n); \Bbbk)
\end{align*}
is surjective for $3d\leq 2n-1$ and an isomorphism for $3d< 2n-1$.

In particular, if $\operatorname{SL}_2(R)=\operatorname{E}(R^2)[\operatorname{GL}_2(R),\operatorname{GL}_2(R)]$, then the map
\begin{align*}
H_d(\operatorname{RBS}(R^{n-1}); \Z)\longrightarrow H_d(\operatorname{RBS}(R^n); \Z)
\end{align*}
is surjective for $3d\leq 2n-1$ and an isomorphism for $3d< 2n-1$.
\end{proposition}
\begin{proof}
Indeed, by \Cref{H1 of RBS}, this exactly corresponds to the map $\operatorname{RBS}_1(R)\rightarrow \operatorname{RBS}_2(R)$ inducing a surjection on the first homology group. The claim then follows from the additional statement of \cite[Theorem 18.3]{GalatiusKupersRandalWilliams}.
\end{proof}

Recall that for a general ring $A$, $\operatorname{E}_n(A)\subseteq \operatorname{GL}_n(A)$ denotes the subgroup generated by elementary matrices and that we have an inclusion $\operatorname{E}_n(A)\subseteq \operatorname{E}(A^n)$ (which is strict in general, see \cite[\S 1.1]{ClausenOrsnesJansen}).

\begin{theorem}\label{Euclidean domain 2/3}
For $R$ a local commutative ring, the map
\begin{align*}
H_d(\operatorname{RBS}(R^{n-1}); \Z)\longrightarrow H_d(\operatorname{RBS}(R^n); \Z)
\end{align*}
is surjective for $3d\leq 2n-1$ and an isomorphism for $3d< 2n-1$. 
\end{theorem}
\begin{proof}
We have $E_2(R)=\operatorname{SL}_2(R)$; this follows from the equality $\operatorname{GL}_2(R)=\operatorname{E}_2(R)\operatorname{GL}_1(R)$ (see the proof of \cite[Lemma 1.4, Chapter III]{Weibel99}; see also \cite[Lemma 5.17]{ClausenOrsnesJansen})
\end{proof}

\begin{example}\label{2/3 euclidean}
If $R$ is a Euclidean domain, then we also have $E_2(R)=\operatorname{SL}_2(R)$ (\cite[Exercise 1.5, Chapter III]{Weibel99}) and thus homological stability of slope $\sfrac{2}{3}$ as in the proposition. We'll establish slope $1$ stability in just a moment, however, so we don't state it here.
\end{example}

\begin{example}\label{rings of integers and E2}
Let $\mathcal{O}$ be the ring of integers in an algebraic number field $K$.
\begin{enumerate}
\item If $K$ is not an imaginary quadratic extension, then $\operatorname{SL}_2(\mathcal{O})=\operatorname{E}_2(\mathcal{O})$ (\cite{Vaserstein}, see also \cite{Nica}).
\item For $K=\Q(\sqrt{-d})$ an imaginary quadratic extension, Cohn shows that
\begin{align*}
\operatorname{SL}_2(\mathcal{O})\neq\operatorname{E}_2(\mathcal{O})
\end{align*}
unless $\mathcal{O}$ is Euclidean (\cite{Cohn66}).
\item Note that for $K=\Q(\sqrt{-d})$ an imaginary quadratic extension, the ring of integers $\mathcal{O}$ is a PID if and only if $d=1,2,3,7,11,19,43,67,163$; this is the so-called Heegner Theorem (\cite{Stark}). Only the first five of these are Euclidean domains, i.e. $d=1,2,3,7,11$ (\cite{Motzkin}).
\item As observed in item (2) above, for $\mathcal{O}$ the ring of integers in $\Q(\sqrt{-19})$, the special linear group $\operatorname{SL}_2(\mathcal{O})$ is \textit{not} generated by elementary matrices, but we \textit{do} in fact have
\begin{align*}
\operatorname{SL}_2(\mathcal{O})=\operatorname{E}_2(\mathcal{O})[\operatorname{GL}_2(\mathcal{O}),\operatorname{GL}_2(\mathcal{O})]
\end{align*}
by \cite[\S 16]{Swan71}. Swan sets up a general algorithm for determining a presentation for certain special linear groups by identifying a fundamental domain for its action on the upper half plane.
\item For the remaining three values of $d=43,67,163$ for which the ring of integers $\mathcal{O}$ in $\Q(\sqrt{-d})$ is a PID, we have not been able to determine whether we can apply \Cref{improve stability to 2/3} \textit{integrally}. It's possible that Swan's algorithm can be used to show that also for these number rings we have $\operatorname{SL}_2(\mathcal{O})=\operatorname{E}_2(\mathcal{O})[\operatorname{GL}_2(\mathcal{O}),\operatorname{GL}_2(\mathcal{O})]$, but as Swan observes: ``the length of the calculation increases so rapidly with the discriminant
that machine computation seems to be the only reasonable approach.'' (\cite[\S 17]{Swan71}).
\item For $d=43, 67, 163$, we can still do something though: Indeed, we have
\begin{align*}
(\operatorname{St}_2(\mathcal{O})\otimes \Z[\sfrac{1}{6}])_{\operatorname{GL}_2(\mathcal{O})}=0
\end{align*}
by \cite[Corollary 5.6]{MillerPatztWilsonYasaki}. In other words,
\begin{align*}
H_{2,1}^{\mathbb{E}_1}(|\mathscr{M}_{\operatorname{RBS}}(\mathcal{O})|; \Z[\sfrac{1}{6}])=0.
\end{align*}
Arguing as in the proof of \cite[Theorem 1]{JansenMiller}, we conclude that the first stabilisation map $\operatorname{RBS}(\mathcal{O})\rightarrow \operatorname{RBS}(\mathcal{O}^2)$ induces a surjection on $H_1(-; \Z[\sfrac{1}{6}])$. \qedhere
\end{enumerate}
\end{example}

\begin{theorem}\label{rings of integers 2/3}
Let $\mathcal{O}$ be the ring of integers in an algebraic number field $K$. Suppose $\mathcal{O}$ has class number one and that it is \textit{not} the ring of integers in the imaginary quadratic extension $\Q(\sqrt{-d})$ for $d=43, 67, 163$. Then the map
\begin{align*}
H_d(\operatorname{RBS}(\mathcal{O}^{n-1}); \Z)\longrightarrow H_d(\operatorname{RBS}(\mathcal{O}^n); \Z)
\end{align*}
is surjective for $3d\leq 2n-1$ and an isomorphism for $3d< 2n-1$. 
\end{theorem}

The remaining rings of integers also have slope $\sfrac{2}{3}$ stability after inverting $2$ and $3$ by observation (7) in \Cref{rings of integers and E2} above.

\begin{theorem}\label{rings of integers 2/3 mod 6}
Let $\mathcal{O}$ be the ring of integers in the imaginary quadratic extension $\Q(\sqrt{-d})$ for $d=43, 67, 163$. Then the map
\begin{align*}
H_d(\operatorname{RBS}(\mathcal{O}^{n-1}); \Z[\sfrac{1}{6}])\longrightarrow H_d(\operatorname{RBS}(\mathcal{O}^n); \Z[\sfrac{1}{6}])
\end{align*}
is surjective for $3d\leq 2n-1$ and an isomorphism for $3d< 2n-1$. 
\end{theorem}

The following proposition tells us that vanishing coinvariants of the Steinberg module in all degrees yields slope $1$ homological stability.

\begin{proposition}\label{improved homological stability}
Let $\Bbbk$ be a commutative ring. If in the situation of \Cref{PID homological stability}, we have vanishing coinvariants
\begin{align*}
\operatorname{St}(R^n; \Bbbk)_{\operatorname{GL}_n(R)}=0\qquad\text{for all }n>1,
\end{align*}
then the map
\begin{align*}
H_d(\operatorname{RBS}(R^{n-1}); \Bbbk)\longrightarrow H_d(\operatorname{RBS}(R^n); \Bbbk)
\end{align*}
is surjective for $d\leq n-1$ and an isomorphism for $d< n-1$.
\end{proposition}
\begin{proof}
This is a direct application of \cite[Theorem 1]{JansenMiller} (if $2$ is a unit in $\Bbbk$, then \cite[Proposition 5.1]{KupersMillerPatzt} is enough).
\end{proof}

It is a result due to Lee-Szczarba that for a Euclidean domain $R$, we have vanishing coinvariants: $\operatorname{St}(R^n)_{\operatorname{GL}_n(R)}=0$ for all $n>1$ (\cite[Theorem 1.3]{LeeSzczarba}).

\begin{theorem}\label{Euclidean domain homological stability}
For $R$ a Euclidean domain, the map
\begin{align*}
H_d(\operatorname{RBS}(R^{n-1}); \Z)\longrightarrow H_d(\operatorname{RBS}(R^n); \Z)
\end{align*}
is surjective for $d\leq n-1$ and an isomorphism for $d< n-1$.
\end{theorem}

In recent work, Bernard--Miller--Sroka vastly improve the slope of homological stability for general linear groups over a certain class of rings by analysing the so-called complex of partial bases $B_n(R)$ (\cite{BernardMillerSroka}). In the setting of \Cref{PID homological stability}, we may exploit parts of their results to get analogous homological stability results for the $\operatorname{RBS}$-categories. We refer to \cite{BernardMillerSroka} for details on the complex of partial bases (note that they refer to the Tits complex introduced by Scalamandre, i.e. considering only free submodules, but as observed earlier, this coincides with our Tits complex under the assumption of \Cref{PID homological stability} that finitely generated projective modules are free).

\begin{theorem}\label{PID and partial bases}
Let $R$ be a commutative ring with connected spectrum over which finitely generated projective modules are free and assume that the complex of partial bases $B_n(R)$ is Cohen–Macaulay for all $n\geq 0$, e.g. $R$ local commutative. Then the map
\begin{align*}
H_d(\operatorname{RBS}(R^{n-1}); \Z[\sfrac{1}{2}])\longrightarrow H_d(\operatorname{RBS}(R^n); \Z[\sfrac{1}{2}])
\end{align*}
is surjective for $d\leq n-1$ and an isomorphism for $d< n-1$.
\end{theorem}
\begin{proof}
By \cite[Corollary 4.2]{BernardMillerSroka}, we have vanishing coinvariants $(
\operatorname{St}(R^n)\otimes \Z[\sfrac{1}{2}])_{\operatorname{GL}_n(R)}=0$ for all $n>1$ so \Cref{improved homological stability} applies.
\end{proof}

\begin{remark}\label{compare with GL}
Let's compare our results to the established results for the general linear groups.
\begin{enumerate}
\item If in the situation of \Cref{PID homological stability}, $R$ satisfies that the complex of partial bases partial bases $B_n(R)$ is Cohen–Macaulay for all $n\geq 0$, then the map
\begin{align*}
H_d(\operatorname{GL}_{n-1}(R); \Z[\sfrac{1}{2}])\longrightarrow H_d(\operatorname{GL}_n(R); \Z[\sfrac{1}{2}])
\end{align*}
is surjective for $d\leq n-1$ and an isomorphism for $d< n-1$ (\cite[Theorem 1.3]{BernardMillerSroka}).
\item The above class of rings includes Euclidean domains. For the $\operatorname{RBS}$-categories, we observe a priori better homological stability as we establish the same stability range but with \textit{integral} coefficients (\Cref{Euclidean domain homological stability}). In \Cref{small degrees} below we see that this is strictly better by explicitly comparing low degree homology groups.
\item To the best of the author's knowledge, the established stability range for the $\operatorname{RBS}$-category over the ring of integers $\mathcal{O}$ in $\Q(\sqrt{-19})$ is also better than the best known range for the general linear groups, namely slope $\sfrac{1}{2}$ (\cite[Theorem 4.11]{vanderKallen}, see also \cite[Theorem 18.1]{GalatiusKupersRandalWilliams}). Note that in the $\operatorname{RBS}$ case we used that the first homology group of $\operatorname{RBS}(\mathcal{O}^2)$ is a very explicit non-trivial quotient of $\operatorname{GL}_2(\mathcal{O})$ (\Cref{improve stability to 2/3} and \Cref{rings of integers and E2} (4)). Note that for $\operatorname{GL}_2(\mathcal{O})$ we have
\begin{align*}
H_1(\operatorname{GL}_2(\mathcal{O}))\cong\operatorname{GL}_2(\mathcal{O})/[\operatorname{GL}_2(\mathcal{O}),\operatorname{GL}_2(\mathcal{O})]\cong (\Z/2\Z)^2
\end{align*}
(\cite[Corollary 16.3]{Swan71}), but the units of $\mathcal{O}$ are $\pm 1$ . Hence we can apply the additional statement of \cite[Theorem 18.1]{GalatiusKupersRandalWilliams} with $\Z[\sfrac{1}{2}]$ coefficients but \textit{not} integrally.
\item For $\mathcal{O}$ the ring of integers in $\Q(\sqrt{-d})$ for $d=43,67, 163$, there does not seem to be a homological stability result for general linear groups analogous to that of \Cref{rings of integers 2/3 mod 6} in the literature, i.e. slope $\sfrac{2}{3}$ stability with $\Z[\sfrac{1}{6}]$-coefficients. We do not know, however, if the additional statement of \cite[Theorem 18.1]{GalatiusKupersRandalWilliams} applies with $2$ and/or $3$ inverted as for $d=19$ above (see also \cite[Example 18.2]{GalatiusKupersRandalWilliams}). Swan observes in \cite[\S 17]{Swan71} that in all the examined cases the abelianisation of $\operatorname{GL}_2(\mathcal{O})$ is an elementary abelian $2$-group, but he does not go beyond the case $d=19$.
\item We believe that the remaining results simply recover the same range of stability as the general linear groups.\qedhere
\end{enumerate}
\end{remark}

As remarked above, the difference between the $\operatorname{GL}$ and $\operatorname{RBS}$ cases over Euclidean domains lies in the $\F_2$-coefficients, and likewise for $\mathcal{O}$ the ring of integers in $\Q(\sqrt{-19})$. Let's make an explicit comparison.

\begin{remark}\label{small degrees}
If $R$ is a ring with a homomorphism to $\F_2$ (e.g. $\Z$), then it is observed in \cite[Remark 6.6]{GalatiusKupersRandalWilliams18b} that the maps
\begin{align*}
H_1(\operatorname{GL}_1(R);\F_2)\rightarrow H_1(\operatorname{GL}_2(R);\F_2) \quad \text{and}\quad H_2(\operatorname{GL}_3(R);\F_2)\rightarrow H_2(\operatorname{GL}_4(R);\F_2)
\end{align*}
are \textit{not} surjective. It follows that the stability range for $R$ established in \cite{KupersMillerPatzt} cannot be extended to $\Z$-coefficients without at the very least introducing a constant correction factor.

If $R$ is a Euclidean domain, then in the $\operatorname{RBS}$-case, the maps
\begin{align*}
H_1(\operatorname{RBS}(R);\Z)\rightarrow H_1(\operatorname{RBS}(R^2);\Z) \quad \text{and}\quad H_2(\operatorname{RBS}(R^3);\Z)\rightarrow H_2(\operatorname{RBS}(R^4);\Z)
\end{align*}
are isomorphisms (integrally): the first follows by combining \Cref{H1 of RBS} with the observation that $E_2(R)=\operatorname{SL}_2(R)$, the second follows from \ref{Euclidean domain homological stability} (in fact \Cref{2/3 euclidean} will do; this also uses that $E_2(R)=\operatorname{SL}_2(R)$).
\end{remark}

We finish this section with a final remark on why we restrict our attention to PIDs and local rings and a reference to recent work of Randal-Williams going beyond this situation.

\begin{remark}\label{improved stability}
The reason why we only consider rings over which finitely generated projective modules are free is that \cite[Theorem 18.3]{GalatiusKupersRandalWilliams} requires $H_{\ast,0}(\mathbf{R}_\Bbbk)\cong \Bbbk [\sigma]$ for an object $\sigma$ of bidegree $(1,0)$; in our setting $\sigma$ would be the object $(R)$ and the assumption needed is that each $|\operatorname{RBS}_n|$ is connected. The underlying reason for making this assumption is the following: the free $\mathbb{E}_k$-algebra on one generator has stability and adding $\mathbb{E}_k$-cells of high enough degree does not interfere with this stability. The free $\mathbb{E}_k$-algebra on more than one generator does \textit{not} have stability, so one cannot get started with the same proof strategy.\footnote{We thank Oscar Randal-Williams for explaining this (and many other things in this section) to us.}

Another way around this would be to restrict to the exact category $\mathcal{F}(R)$ of finitely generated free modules. For this, we would be replacing the $\operatorname{RBS}$-categories introduced in \cite{ClausenOrsnesJansen} with their \textit{free} analogues and the Tits complex with Scalamandre's version (\cite{Scalamandre}). See \Cref{compare flag and list version of RBS} and \Cref{Scalamandre}.
\end{remark}

\begin{remark}\label{general Dedekind}
In recent work, Randal-Williams treats the question of homological stability in situations where we have several stabilisation maps. For $R$ a general Dedekind domain, each rank $1$ projective module $L$ gives rise to a stabilisation map
\begin{align*}
\operatorname{RBS}_n(\mathcal{P}(R))\rightarrow \operatorname{RBS}_{n+1}(\mathcal{P}(R)),\quad (M_1,\ldots,M_p)\mapsto (M_1,\ldots,M_p,L),
\end{align*}
and analogously for the general linear groups; there's a priori no particular reason to favour the rank $1$ free module $R$ as the object by which to stabilise.

Randal-Williams shows that for $\mathcal{O}$ a Dedekind domain, $[L]\in \operatorname{Pic}(\mathcal{O})$ and $M$ a finitely generated projective $\mathcal{O}$-module, the map
\begin{align*}
(-,L)^*\colon H_d(|\operatorname{RBS}(M)|,\Z)\rightarrow H_d(|\operatorname{RBS}(M\oplus L)|,\Z)
\end{align*}
is an isomorphism for $2d<\operatorname{rk}(M)-2$ and an epimorphism for $2d<\operatorname{rk}(M)$ (and analogously for the stabilisation maps $\operatorname{GL}(M)\rightarrow \operatorname{GL}(M\oplus L)$). In fact, he proves a more general and somewhat more conceptual homological stability result in this setting, which says that
\begin{align*}
H_{\ast,d}(|\mathscr{M}_{\operatorname{RBS}}(\mathcal{O})|)
\end{align*}
is generated as an $H_{\ast,0}(|\mathscr{M}_{\operatorname{RBS}}(\mathcal{O})|)$-module in degrees $\leq 2d$ and presented in degrees $\leq 2d+1$. See Theorems 1.1 and 1.2 and \S 1.4 in \cite{RandalWilliams24}.

To apply the generic result (\cite[Theorem 1.5]{RandalWilliams24}), one needs to check that certain relative $\mathbb{E}_1$-homology groups vanish, which in the $\operatorname{RBS}$-case boils down to surjectivity of a certain map out of the coinvariants of the Steinberg module; this, as it turns out, has already been established by Church--Farb--Putman (\cite[\S 5.2]{ChurchFarbPutman}).
\end{remark}

\section{Questions and future directions}\label{questions}

To round off, we state a small selection of questions that should or could be addressed in future investigations of the $\operatorname{RBS}$-categories.

\medskip

\textbf{Improved homological stability.} The canonical questions apply:
\begin{enumerate}
\item Can we improve the slope of stability?
\item Can we prove homological stability for a larger class of rings?
\end{enumerate}

Some perhaps more concrete questions include:

\medskip

\textit{Can we establish slope $\sfrac{2}{3}$ stability for a larger class of rings?}

As seen in \Cref{improve stability to 2/3}, slope $\sfrac{1}{2}$ stability can be improved to slope $\sfrac{2}{3}$ if the first stabilisation map induces a surjection on first homology groups, for example if we have that
\begin{align*}
SL_2(R)= E(R^2)[GL_2(R),GL_2(R)].
\end{align*}
One should investigate whether this is the case for e.g. the ring of integers $\mathcal{O}$ in the imaginary quadratic extensions $\Q(\sqrt{-d})$, $d=43,67,163$. To this end, one could attempt to apply Swan's algorithm for determining a presentation for $\operatorname{SL}_2(\mathcal{O})$ (\cite{Swan71}).

\medskip

\textit{How do the $\operatorname{RBS}$-categories compare with general linear groups?}

It would be interesting to have a more explicit comparison with the general linear groups and a better understanding of how, why and when they differ.

\medskip

\textit{What about twisted coefficients?}

Kupers--Miller--Patzt establish homological stability with polynomial coefficients for general linear groups over the integers, the Gaussian integers, and the Eisenstein integers (\cite[Theorem D]{KupersMillerPatzt}). It's possible that the same strategy applies to the $\operatorname{RBS}$-categories for general Euclidean domains.

\medskip

\textbf{Homotopical stability.} This question is in some ways an extension of the previous question on homological stability. If we can prove homological stability with local coefficients, then we can conclude that in fact the $\operatorname{RBS}$-categories satisfy homotopical stability: if
\begin{align*}
H_d(\operatorname{RBS}(R^{n-1}); G)\longrightarrow H_d(\operatorname{RBS}(R^{n}); G)
\end{align*}
is an isomorphism for $d\ll n$ and any $\pi_1(|\operatorname{RBS}(R^{n})|)$-module $G$, then the stabilisation map
\begin{align*}
|\operatorname{RBS}(R^{n-1})|\longrightarrow |\operatorname{RBS}(R^{n})|
\end{align*}
is a homotopy equivalence for $d\ll n$ (see e.g. \cite[Lemma 1.6, Chapter IV]{Weibel99}).

Alternatively, note that for a vector space $V$ over a field, $|\operatorname{RBS}(V)|$ is a simple space (\cite[Theorem 1.3 and Theorem 1.4]{ClausenOrsnesJansen}). Does this hold more generally? That is, is there some nice class of rings $R$ for which $|\operatorname{RBS}(M)|$ is simple for $M$ a finitely generated projective $R$-module. In this case, the homotopical stability alluded to above follows from homological stability with integral coefficients. For a geometric perspective on this, see the final question of this section on the reductive Borel--Serre compactification as a moduli stack.

\medskip

\textbf{Another kind of homological stability.} In \cite{Quillen73a}, Quillen proves a homological stability result for the rank filtration of the $Q$-construction. In \cite{Damia}, Banús proves an analogous results for the $\operatorname{RBS}$-categories. We already mentioned this in \Cref{rank filtration for commutative ring}, but let's spell out what it means for the $\operatorname{RBS}$-categories: for $R$ a Dedekind domain, $M$ a finitely generated projective $R$-module, and $n\geq 0$, let
\begin{align*}
\partial_n\operatorname{RBS}(M)\subseteq\operatorname{RBS}(M)
\end{align*}
denote the full subcategory spanned by lists of objects of rank at most $n$. The map
\begin{align*}
H_i(\partial_{n-1}\operatorname{RBS}(M);\Z) \rightarrow H_i(\partial_n\operatorname{RBS}(M);\Z)
\end{align*}
induced by the inclusion is surjective for $i\leq n-2$ and an isomorphism for $i\leq n-3$ (\cite[Theorem 4.2.1]{Damia}). We can interpret this statement as saying that the $i$'th homology group of $\operatorname{RBS}(M)$ only depends on modules of rank less than $i+2$.

An aspect of this that should be investigated is how this type of homological stability relates to the more classical homological stability explored in \S \ref{homological stability}. The stabilisation map $\operatorname{RBS}(R^{n-1})\rightarrow \operatorname{RBS}(R^n)$ factors through the boundary $\partial \operatorname{RBS}(R^n)$; hence, for any $i\leq n-2$, we have a commutative diagram
\begin{center}
\begin{tikzpicture}
\matrix (m) [matrix of math nodes,row sep=1em,column sep=1em,nodes={anchor=center}]
{
\operatorname{RBS}(R^{i+2}) & & \operatorname{RBS}(R^n) \\
& \partial_{i+2} \operatorname{RBS}(R^n) & \\
};
\path[-stealth]
(m-1-1) edge (m-1-3) edge (m-2-2)
(m-2-2) edge (m-1-3)
;
\end{tikzpicture}
\end{center}

Applying $H_j(-; \Z)$ for $j\leq i$ yields a diagram where the right hand diagonal map is an isomorphism by \cite[Theorem 4.2.1]{Damia}, and if $R$ is a Euclidean domain then so is the upper horizontal map (\Cref{Euclidean domain homological stability}). Hence, the inclusion $\operatorname{RBS}(R^{i+2})\hookrightarrow \partial_{i+2}\operatorname{RBS}(R^n)$ induces an isomorphism on low degree homology groups $H_j(-;\Z)$ for $j\leq i$.

A better understanding of the nature of these maps may reveal something about the underlying dynamics of homological stability of the $\operatorname{RBS}$-categories.


\medskip

\textbf{Homotopy invariance.} It is a classical result of Quillen, also known as The Fundamental Theorem, that for a regular Noetherian ring $R$, we have a homotopy equivalence $K(R) \rightarrow K(R[t])$; in particular, the $K$-groups of $R$ and the polynomial ring $R[t]$ coincide (\cite{Quillen73}). There are unstable versions of this result due to Soulé, Knudson and Wendt where the relevant statement is that the map $\operatorname{GL}_n(k) \rightarrow  \operatorname{GL}_n(k[t])$ induces an isomorphism on homology for any infinite field $k$ (\cite{Soule,Knudson97,Wendt}). In these cases the $\operatorname{RBS}$-categories identify with the plus-constructions of the corresponding general linear groups (\cite[Theorem 1.3]{ClausenOrsnesJansen}). Hence, the homology isomorphism can be transferred onto the $\operatorname{RBS}$-categories 
\begin{align*}
\operatorname{RBS}(k^n)\rightarrow \operatorname{RBS}(k[t]^n),
\end{align*}
(see \cite[Corollary 5.3.3]{Damia}). It seems reasonable to believe that one can generalise Soulé's strategy to prove this result more directly. An interesting question would then be whether this strategy extends to a larger class of rings.

\medskip

\textbf{Hopf algebra structure and Koszul duality.} In recent work, Brown--Chan--Galatius--Payne introduce a Hopf algebra structure on Quillen's spectral sequence
\begin{align*}
E^1_{s,t} = H_t(\operatorname{GL}_s(\Z), \operatorname{St}_s(\Z)\otimes \Q) \Rightarrow H_*(\operatorname{BK}(\Z);\Q)
\end{align*}
(\cite{BrownChanGalatiusPayne}). \Cref{E1-homology for general ring} implies that this spectral sequence arises as the $\mathbb{E}_1$-homology of
\begin{align*}
|\mathscr{M}_{\operatorname{RBS}}(\Z)|\simeq K^\partial(\Z).
\end{align*}
In other words, Quillen's spectral sequence arises as the homology of the bar construction of an $\mathbb{E}_\infty$-algebra (appealing to the fact that $\mathbb{E}_1$-indecomposables may be computed via the bar construction \cite[\S 13]{GalatiusKupersRandalWilliams}). It seems plausible that the coproduct on $E^1_{s,t}$ introduced in \cite{BrownChanGalatiusPayne} is related to the algebra structure on $|\mathscr{M}_{\operatorname{RBS}}(\Z)|\simeq K^\partial(\Z)$, possibly via Koszul duality. This would be interesting to explore further.

\medskip

\textbf{The reductive Borel--Serre compactification as a moduli stack.} On a slightly different note, we turn to the geometric origins of the $\operatorname{RBS}$-categories. As explained in \cite{ClausenOrsnesJansen}, the $\operatorname{RBS}$-categories arise as a generalisation of the $1$-category $\operatorname{RBS}_\Gamma$ capturing the stratified homotopy (or exit path $\infty$-category) of the reductive Borel--Serre compactification
\begin{align*}
\overline{\Gamma\backslash X}{}^{\operatorname{RBS}}
\end{align*}
of a locally symmetric space associated to a neat arithmetic group $\Gamma\subset \operatorname{GL}_n(\Z)$ (see \cite{OrsnesJansen,ClausenOrsnesJansen}). The condition that $\Gamma$ is neat ensures that $\overline{\Gamma\backslash X}{}^{\operatorname{RBS}}$ is a stratified topological \textit{space}. The definition of the category $\operatorname{RBS}_\Gamma$ generalises verbatim to the case $\Gamma=\operatorname{GL}_n(\Z)$ without the neatness condition, but it is not quite clear how to define the reductive Borel--Serre compactification as a stratified topological \textit{stack}. The correct definition should have stratified homotopy type given by $\operatorname{RBS}_\Gamma$ for (potentially non-neat) $\Gamma$ and it is of course preferable to have a moduli description of such a stack. Such a description combined with the stability results of this paper would connect the reductive Borel--Serre compactifications with algebraic K-theory in an extremely explicit way.

In \cite{Grayson84}, Grayson provides an intrinsic construction of the Borel--Serre compactification using the notion of ``semistability'' of lattices in Euclidean space. We believe that the ingredients of that construction may be the right framework with which to define the reductive Borel--Serre compactification as a moduli stack.

\medskip

\textbf{Other $\operatorname{RBS}$-type categories?} It would be interesting to study $\operatorname{RBS}$-type categories arising in different settings. As observed in \cite[\S 6]{OrsnesJansen}, one can define similar categories whenever we have a group acting on a poset and compatible choices of subgroups of the stabilisers. Maybe some of the results and strategies of the present paper and of the previous \cite{ClausenOrsnesJansen} can be extended to other $\operatorname{RBS}$-type categories. One concrete example is the category of stable curves introduced by Charney and Lee (\cite{CharneyLee84}); in \cite{OrsnesJansen23a}, we show that this category captures the stratified homotopy type of the moduli stack of stable nodal curves, also known as the Deligne--Mumford--Knudsen compactification. One can think of this category as a ``compactification'' of the mapping class group.

Other possibly interesting examples would be flag-interpretations of the category $\operatorname{RBS}_\Gamma$ capturing the stratified homotopy type of the reductive Borel--Serre compactification for other instances of reductive algebraic groups. The categories $\operatorname{RBS}(M)$ of this paper arise as a generalisation of $\operatorname{RBS}_\Gamma$ for $\Gamma\subset\operatorname{GL}_n(\Z)$ exploiting the fact that parabolic subgroups of $\operatorname{GL}_n$ correspond to flags. One could for instance plug in the symplectic group $\operatorname{Sp}_{2n}$, consider \textit{isotropic} flags, and investigate whether the resulting categories provide useful unstable models in the setting of Hermitian K-theory.

\bibliographystyle{alpha}

\end{document}